\documentclass[10pt,a4paper]{amsart}
\usepackage[dvips]{epsfig}
\usepackage{graphics}
\usepackage{latexsym}
\usepackage{verbatim}
\usepackage{amsmath}
\usepackage{amsthm}
\usepackage{amssymb}
\usepackage{hyperref}
\newtheorem{theorem}{Theorem}[section]
\newtheorem{lemma}[theorem]{Lemma}

\newtheorem{proposition}[theorem]{Proposition}

\newtheorem{remark}[theorem]{Remark}

\newcommand{\ens}[1]{\mathbb{#1}}

\newcommand{\M}{\mathcal{M}}

\newcommand{\R}{\mathbb{R}}

\newcommand{\C}{\mathbb{C}}
\newcommand{\QQ}{\mathbb{Q}}
\newcommand{\LL}{\mathcal{L}}
\newcommand{\NN}{\mathcal{N}}
\newcommand{\Q}{\mathcal{Q}}
\newcommand{\B}{\mathcal{B}}

\def\derpar#1#2{\frac{\partial#1}{\partial#2}}
\def\var{\varepsilon}

\def\signff{\bigskip\bigskip\hspace{80mm}
\vbox{{\sc Francis Filbet\par\vspace{3mm}
Universit\'e de Lyon,\par
Universit\'e Lyon I, CNRS \par
UMR 5208, Institut Camille Jordan \par
43, Boulevard du 11 Novembre 1918\par
69622 Villeurbanne cedex, FRANCE\par\vspace{3mm}
e-mail:} filbet@math.univ-lyon1.fr }}

\def\signsj{\bigskip\bigskip\hspace{80mm}
\vbox{{\sc Shi Jin\par\vspace{3mm}
 Department of Mathematics \par
University of Wisconsin\par
Madison, WI 53706, USA 
\par\vspace{3mm} e-mail:} jin@math.wisc.edu }}

\title[Asymptotic preserving schemes for kinetic equations]{A class of asymptotic preserving schemes for kinetic equations and related problems with stiff sources}\thanks{F. Filbet is partially supported by the french ANR  project ``Jeunes Chercheurs'' \textit{M\'ethodes Num\'eriques pour les \'Equations Cin\'etiques} (MNEC). S. Jin was
partially supported by NSF grant No. DMS-0608720, NSF FRG grant
DMS-0757285, and a Van Vleck Distinguished
Research Prize from University of Wisconsin-Madison.}

\author{Francis Filbet \and Shi Jin }

\begin{document}

\maketitle

\begin{abstract}
In this paper, we propose a general framework to design asymptotic 
preserving schemes for the Boltzmann kinetic kinetic and related 
equations. Numerically solving these equations are challenging
due to the nonlinear stiff collision (source)
 terms induced by small mean free or relaxation time.
 We propose to penalize the nonlinear collision term
by a BGK-type relaxation term, which can be solved explicitly even if 
discretized implicitly in time. Moreover, the BGK-type relaxation operator
helps to drive the density distribution toward the local Maxwellian, thus 
natually imposes an asymptotic-preserving scheme in the Euler limit.
The scheme so designed  does not need any nonlinear iterative solver or 
the use of Wild Sum. It 
is uniformly stable in terms of the (possibly small)
Knudsen number, and can capture the macroscopic fluid dynamic (Euler) limit
even if the small scale determined by the  Knudsen number is not numerically
resolved. It is also consistent to the compressible Navier-Stokes equations
if the viscosity and heat conductivity are numerically resolved. 
The method is  applicable to many other related problems, such as
hyperbolic systems with stiff relaxation, and high order parabilic equations.
\end{abstract}

\tableofcontents

\section{Introduction}\label{sec1}
\setcounter{equation}{0}
The Boltzmann equation describes the time evolution of the density distribution
of a dilute gas of
particles when the only interactions taken into account are binary
elastic collisions. For space variable $x \in \Omega \in \R^{d_x}$, 
particle velocity $v \in \R^{d_v}$ ($d_v \ge 2$), the Boltzmann eqaution
reads:
 \begin{equation}
\label{eq:B}
 \derpar{f}{t} \,+\, v \cdot \nabla_x f \,=\, \frac{1}{\varepsilon}\,\Q(f)
 \end{equation}
where $f:=f(t,x,v)$ is the time-dependent particles distribution
function in the phase space.  The parameter $\varepsilon>0$ is the 
dimensionless Knudsen number which is the  ratio the mean free path over
a typical length scale such as the size of the spatial domain, thus
measures the rarefiedness of the gas.  The Boltzmann collision operator $\Q$
is a quadratic operator,
 \begin{equation} \label{eq:Q}
 \Q (f)(v) = \int_{\R^{d_v}}
 \int_{\ens{S}^{d_v-1}}  B(|v-v_\star|, \cos \theta) \,
 \left( f^\prime_\star f^\prime - f_\star f \right) \, d\sigma \, dv_\star.
 \end{equation}
We used the shorthanded notation $f = f(v)$, $f_\star = f(v_\star)$,
$f^\prime = f(v^\prime)$, $f_\star^\prime = f(v_\star^\prime)$. The
velocities of the colliding pairs $(v,v_\star)$ and
$(v^\prime,v^\prime_\star)$ are related by
 \begin{equation*}
\left\{
\begin{array}{l}
\displaystyle{ v^\prime   = v - \frac{1}{2} \big((v-v_\star) - |v-v_\star|\,\sigma\big),} 
\\
\,
\\
\displaystyle{ v^\prime_\star = v - \frac{1}{2} \big((v-v_\star) + |v-v_\star|\,\sigma\big),}
\end{array}\right.
 \end{equation*}
with $\sigma\in \ens{S}^{d_v-1}$.
The collision kernel $B$ is a non-negative function which by
physical arguments of invariance only depends on $|v-v_\star|$ and
$\cos \theta = {u} \cdot \sigma$ (where ${u} = (v-v_\star)/|v-v_\star|$ is the normalized relative velocity). In this work we assume that $B$ is locally integrable, given  by
\begin{equation*}
    B(|u|, \cos \theta) = C_\gamma \, |u|^\gamma,
    \end{equation*}
for some $\gamma \in (0,1]$ and a constant $C_\gamma >0$.

Boltzmann's collision operator has the fundamental properties of
conserving mass, momentum and energy: at the formal level
 \begin{equation}
\label{cons:Q}
 \int_{{\R}^{d_v}}\Q(f) \, \phi(v)\,dv = 0, \qquad
 \phi(v)=1,v,|v|^2,
 \end{equation}
and it satisfies the well-known Boltzmann's $H$ theorem
 \begin{equation*}
 - \frac{d}{dt} \int_{{\R}^{d_v}} f \log f \, dv = - \int_{{\R}^{d_v}} \Q(f)\log(f) \, dv \geq 0.
 \end{equation*}
The functional $- \int f \log f$ is the {\em entropy} of the
solution. Boltzmann's $H$ theorem implies that any equilibrium
distribution function, {\em i.e.}, any function which is a maximum of the
entropy, has the form of a local Maxwellian distribution
 \begin{equation*}
 \M_{\rho,u,T}(v)=\frac{\rho}{(2\pi T)^{d_v/2}}
 \exp \left( - \frac{\vert u - v \vert^2} {2T} \right), 
 \end{equation*}
where $\rho,\,u,\,T$ are the {\em density}, {\em macroscopic velocity}
and {\em temperature} of the gas, defined by
 \begin{eqnarray}
\label{rut}
&& \rho = \int_{{\R}^{d_v}}f(v)\,dv=\int_{{\R}^{d_v}}\M{\rho,u,T}(v), \quad u =
 \frac{1}{\rho}\int_{{\R}^{d_v}}v\,f(v)\,dv=
 \frac{1}{\rho}\int_{{\R}^{d_v}}v\,\M{\rho,u,T}(v)\,dv
\\
&&
 \quad T = {1\over{d_v\rho}}
 \int_{{\R}^{d_v}}\vert u - v \vert^2\,f(v)\,dv={1\over{d_v\rho}}
 \int_{{\R}^{d_v}}\vert u - v \vert^2\,\M{\rho,u,T}(v)\,dv
 \end{eqnarray}
Therefore, when the Knudsen number $\varepsilon>0$ becomes very small,
the macroscopic description, which describe the evolution of averaged  quantities such as the 
 density $\rho$, momentum $\rho\,u$ and temperature $T$ of the gas,
 by fluid dynamics equations, namely, the
compressible Euler or Navier-Stokes equations, become adequate. Mor specifically, {\it i.e.} as $\varepsilon \rightarrow 0$, the distribution function will 
converge to a local Maxwellian $\M$, and the system (\ref{eq:Q}) becomes a  
closed system for the $2 + d_v$ moments. The conserved quantities satisfy the 
classical  Euler equations of gas dynamics for a mono-atomic gas: 
\begin{equation}
\label{eq:Euler}
\left\{
\begin{array}{l}
\displaystyle \frac{\partial \rho}{\partial t} \,+\,\nabla_x\cdot \rho \,u \,=\, 0,
\\
\,
\\
\displaystyle \frac{\partial \rho \,u }{\partial t} \,+\,\nabla_x\cdot \left(\rho \,u\otimes u \,+\, p \,{\rm I}\right) \,=\, 0,
\\
\,
\\
\displaystyle \frac{\partial E }{\partial t} \,+\,\nabla_x\cdot\left( (E\,+\, p)\,u\right)\,=\, 0,
\end{array}\right.
\end{equation}
where $E$ represents the total energy
$$
E \,\,=\,\, \frac{1}{2} \,\rho \,u^2 \,+\, \frac{d_v}{2}\,\rho\, T,
$$
and $\rm {I}$ is the identity matrix. These equations  constitute a system of $2 + d_v$ equations in $3 + d_v$ unknowns. The pressure  is related to the internal energy by the constitutive relation for a polytropic gas  
$$
p \,=\, (\gamma - 1)\,\left( E \,-\,\frac{1}{2}\,\rho \,|u|^2\right),
$$ 
where the polytropic constant $\gamma = (d_v + 2)/d_v$ represents the ratio between specific heat at constant pressure and at constant volume, thus yielding
 $p=\rho\,T$.  For small but non zero values of the Knudsen number $\varepsilon$, the evolution equation  for the moments can be derived by the so-called Chapman-Enskog expansion \cite{Cer},  applied to the Boltzmann equation. This approach gives  the Navier-Stokes equations as a second order approximation with respect to $\var$ to the solution to the Boltzmann  equation:
\begin{equation}
\label{eq:CNS}
\left\{
\begin{array}{l}
\displaystyle \frac{\partial \rho_\varepsilon}{\partial t} \,+\,\nabla_x\cdot \rho_\varepsilon \,u_\varepsilon \,=\, 0,
\\
\,
\\
\displaystyle \frac{\partial \rho_\varepsilon \,u_\varepsilon }{\partial t} \,+\,\nabla_x\cdot (\rho_\varepsilon \,u_\varepsilon\otimes u_\varepsilon \,+\, p_\varepsilon \,{\rm I}) \,=\, \varepsilon\,\nabla_x\cdot[\mu_\varepsilon\,\sigma(u_\varepsilon)],
\\
\,
\\
\displaystyle \frac{\partial E_\varepsilon }{\partial t} \,+\,\nabla_x\cdot (E_\varepsilon\,+\, p_\varepsilon)\,u_\varepsilon)\,=\, \varepsilon\, \nabla_x\cdot\left(\mu_\varepsilon\sigma(u_\varepsilon)\,u \,+\, \kappa_\varepsilon\,\nabla_x T_\varepsilon\right).
\end{array}\right.
\end{equation}
In these equations $\sigma(u)$ denotes the strain-rate tensor given by 
$$
\sigma(u) \,=\, \nabla_x u \,+\,\left(\nabla_x u\right)^T  \,-\, \frac{2}{d_v} \,\nabla_x\cdot u\,{\rm I}
$$
while the viscosity $\mu_\varepsilon =\mu(T_\varepsilon)$ and the thermal conductivity $\kappa_\varepsilon =\kappa(T_\varepsilon)$ are defined according the linearized Boltzmann operator with respect to the local Maxwellian \cite{bgl:91}.

The connection between kinetic and macroscopic fluid dynamics results from two  properties of the collision operator: 
\begin{itemize}
\item[$(i)$] conservation properties and an entropy relation that imply that the equilibria are  Maxwellian distributions for the zeroth order limit; 
\item[$(ii)$] the derivative of $\Q(f)$ satisfies a formal Fredholm alternative with a kernel related  to the conservation properties of $(i)$. 
\end{itemize}

Past progress on developing robust numerical schemes for  kinetic equations 
that also work in the fluid regimes has been guided by the fluid dynamic limit,
in the framework of {\it asymptotic-preserving} (AP) scheme. As summarized
by Jin
\cite{J}, a scheme for the kinetic equation is AP if

\begin{itemize}
\item  it preserves the
discrete analogy of the Chapman-Enskog expansion, namely, it is a
suitable scheme for the kinetic equation, yet, when holding
the mesh size and time step fixed and letting
the Knudsen number go to zero, the scheme becomes a suitable scheme
for the limiting Euler equations
\item implicit collision terms can be implemented explicitly, or at least
more efficiently than using the Newton type solvers for nonlinear
algebraic systems.
\end{itemize}

To satisfy the first condition for AP, the scheme must be driven to the
local Maxwellian when $\epsilon \to 0$. This can usually be achieved 
by a backward Euler or any $L$-stable ODE solvers for the collision term
\cite{Jin2}. Such  a scheme requires an implicit collision term to gaurantee
a uniform stability in time. However, how to invert such an implicit,
yet nonlocal and nonlinear, collision operator is a delicate numerical
issue. Namely, it is hard to realize the second condition for AP schemes.

Comparing with a multiphysics domain decomposition type method
\cite{BTP, DJ, DJM, GTPS, KNS, TM}, the AP schemes avoid the
coupling of physical equations of different scales where the
coupling conditions are difficult to obtain, and interface locations
hard to determine. The AP schemes are
based on solving one equation-- the kinetic equation, and they
become  robust macroscopic (fluid) solvers {\it automatically} when the
Knudsen number goes to zero. An AP scheme implying a numerical
convergence uniformly in the Knudsen number was proved by
Golse-Jin-Levermore for linear transport equation in the diffusion
regime \cite{GJL}. This result can be extended to essentially all AP
schemes, although the specific proof is problem dependent. For
examples of AP schemes for kinetic equations in the fluid dynamic or
diffusive regimes see for examples \cite{CP, CJR, JPT1, JPT2, JP, K, K2, GT,
BLM, LM}. The AP framework has also been extended in
\cite{CDV,DDN} for the study of the quasi-neutral limit of Euler-Poisson 
and Vlasov-Poisson systems, and in \cite{DJL, HJL} for 
all-speed (Mach number) fluid equations bridging the passage from 
compressible flows to the incompressible flows.

Since the Boltzmann collision term $\Q$  needs to be treated implicitly, 
how to invert it numerically becomes a tricky issue. 
One solution was offered by Gabetta, Pareschi and Toscani \cite{GPT}. 
They first penalize $\Q$ by a linear function $\lambda f$, and then absorb
the linearly stiff part into the time variable to remove the stiffness.
The remaining implicit nonlinear collision term is approximated by
finite terms in the Wild Sum, with the infinite sum replaced by the local
Maxwellian.  This yields a uniformly stable AP scheme for the collision
term, capruring the Euler limit when $\epsilon \to 0$. Such a time-relaxed
method was also used to develop AP Monte Carlo method, see \cite{CP0, PR}.
Nevertheless, it seems that this  method is not able to capture the 
compressible Navier-Stokes asymptotic for small $\varepsilon$. 

When the collision operator $\Q$ is the BGK collision operator
\begin{equation}
  \Q_{BGK} = \M-f\,,
\end{equation}
it is well-known that even an implicit collision term can be solved
explicitly, using the property that $\Q$ preserves mass, momentum and energy.
{\it Our new idea } is this paper is to utilize this property, and penalize
the Boltzmann collision operator $\Q$ by the BGK operator:
\begin{equation}
  \Q = [\Q -\lambda (\M-f)] + \lambda[\M-f]
\label{Q-new}
\end{equation}
where $\lambda$ is the largest spectrum of the linearized collision operator
of $\Q$ around the local Maxwellian $\M$. 
Now the first term on the right hand side of (\ref{Q-new}) is either
not stiff, or less stiff compared to the second term, thus can be discretized
{\it explicitly}, so as to avoid inverting the nonlinear operator $\Q$.
The second 
term on the right hand side of (\ref{Q-new}) is stiff, thus
will be treated implicitly. Despite this,
as mentioned earlier, the implicit BGK operator can be inverted explicitly.
Therefore we arrive at a scheme which is uniformly stable in $\epsilon$,
with an implicit source term that can be solved explicitly.
 In other words, in terms of
handling the stiffness, the general Boltzmann collision operator can be
handled as easily as the much simpler BGK operator, thus we significantly simplies
an implicit Boltzmann solver!

Although a linear penalty (by removing $\M$ on the right hand size from
(\ref{Q-new}) can also remove the stiffness, it does not have the AP 
property, unless one follows the Wild Sum procedure of \cite{GPT}. 
The BGK operator that we use in (\ref{Q-new}) helps to drive $f$ into
$\M$, thus preserves the Euler limit. This will be proved asymptotically
for prepared initial data (namely data near $\M$), and demonstrated numerically
even for general initial data.
Moreover, we will prove asymptotically that, for suitably small time-step,
this method is also consistent to the Navier-Stokes equations (\ref{eq:CNS})
for $\epsilon <<1$. 

Our method is partly motivated by the work of Haack, Jin and Liu \cite{HJL}, 
where
by subtracting the leading linear part of the pressure in the compresible
Euler equations with a low Mach number, the nonlinear stiffness in the
pressure term due to the
low Mach number is removed and an AP scheme was proposed for the
compressible Euler or Navier-Stokes equations that capture the
incompressible Euler or Navier-Stokes limit when the Mach number goes to
zero.

Our method is not restricted to the Boltzmann equation. It applies to
general nonlinear hyperbolic systems with stiff nonlinear relaxation
terms \cite{CLL, JL, Jin2, CoqP}, and higher-order parabolic equations
(see section 5).  Moreover, it applies to any {\it stiff source term that
admits a stable local equilibrium.}

In the following sections, we present a class of asymptotic preserving schemes designed for kinetic equations even if the general framework can be applied to other partial differential equations. We will focus on 
the Boltzmann equation and its hydrodynamic limit. We present different numerical tests to illustrate the efficiency of the present method. We treat particularly a multi-scale problem where the Knudsen numer $\varepsilon$ depends on the space variable and takes different values ranging from $10^{-4}$ (hydrodynamic regime) to one (kinetic regime). Finally, the last part is devoted to the design of numerical schemes for nonlinear Fokker-Planck equations for which the asymptotic  preserving scheme can be used to remove the CFL constraint of a parabolic
equation.       

\section{An Asymptotic Preserving (AP) stiff ODE solver}
\setcounter{equation}{0}

Since out method
 does not depend on the discretization of the spatial derivative,
but only on the structure of the stiff source term, 
we will first present in in the  simplest framework for 
stiff ordinary differential equations. 

Let us consider a Hilbert space $H$ and  the following nonlinear
autonomous ordinary differential system
\begin{equation}
\label{ode}
\left\{
\begin{array}{l}
\displaystyle \frac{df_\varepsilon}{dt}(t) \,=\, \frac{\Q(f_\varepsilon(t))}{\varepsilon},\qquad t\,\geq\, 0,
\\
\,
\\
f_\varepsilon(0)\,=\,f_0\,\in\, H,
\end{array}\right.
\end{equation}
where the source term $\Q(f)$ satisfies the following properties:
\begin{itemize}
\item there exists a unique stationary solution $\M$ to (\ref{ode}), which 
satisfies $\Q(\M)=0$;
\item  the solution to (\ref{ode}) converges to the steady state $\M$ when 
time goes to infinity, and the spectrum of $\nabla \Q(f)\subset \C^-
={z\in \C^-, {\rm {Im}}(z) < 0}$,
$$
0 \,<\, \alpha \,\leq \,\| \nabla \Q(f)\| \,\leq\, L,\quad  \forall \,f\,\in\, H.
$$  
\end{itemize}

\begin{remark}
The second hypothesis above  is certainly not the most general, but is convenient for our purpose. The lower bound implies that  the solution converges to 
the steady state $\M$, while the upper bound is a sufficient condition for existence and uniqueness  of a global solution. 
\end{remark}

When $\varepsilon$ becomes small, the differential equation (\ref{ode}) 
becomes stiff and explicit schemes are subject to severe stability 
constraints. Of course, implicit 
schemes allow larger time step, but new difficulty arises in
seeking the numerical solution of a fully nonlinear problem at each time step. 
Here we want to combine both advantages of implicit and explicit schemes : large time step for stiff problems and low computational cost of the numerical solution at each time step.  

We denote by $f^n$ an approximation of $f(t^n)$ with $t^n=n\,\Delta t$ and 
the time step $\Delta t>0$,
Two classical procedures handle the aforementioned difficulties well. One is to 
linearize the unknown $Q(f^{n+1})$ at time step $t^{n+1}$
around $f$ at the previous time step $f^n$:
$$
  \Q(f^{n+1})\approx \Q(f^n) + \nabla \Q(f^n) (f^{n+1}-f^n)
$$
yielding a problem that only needs to solve a linear system with coefficient
matrix depending on $\nabla \Q(f^n)$ \cite{Yee}.
This approach gives a uniformly stable time discretization without nonlinear
solvers, however, it is not AP since the right hand size, as $\epsilon \to 0$,
does not project $f^{n+1}$ to the local equilibrium $f=\M$, even if 
$f^n=\M$.  The second approach, introduced by \cite{GPT}, takes
$$
\Q(f)=  [\Q(f) -\mu f] + \mu f\,.
$$
As mentioned in the introduction, it uses the Wild Sum expansion for $\Q(f)$
on the right side, which is truncated and the remaining
 infinite series is replaced
by the local Maxwellian in order to be AP for the Euler limit. 

Under our hypothesis, the asymptotic behavior of the exact solution 
$f_\varepsilon$ is known when $\varepsilon\rightarrow 0$. Therefore,  
we split the source term of (\ref{ode}) in a stiff and non- (or less)
stiff part as 
$$
\frac{\Q(f)}{\varepsilon} \,=\, \underbrace{\frac{\Q(f)\,-\,P(f)}{\varepsilon}}_{\textrm{non stiff part}} \,+\, \underbrace{\frac{P(f)}{\varepsilon}}_{\textrm{ stiff part}},  
$$ 
where $P(f)$ is a {\it well balanced}, {\it i.e.} preserving the steady state,
$P(\M)=0$, linear operator and is close to  the source term $\Q(f)$. 
For instance, performing a simple Taylor expansion, we get 
$$
\Q(f) \,=\, \Q(\M) \,+\, \nabla \Q(\M)\,(f-\M) \,+\, O(\|f-\M\|_H^2) 
$$ 
and we may choose
$$
P(f) \,:=\, \nabla \Q(\M)\,(f-\M).
$$
Since it is not always possible to compute exactly $\nabla \Q(\M)$, 
 we  may simply choose
$$
P(f) \,:=\, L\,(f-\M)\,,
$$
where $L$ is an estimate of $\nabla \Q(\M)$.

Now, we simply apply a first order implicit-explicit (IMEX) scheme for the time discretization of (\ref{ode}):  
\begin{equation}
\label{sch:01}
\frac{f^{n+1}-f^n}{\Delta t} \,= \, \frac{\Q(f^n)\,-\,P(f^n)}{\varepsilon} \,+\, \frac{P(f^{n+1})}{\varepsilon},  
\end{equation}
or
$$
f^{n+1} \,= \, \left[\varepsilon\,{\rm I} \,-\, \Delta t \,\nabla \Q(\M) \right]^{-1} \,\left[ \varepsilon \, f^n  + \Delta t\, (\Q(f^n)\,-\,P(f^n)) \,-\,  \Delta t \,\nabla \Q(\M) \,\M \right].  
$$

This method is easy to implement, since $f^{n+1}$ is linear in the
right hand side of (\ref{sch:01}). 
 For linear problems, we have the following result:
\begin{theorem}
Consider the differential system (\ref{ode}) with $\Q(f) = -\lambda\,f$, where Re$(\lambda)>0$. Set $P(f):= -\nu \,\lambda\,f$ with $\nu\geq 0$. Then, the scheme (\ref{sch:01})   is A-stable and L-stable for $\nu > 1/2$. 
\end{theorem}
\begin{proof}
For linear systems,  the scheme simple reads 
$$
f^{n+1} \,=\, \frac{\varepsilon \,  + \, (\nu-1)\,\lambda\,\Delta t}{\varepsilon\, \,+\, \nu\lambda\,\Delta t} \, f^n \,=\, \left(1 - \frac{\lambda \Delta t}{\varepsilon + \nu\,\lambda \Delta t}  \right) \,f^n. 
$$

Observe that $\nu=0$ gives the explicit Euler scheme, which is stable 
only for $\Delta t \leq \varepsilon/\lambda$, whereas for $0\leq \nu\leq 1$, 
it yields  the so-called $\theta$-scheme, which is  $A$-stable for 
$\nu>1/2$. For $\nu=1$ it corresponds to the $A$-stable  implicit Euler 
scheme. Moreover, for $\nu >1$, the scheme is  $A$-stable, that is 
$$
\|f^{n+1}\|_H \,\leq\, \left| 1 - \frac{\lambda \Delta t}{\varepsilon + \nu\,\lambda \Delta t}  \right| \,\|f^{n}\|_H \sim \left(1-\frac1{\nu}\right)
 \|f^{n}\|_H\,  \quad {\hbox {for}} \quad \epsilon \sim 0\,,   
$$
where $ |1-\frac1{\nu}|<1$ for $\nu >1/2$. This is also the condition for
the L-stability \cite{Gear}.
\end{proof}

To improve the numerical accuracy, second order schemes are sometimes more 
desirable. Thus, we propose the following second order IMEX extension. Assume that an approximate solution $f^n$ is known at time $t^n$, we compute a first approximation at time $t^{n+1/2}=t^n+\Delta t/2$ using a first order IMEX scheme  and next apply the trapezoidal rule and the mid-point formula. The scheme reads
\begin{equation}
\label{sch:02}
\left\{
\begin{array}{l}
\displaystyle 2\,\frac{f^{\star}-f^n}{\Delta t} \,= \, \frac{\Q(f^n)\,-\,P(f^n)}{\varepsilon} \,+\, \frac{P(f^\star)}{\varepsilon}.
\\
\,
\\
\displaystyle\frac{f^{n+1}-f^n}{\Delta t} \,= \, \frac{\Q(f^\star)\,-\,P(f^\star)}{\varepsilon} \,+\, \frac{P(f^n) +P(f^{n+1})}{2\,\varepsilon}.
\end{array}\right.
\end{equation}

Note that both (\ref{sch:01}) and (\ref{sch:02}) are AP for prepared initial
data. Let us use (\ref{sch:02}) as the example. Assume $f^n=\M+O(\varepsilon)$.
Then $\Q(f^n)=O(\varepsilon), P(f^n)=O(\varepsilon)$, thus the first step
in (\ref{sch:02}) gives $P(f^*)=O(\varepsilon)$. This further implies
that $\Q(f^*)=O(\varepsilon)$. Applying all these in the second step of
(\ref{sch:02}) gives $P(f^{n+1})=O(\varepsilon)$, thus $f^{n+1}=\M+
O(\varepsilon)$, which is the desired AP
property.

To illustrate the efficiency of  (\ref{sch:01}) and (\ref{sch:02}) in various situations, we consider a simple linear problem with different scales for which only some components rapidly converge to a steady state whereas the remaining part oscillates. We solve
\begin{equation}
\Q(f) \,=\, A \,f, 
\label{Q-eq}
\end{equation}
where
\begin{eqnarray}
\label{A-def}
A = \left( 
\begin{array}{lll}
-1000 & \quad 1& 0
\\
 -1 & -1000 & 0
\\
\quad 0 & \quad 0 & i
\end{array}\right)
\end{eqnarray}
for which the eigenvalues are Sp$(A)\,=\,\{ -1000+\,i,\,-1000-\,i,\, i\}$. 
The first block represents the fast scales whereas the last one is the
 oscillating part.  Indeed, the first components go to zero exponentially fast whereas the third one oscillates with respect to time with a period of $2\pi$. We want to solve accurately the oscillating part with a large time step without resolving small scales. Then, we apply  the first order (\ref{sch:01}) and second order (\ref{sch:02})  schemes by choosing
$$
P(f) = \nu\,A \,f,  
$$
with $\nu\geq 0$. Here we take a large time step $\Delta t=0.3$ and $\nu=2$, which means that $P(f)$ has the same structure of $\Q(f)$ but the eigenvalues are over estimated. Thus, fast scales are under-resolved whereas this times step is a good discretization of the third oscillating component. Therefore, an efficient AP  scheme would give an accurate behavior of the slow oscillating 
scale with large time step with respect to the fast scale. It clearly 
appears in Figure~\ref{fig:01} that the time step is too large to give 
accurate results for the first order scheme (\ref{sch:01}): the solution is 
stable but the oscillation of the third component is damped for this
time step which is too large. This approximation is compared with the one 
obtained with a first order explicit Euler using a times step ten times 
smaller. We also compare the numerical solution of the second order scheme
 (\ref{sch:02}) with the one obtained using a second order explicit Runge-Kutta 
scheme corresponding to $\nu=0$ with a time step three hundred times smaller. In Figure~\ref{fig:01}, we observe the stability and good accuracy of the second order 
scheme (\ref{sch:02}). Note that for the same time step, the explicit Runge-Kutta scheme blows-up! 

\begin{figure}[htbp]
\begin{tabular}{cc}
\includegraphics[width=7.75cm]{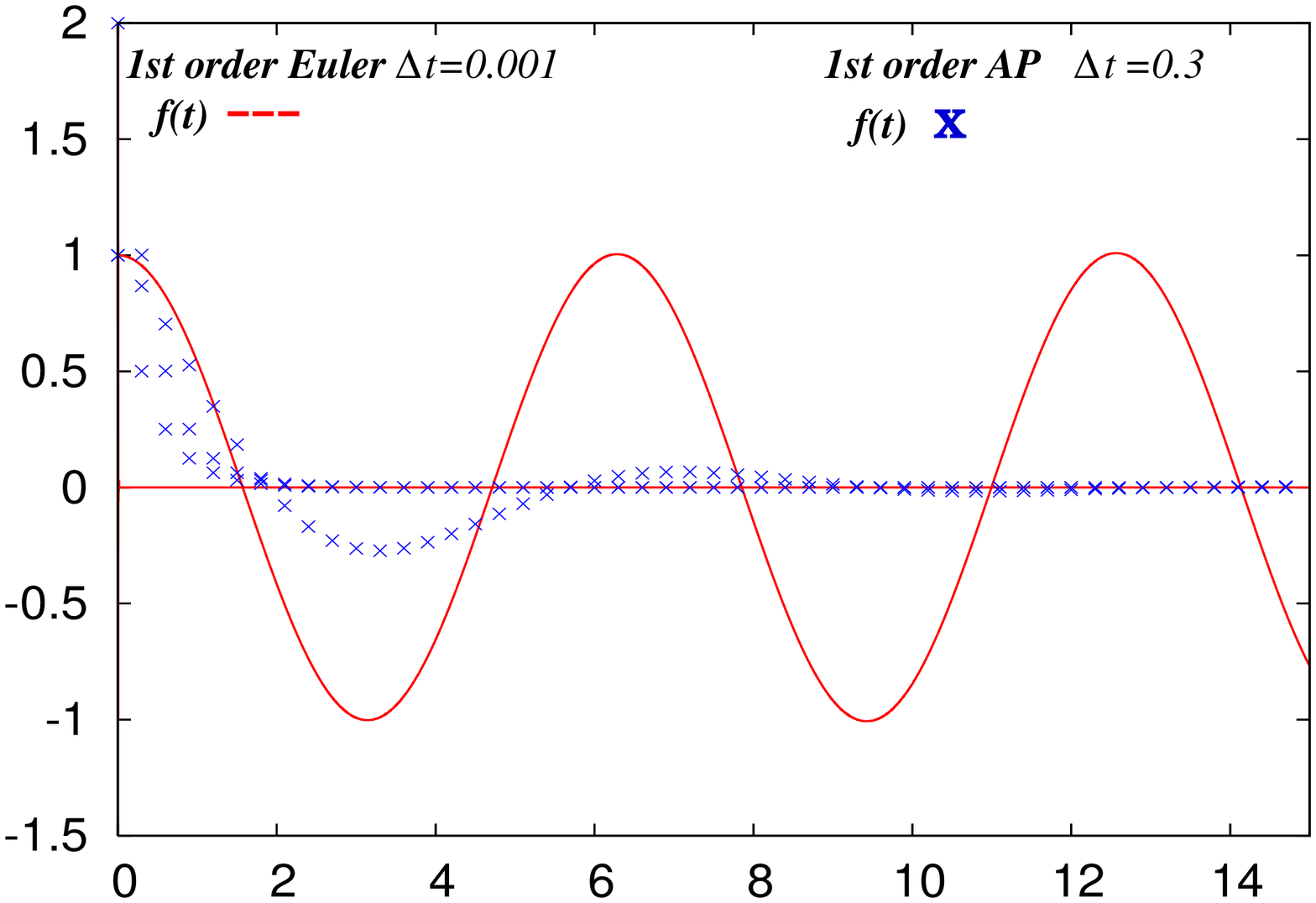}    
&
\includegraphics[width=7.75cm]{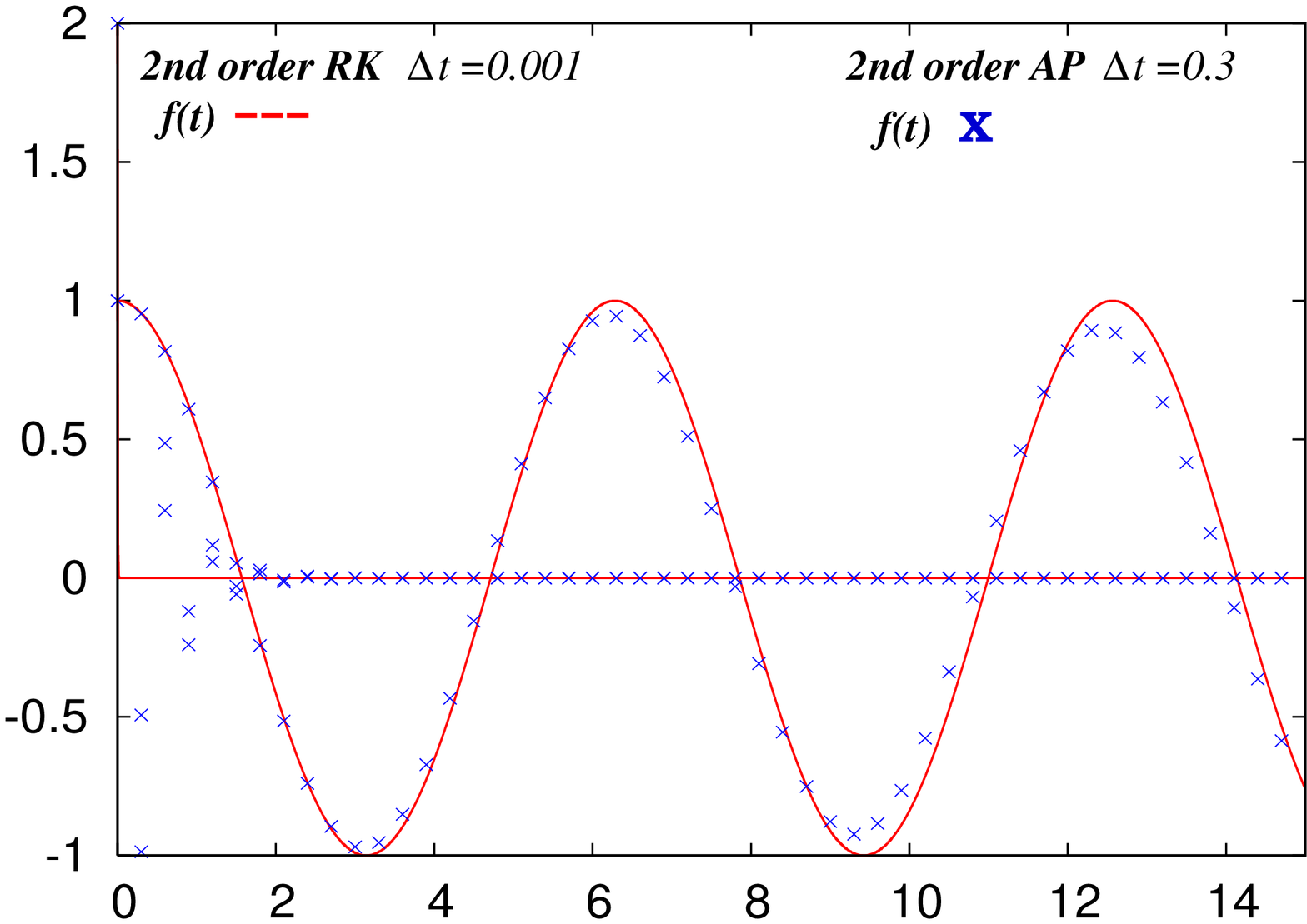}   
\\
(1)&(2)
\end{tabular}
\caption{Comparison of (1) first and (2) second order Asymptotic Preserving and explicit Runge-Kutta schemes for the differential system 
(\ref{Q-eq})-(\ref{A-def}).}
\label{fig:01}
\end{figure}
     
In the following sections we apply this approach
 to the Boltzmann equation and verify its accuracy and  efficiency on several classical problems dealing with fluid, kinetic and multi-scale regimes. 
     
\section{Application to the Boltzmann equation}
\label{sec3}
\setcounter{equation}{0}



We now extend the stiff ODE solver of the previous section to the
Boltzmann eqaution (\ref{eq:B}). To this aim, we rewrite the Boltzmann equation (\ref{eq:B}) in the following form
\begin{equation} 
\label{perturb}
\left\{
\begin{array}{l}
  \displaystyle{\frac{\partial f}{\partial t}  \,+\, v\,\nabla_x f \,=\,  \frac{\Q(f) \,-\, P(f)}{\varepsilon}\,+\,\frac{P(f)}{\varepsilon},  
   \quad x \in \Omega\subset\R^{d_x},\, v\in \R^{d_v},} 
  \\
  \,
  \\
  f(0,x,v)  \,=\, f_{0}(x,v), \quad x\in\Omega,\,v\in\R^{d_v}, 
\end{array}\right.
\end{equation}
where the operator $P$ is a ``well balanced relaxation approximation'' of $\Q(f)$, which means that it satisfies the following (balance law)
\begin{equation*}
 \int_{{\R}^d}P(f) \, \phi(v)\,dv = 0, \qquad
 \phi(v)=1,v,|v|^2,
 \end{equation*}
and preserves  the steady state {\it i.e.} $P(\M_{\rho,u,T})=0$ where $\M_{\rho,u,T}$ is the  Maxwellian distribution associated to $\rho$, $u$ and $T$ given by (\ref{rut}). Moreover, it is a relaxation operator in velocity
\begin{equation}
P(f) = \beta \, \left[\M_{\rho,u,T}(v) - f(v) \right]. 
\label{BGK}
\end{equation}
  
For instance, $P(f)$ can be computed from  an expansion of the Boltzmann 
operator with respect to $\M{\rho,u,T}$:
\begin{eqnarray*}
\Q(f)  &\simeq& \Q(\M_{\rho,u,T}) \,+\, \nabla \Q(\M_{\rho,u,T})\,\left[\M_{\rho,u,T}-f\right]. 
\end{eqnarray*} 
Thus, we choose $\beta>0$ as an upper bound of the operator $\nabla \Q(\M_{\rho,u,T})$.
Then $P(f)$ given by (\ref{BGK}) is just the BGK collisional operator
\cite{BGK}.
 
Since the convection term in (\ref{perturb}) is not stiff, we will
treat it explicitly. For source terms on the right hand side of 
(\ref{perturb}) will be handled using the ODE solver in the previous
section. For example, if the first order scheme (\ref{sch:01}) is used,
then we have 
\begin{equation} 
\label{sch:perturb}
\left\{
\begin{array}{l}
  \displaystyle{\frac{f^{n+1}-f^n}{\Delta t }  + v\cdot\nabla_x f^n \,=\, \frac{\Q(f^n) \,-\, P(f^n)}{\varepsilon}\,+\,\frac{P(f^{n+1})}{\varepsilon},} 
  \\
  \,
  \\
  f^0(x,v)  = f_{0}(x,v)\,.
\end{array}\right.
\end{equation}
 Using the relaxation structure of $P(f)$, it can be written as
\begin{eqnarray*}
\label{AP-1}
f^{n+1} &=& \frac{\varepsilon}{\varepsilon+\beta\Delta t} \left[f^n - \Delta t  \,v\,\nabla_x f^n\right] \,+\, \Delta t\,\frac{\Q(f^n) \,-\, P(f^n)}{\varepsilon+\beta \Delta t} 
\\
\,
\\
&&+ \frac{\beta \Delta t}{\varepsilon+\beta \Delta t}\,\M^{n+1},
 \end{eqnarray*}
where $\M^{n+1}$ is the Maxwellian distribution computed from $f^{n+1}$.

Although (\ref{AP-1}) appears  nonlinearly implicit, it can be  
computed explicitly. Specifically, upon multiplying (\ref{AP-1}) by $\phi(v)$ 
defined
in (\ref{cons:Q}), and use the conservation property of
$\Q$ and $P$ and the definition of $\M$ in  (\ref{rut}), 
one gets
$$
U^{n+1} = \frac{\varepsilon}{\varepsilon+\beta \Delta t} \int \phi
(f^n-\Delta t\,  v \cdot \nabla_x f^n) \, dv 
+ \frac{\beta\Delta t}{\varepsilon+\beta \Delta t} U^{n+1}\,,
$$
or simply
$$
U^{n+1} =  \int \phi
(f^n-\Delta t v \cdot \nabla_x f^n)  \, dv \,.
$$
Thus $U^{n+1}$ can be obtained explicitly, which defines $\M^{n+1}$. Now
$f^{n+1}$ can be obtained from (\ref{AP-1} explicitly. 
In summary, although (\ref{sch:perturb}) is nonlinearly
implicit, it can be solved {\it explicitly}, thus satisfies the second condition of an AP scheme.

  We define the macroscopic quantity $U$ by  $U:=(\rho,\rho\,u,T)$ computed from $f$. Clearly, the scheme (\ref{sch:perturb}) satisfies the following properties
    
\begin{proposition}
Consider the numerical solution given by (\ref{sch:perturb}). Then,
\begin{itemize}
\item[$(i)$]  If $\varepsilon \to 0$  and $f^n=\M^{n}+O(\varepsilon)$,
 then the scheme  (\ref{sch:perturb}) is asymptotic preserving, that is $f^{n+1}=\M^{n+1}
+O(\varepsilon)$, thus the scheme is AP to the Euler limit in the sence
that, when $\varepsilon\to 0$, the (moments of the) scheme becomes a
consistent discretization of the Euler system (\ref{eq:Euler}). 
\item[$(ii)$] For $\varepsilon\ll 1$ and there exists a constant $C>0$ such that
\begin{equation}
\label{hyp:00}
\left\|\frac{f^{n+1}-f^n}{\Delta t}\right\| \,+\,\left\|\frac{U^{n+1}-U^n}{\Delta t}\right\| \leq C, 
\end{equation}
then the  scheme (\ref{sch:perturb})  asymptotically becomes
 a first order in time approximation of the compressible Navier-Stokes (\ref{eq:CNS}).
\end{itemize}
\end{proposition}
\begin{proof}
We easily first check that for $\varepsilon \to 0$ and $f^n=\M^n$, we get $f^{n+1}=\M^{n+1}$. Therefore, we multiply (\ref{sch:perturb}) by $(1,v, |v|^2/2)$ and integrate with respect to $v$, which yields that $U^n$ is given by a time explicit scheme of the Euler system (\ref{eq:Euler}).    

Now let us prove (ii). We apply the classical Chapman-Enskog expansion: 
\begin{equation}
\label{decomp}
f^n \,=\, \M ^n \,+\, \varepsilon \, g^n   
\end{equation}  
and integrate (\ref{sch:perturb}) with respect to $v\in\R^{d_v}$. By
 using  the conservation properties of the Boltzmann operator (\ref{cons:Q}) and of the well-balanced approximation $P(f)$, 
\begin{equation}
\label{moment}
\frac{U^{n+1}-U^{n}}{\Delta t} \,+\,\nabla_v\cdot \int_{\R^{d_v}}\left(\begin{array}{l} 1\\ v \\\displaystyle  \frac{|v|^2}{2} \end{array}\right)\, v \, (\M^n \,+\,\varepsilon g^n)  dv \,=\, 0. 
\end{equation} 
For $\varepsilon g=0$, this is the compressible Euler equations 
(\ref{eq:Euler}).
Thus, a consistent approximation of the compressible Navier-Stokes is  directly related to a consistent approximation of  $g^n$.  Inserting decomposition (\ref{decomp}) into the scheme (\ref{sch:perturb}) gives
\begin{eqnarray*}
 \displaystyle \frac{\M^{n+1}-\M^n}{\Delta t }  &+& v\,\nabla_x \M^n \,+\, \varepsilon\left(\frac{g^{n+1}-g^n}{\Delta t }  + v\,\nabla_x g^n \right) 
\\
&=& \frac{\Q(\M^n+\varepsilon g^n)}{\varepsilon}\,-\, \left[\beta(\rho^n)g^n-\beta(\rho^{n+1})g^{n+1}\right],
\end{eqnarray*}
Since $\Q$ is a bilinear and $\Q(\M)=0$, one has
$$
\Q(\M +\varepsilon \,g ) \,=\, \Q(\M) \,+ \,\varepsilon \,\LL_{\M}(g)  + \varepsilon^2 \;\Q(g),
$$
where $\LL_{\M}$ is the linearized collision operator with respect to $\M$. Thus, we get
\begin{eqnarray}
\nonumber
&& \frac{\M^{n+1}-\M^n}{\Delta t }  \,+\, \left[\beta(\rho^n)g^n-\beta(\rho^{n+1})g^{n+1}\right]
\\
\nonumber
&+&\varepsilon\left[\frac{g^{n+1}-g^n}{\Delta t }  + v\,\nabla_x g^n - \Q(g^n) \right] 
\\
&&=\LL_{\M}(g^n) \,-\, v\,\nabla_x \M^n\,,
\label{tmp:0}
\end{eqnarray}
It is well known that $\LL_{\M}$ is a non-positive self-adjoint operator on $L^2_{\M}$ defined by the set 
$$
L^2_\M:=\{\varphi :  \quad \varphi\,\M^{-1/2}\in L^2(\R^{d_v})\}
$$ 
and that its kernel is $\NN(\LL_{\M})={\rm Span}\{\M, v\,\M, |v|^2\M\}$. Let $\Pi_{\M}$ be the orthogonal projection in $L^2_\M$ onto $\NN(\LL_\M)$. After easy computations in the orthogonal basis, one finds that 
$$
\Pi_{\M}(\psi) = \frac{\M}{\rho} \left[m_0 + \frac{v-u}{T} m_1  +\left(  \frac{|v-u|^2}{2T} - \frac{d_v}{2}  \right) m_2   \right]
$$
where
$$
m_0 = \int_{\R^{d_v}} \psi \,dv, \quad m_1 = \int_{\R^{d_v}} (v-u)\psi\, dv,  \quad m_2 = \int_{\R^{d_v}} \left(\frac{|v-u|^2}{2T} - \frac{d_v}{2}\right)\,\psi \,dv.
$$
It is easy to verify that $\Pi_{\M^n}(\M^{n})=\M^n$ and
$$
\Pi_{\M^n}(g^n) \,=\,\Pi_{\M^n}(g^{n+1}) \,=\,  \Pi_{\M^n}(Q(g^{n})) \,=\,\Pi_{\M^n}(\LL_{\M^n}(g^n))\,=\, 0.
$$
Then applying the orthogonal projection ${\rm I} - \Pi_{\M^n}$  to (\ref{tmp:0}), it yields 
\begin{eqnarray*}
&& \displaystyle \left({\rm I}-\Pi_{\M^n}\right)\left(\frac{\M^{n+1}-\M^n}{\Delta t }\right)
+ \left(\beta(\rho^n)g^n-\beta(\rho^{n+1})g^{n+1}\right)
\\
&&+\,\varepsilon\left[ \frac{g^{n+1}-g^n}{\Delta t } + \left({\rm I}-\Pi_{\M^n}\right)(v\,\nabla_x g^n) -  \Q(g^n)\right] 
\\
&=& \LL_{\M}(g^n) -\left({\rm I}-\Pi_{\M^n}\right)(v\,\nabla_x \M^n).
\end{eqnarray*}
Finally, it remains to estimate 
$$
\left({\rm I}-\Pi_{\M^n}\right)\left(\frac{\M^{n+1}-\M^n}{\Delta t }\right).
$$ 
Using a Taylor expansion we find that
$$
\M^{n+1} = \M^{n}\left[1+ \frac{\rho^{n+1}-\rho^n}{\rho^n} + \frac{v-u^n}{T^n}\left(u^{n+1}-u^{n}\right) + \left(\frac{|v-u^n|^2}{2 T^n}- \frac{d}{2}\right) \frac{T^{n+1}-T^{n}}{T^n}\right] \,+\, O(\Delta t^2)
$$
and by definition of $\Pi_\M$
\begin{eqnarray*}
&&\Pi_{\M^n}(\M^{n+1}) \,\,= 
\\
&&\M^n\,\left(1 \,+\, \frac{\rho^{n+1}-\rho^n}{\rho^n} \,+\, \frac{v-u^n}{T^n}\left(u^{n+1}-u^{n}\right) \,+\, \left(\frac{|v-u^n|^2}{2 T^n}- \frac{d}{2}\right) \frac{T^{n+1}-T^{n}}{T^n} \right] 
\\
&+&  \M^n  \,\left(\frac{|v-u^n|^2}{2 T^n}\,-\, \frac{d}{2}\right) \left[ \frac{T^{n+1}-T^n}{\rho^n\,T^n}(\rho^{n+1}-\rho^n) \,+\, \frac{\rho^{n+1}}{d\rho^n\,T^n}(u^{n+1}-u^n)^2\right]
\\
&+&  \M^n\, \frac{v-u^n}{T^n}\, \frac{\rho^{n+1}-\rho^n}{\rho^n}\, \left(u^{n+1}-u^{n}\right) + O(\Delta t^2).
\end{eqnarray*}
Thus, under the assumption (\ref{hyp:00}), we have
$$
\left({\rm I}-\Pi_{\M^n}\right)\left(\frac{\M^{n+1}-\M^n}{\Delta t }\right) = O(\Delta t)
$$ 
and the residual distribution function is given by 
$$
g^n \,=\,  \LL_{\M^n}^{-1}\big(\left({\rm I}-\Pi_{\M^n}\right)\left(v\cdot\nabla_x \M^n\right)\big) + O(\varepsilon)\,+\,  O(\Delta t).
$$
Now, substituting this latter expression in (\ref{moment}), we get 
\begin{eqnarray*}
\frac{U^{n+1}-U^{n}}{\Delta t} \,+\,\nabla_x\cdot F(U) &=& -\varepsilon\, \nabla_x\cdot \int_{\R^{d_v}}\left(\begin{array}{l} v\\ v\otimes v \\\displaystyle  v\frac{|v|^2}{2} \end{array}\right) \, \LL_{\M^n}^{-1}\big(\left({\rm Id}-\Pi_{\M^n}\right)\left(v\cdot\nabla_x \M^n\right)\big)  dv 
\\
&&+ O(\varepsilon\Delta t+\varepsilon^2),
\end{eqnarray*} 
where 
$$
F(U) = \left(\begin{array}{l}
\rho\,u
\\
\rho\,u\otimes u + p\,{\rm I}
\\
(E\,+\, p)\,u 
\end{array}\right).
$$
To complete the proof, it remains to compute the term in $O(\var)$. An easy computation first gives
$$
\left({\rm I}-\Pi_{\M^n}\right)\left(v\cdot\nabla_x \M^n\right) \,=\,\left[B \, \left( \nabla u + (\nabla u)^T - \frac{d}{2}\nabla\cdot u\,{\rm I}\right)  \,+\, A\,\frac{\nabla T}{\sqrt{T}}\right]\,M(v),   
$$
with 
$$
A = \left(\frac{|v-u|^2}{2T} - \frac{d+2}{2} \right) \frac{v-u}{\sqrt{T}}, \quad B = \frac{1}{2}\left(\frac{(v-u)\otimes(v-u)}{2T} - \frac{|v-u|^2}{dT}{\rm I} \right).  
$$
Therefore, it yields
\begin{eqnarray*}
\LL_{\M^n}^{-1}\big(\left({\rm I}-\Pi_{\M^n}\right)\left(v\cdot\nabla_x \M^n\right)\big)  &=& \LL_{\M^n}^{-1}(B\,M)  \left( \nabla u + (\nabla u)^T - \frac{d}{2}\nabla\cdot u\,{\rm I}\right)
\\
&&+ \LL_{\M^n}^{-1}(A\,M)\,\frac{\nabla T}{\sqrt{T}}. 
\end{eqnarray*}
Substituting this expression in (\ref{moment}), we get a consistent time discretization scheme to the compressible Navier-Stokes system where the term of order of $\varepsilon$ is given by
\begin{equation*}
\varepsilon\, \nabla_x\cdot\left(
\begin{array}{l}
0
\\
\mu_\varepsilon\,\sigma(u_\varepsilon)
\\
\mu_\varepsilon\sigma(u_\varepsilon)\,u \,+\, \kappa_\varepsilon\,\nabla_x T_\varepsilon
\end{array}\right)
\end{equation*}
with
$$
\sigma(u) \,=\, \nabla_x u \,+\,\left(\nabla_x u\right)^T  \,-\, \frac{2}{d_v} \,\nabla_x\cdot u\,{\rm I}
$$
while the viscosity $\mu_\varepsilon =\mu(T_\varepsilon)$ and the thermal conductivity $\kappa_\varepsilon =\kappa(T_\varepsilon)$ are defined according the linearized Boltzmann operator with respect to the local Maxwellian \cite{bgl:91}.
\end{proof}

\begin{remark}
To capture the Navier-Stokes approximation that has $O(\varepsilon)$ 
viscosity and heat conductivity, one needs the mesh size and $c \Delta t$
to be $o(\varepsilon)$ ($c$ is a characteristic speed). Thus conclusion (ii) in
the above proposition shows that the scheme is {\it consistent} to
the Navier-Stokes equations provided that the viscous terms are resolved,
while to capturing the Euler limit one can use mesh size and $c\Delta t$
much larger than $\varepsilon$, in the usual sense of asymptotic-preserving.
\end{remark}

\section{Numerical tests}

In this section we perform several numerical simulations for the
Boltzmann equation in different asymptotic regimes
in order to check the performance (in stability and
accuracy) of our methods. 
We have implemented the first order  (\ref{sch:01})  and second order 
(\ref{sch:02}) scheme for the approximation of the Boltzmann equation. 
Here,  the Boltzmann collision operator is discretized by a 
deterministic  method \cite{FiPa:02,FiRu:FBE:03,FiPa:03,FMP}, which gives a 
spectrally accurate approximation. A classical second 
order finite volume scheme with slope limiters is applied for the transport 
operator.
  
\subsection{Approximation of smooth solutions.}
This test is used to evaluate the order of accuracy of our new methods.
More precisely, we want to show that our methods (\ref{sch:01}) and
(\ref{sch:02}) are uniformly accurate with respect to the parameter 
$\varepsilon>0$. We consider the Boltzmann equation (\ref{eq:B}) 
in $1\,d_x\times 2\,d_v$. We take a smooth initial data
$$
f_0(x,v)\,\,=\,\, \frac{\rho_0(x)}{2\pi\,T_0(x)}\, \exp\left(-\frac{|v|^2}{2\,T_0(x)}\right), \quad (x,v)\in [-L,L]\times \R^2,
$$
with $\rho_0(x)\,=\, (11 \,-\,9 \tanh(x))/10$, $T_0(x)\,=\, (3 - \tanh(x))/4$, $L=1$ and assume specular reflection boundary conditions in $x$. Numerical solutions are computed from different phase space meshes :  the number of point in space is $n_x=50, 100$, $200$,...,$1600$ and the number of  points in velocity is $n_v^2$ with $n_v=8$,...,$64$ (for which the spectral accuracy is achieved), the time step is computed such that the CFL condition for the transport is satisfied $\Delta t\leq \Delta x/v_{\rm max}$, where $\Delta x$ is the space step and $v_{\rm max}=7$ is the truncation of the velocity domain. Then different values of $\varepsilon$ are considered starting from the fully kinetic regime $\varepsilon=1$, up to the fluid limit $\varepsilon=10^{-5}$ corresponding to the solution of the Euler system (\ref{eq:Euler}). The final time is $T_{\rm max}=1$ such that the solution is smooth for the different regimes.

 An estimation of the relative error in $L^p$
norm is given by
$$
\varepsilon_{2\,h} = \max_{t\in(0,T)}\left(\frac{\| f_h(t) - f_{2\,h}(t) \|_p}{\|f_{0}\|_p}\right), \qquad 1\leq p \leq +\infty,
$$
where $f_h$ represents the approximation computed from a grid of order $h$.
The numerical scheme is said to be $k$-th order if $\varepsilon_{2\,h}  \leq C \,h^k$, for all $0 < h \ll 1$.
\begin{figure}[htbp]
\begin{tabular}{cc}
\includegraphics[width=7.75cm]{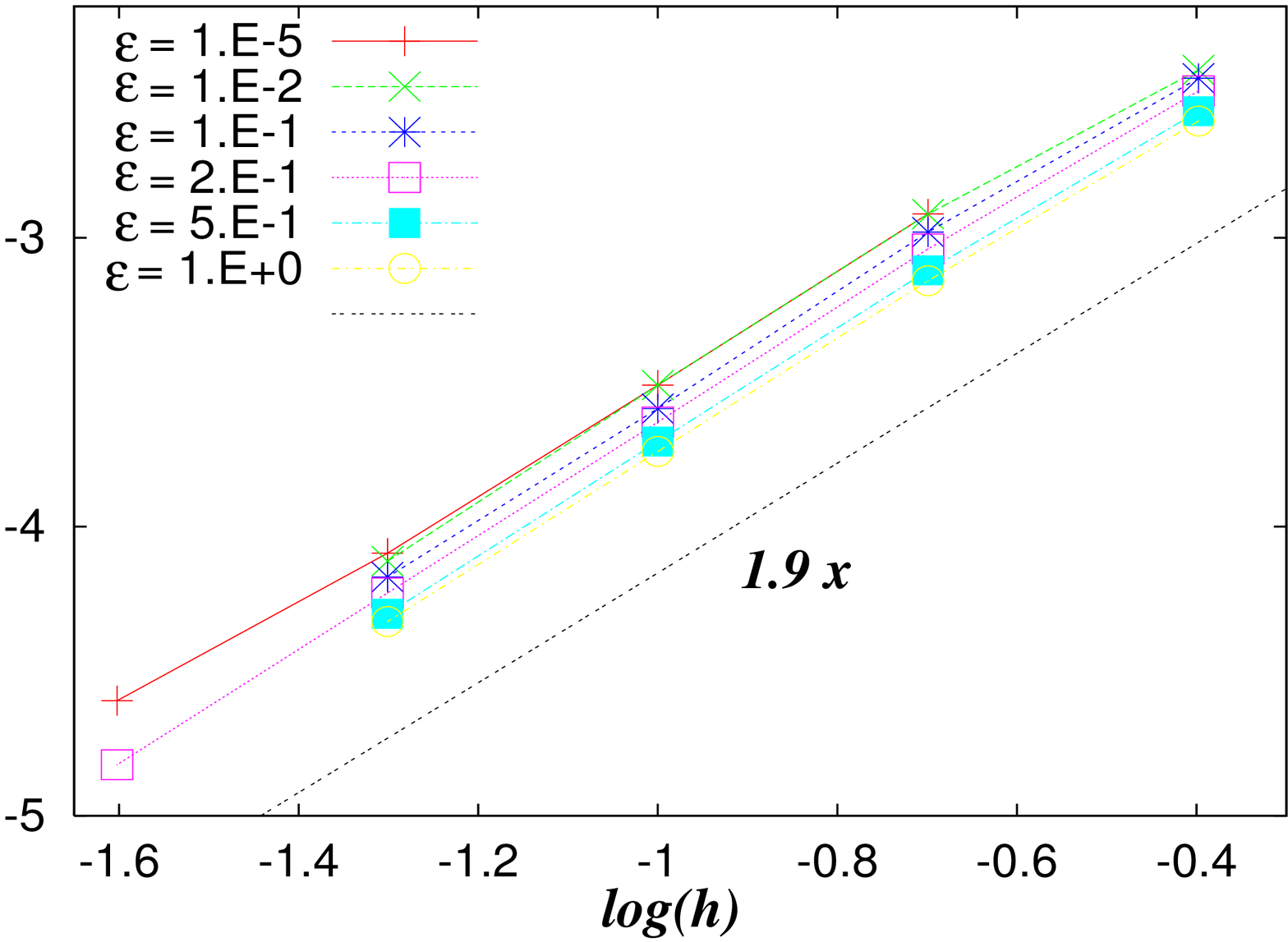}    
&
\includegraphics[width=7.75cm]{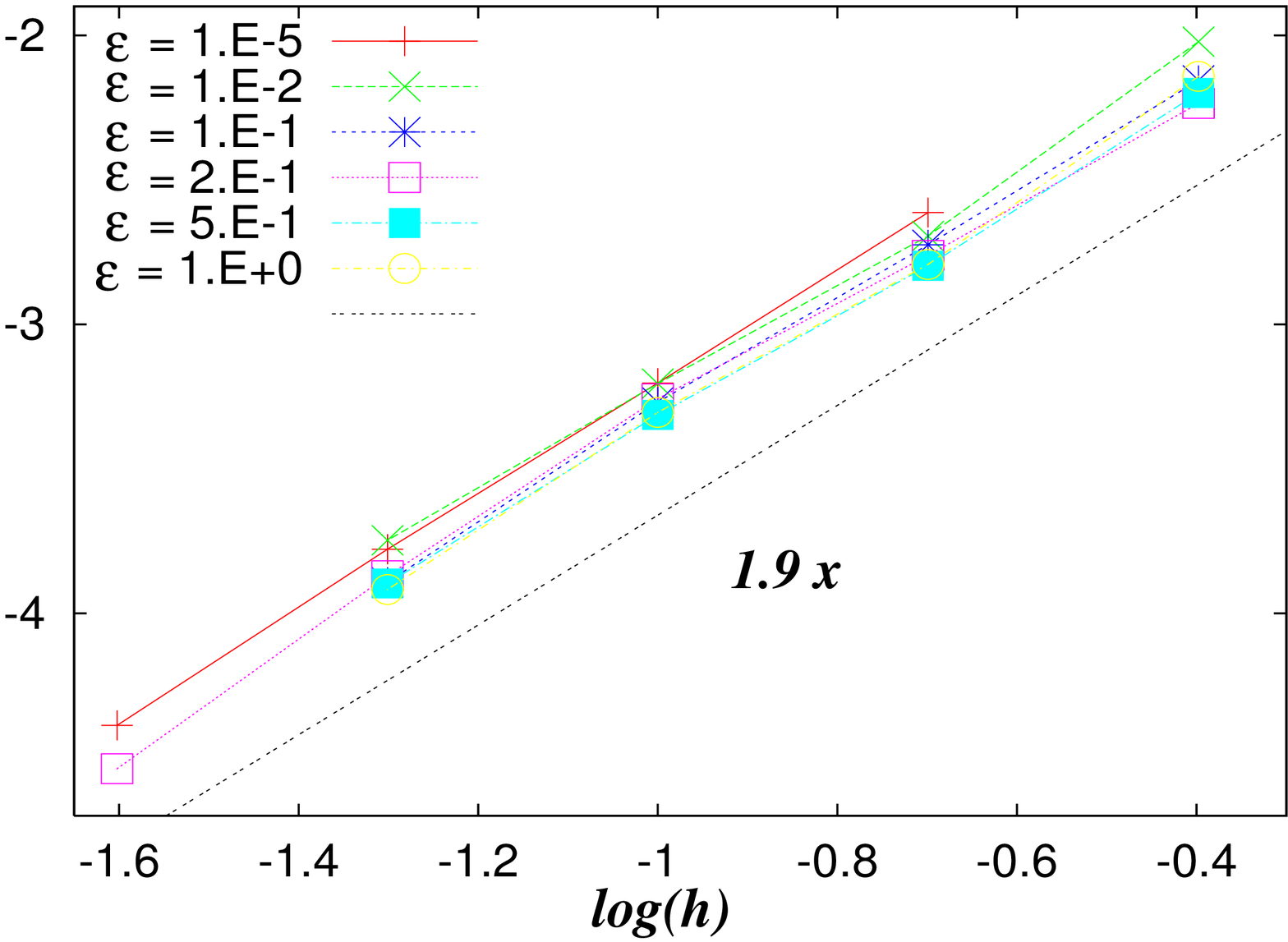}   
\\
(1)&(2)
\end{tabular}
\caption{The $L^1$ and $L^\infty$ errors of the second order method
(\ref{sch:02}) for different values  of the Knudsen number $\varepsilon\,=\,10^{-5},\dots,1$.}
\label{fig:02}
\end{figure}

 In Figure \ref{fig:02}, the $L^1$ and $L^\infty$ errors of the second order
method (\ref{sch:02}) are
 presented.  They show a uniformly second order convergence rate
(an estimation of the slope is $1.9$) in space and time (the velocity
discretization is spectrally accuracy in $v$ thus does not contribute
much to the errors). The time step  is not constrained by the value of 
$\varepsilon$, showing a uniform stability in time.

\subsection{The Sod tube problem}
This test deals with the numerical solution to the $1d_x\times 2d_v$ Boltzmann equation for Maxwellian molecules ($\gamma =0$). We present numerical simulations for one dimensional Riemann problem and compute an approximation for different Knudsen numbers, from rarefied regime  to the fluid regime.

Here, the initial data corresponding to the Boltzmann equations are given by 
the Maxwellian distributions computed from the following macroscopic quantities
\begin{eqnarray*}
\left\{
\begin{array}{ll}
(\rho_l, u_l, T_l) = (1,0,1)\,, &\textrm{ if } 0 \leq x \leq 0.5\,,
\\ 
\\
(\rho_r, u_r, T_r) = (0.125,0,0.25)\,, &\textrm{ if } 0.5 < x \leq 1\,.
\end{array} \right.
\end{eqnarray*}

We perform several computations for $\varepsilon\,=\, 1, 10^{-1}$, $10^{-2}$,...,$10^{-4}$. In Figure \ref{fig:03-2}, we only show the results obtained in the kinetic regime ($10^{-2}$) using a spectral scheme for the discretization of the collision operator \cite{FMP} (with $n_v=32^2$  and a truncation of the velocity domain $v_{\rm max}=7$) and second order explicit Runge-Kutta  and 
second order method (\ref{sch:02}) for the time discretization with a time step $\Delta t=0.005$ satisfying the CFL condition for the transport part (with $n_x=100$). For such a value of $\varepsilon$, the problem is not stiff and this test is only performed to compare the accuracy of our second order scheme (\ref{sch:02}) with the classical (second order) Runge-Kutta method. We present several snapshots of the density, mean velocity, temperature and heat flux 
$$
\QQ(t,x) \,:=\, \frac{1}{\varepsilon}\,\int_{\R^{d_v}} (v-u_\varepsilon)\,|v-u_\varepsilon|^2 \,f_\varepsilon(t,x,v)dv 
$$
at different time $t=0.10$ and $0.20$. Both results agree well with only $n_x=100$  in the space domain and $n_v=32$ for the velocity space. Thus, in the kinetic regime our second order method (\ref{sch:02}) gives the same accuracy as a second order fully explicit scheme without any additional computational effort.


\begin{figure}[htbp]
\begin{tabular}{cc}
\includegraphics[width=7.75cm]{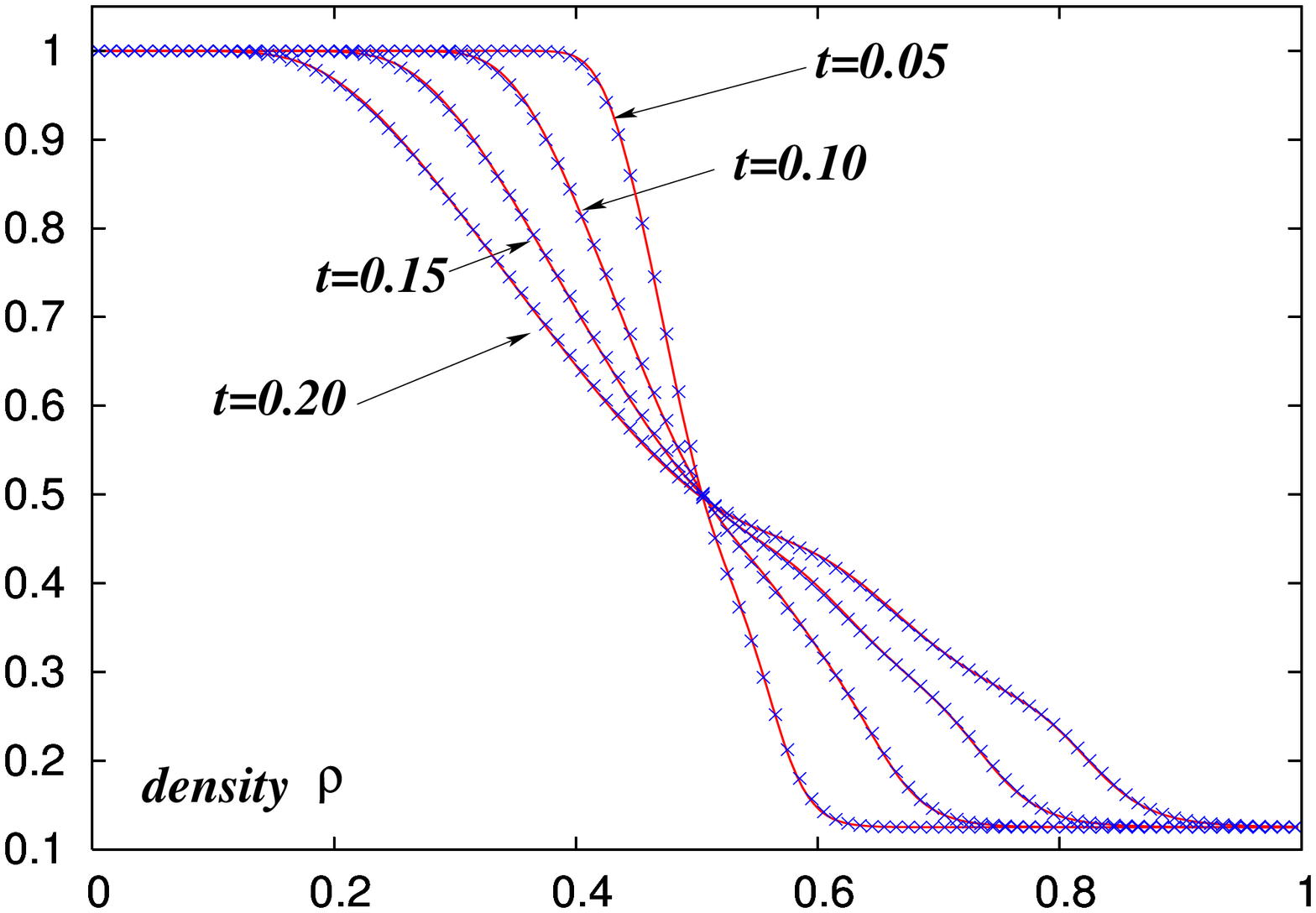}    
&
\includegraphics[width=7.75cm]{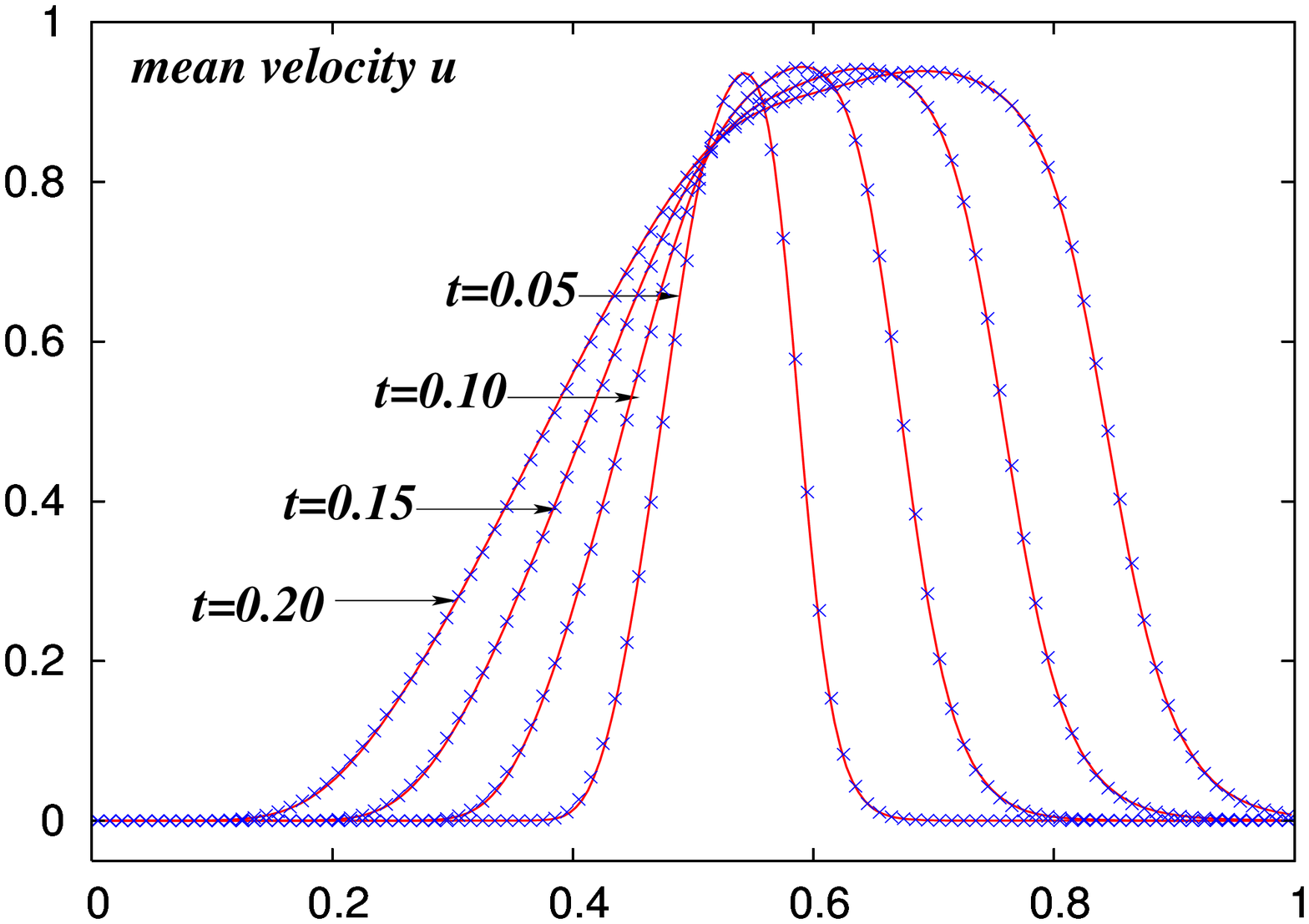}    
\\
(1)&(2)
\\
\includegraphics[width=7.75cm]{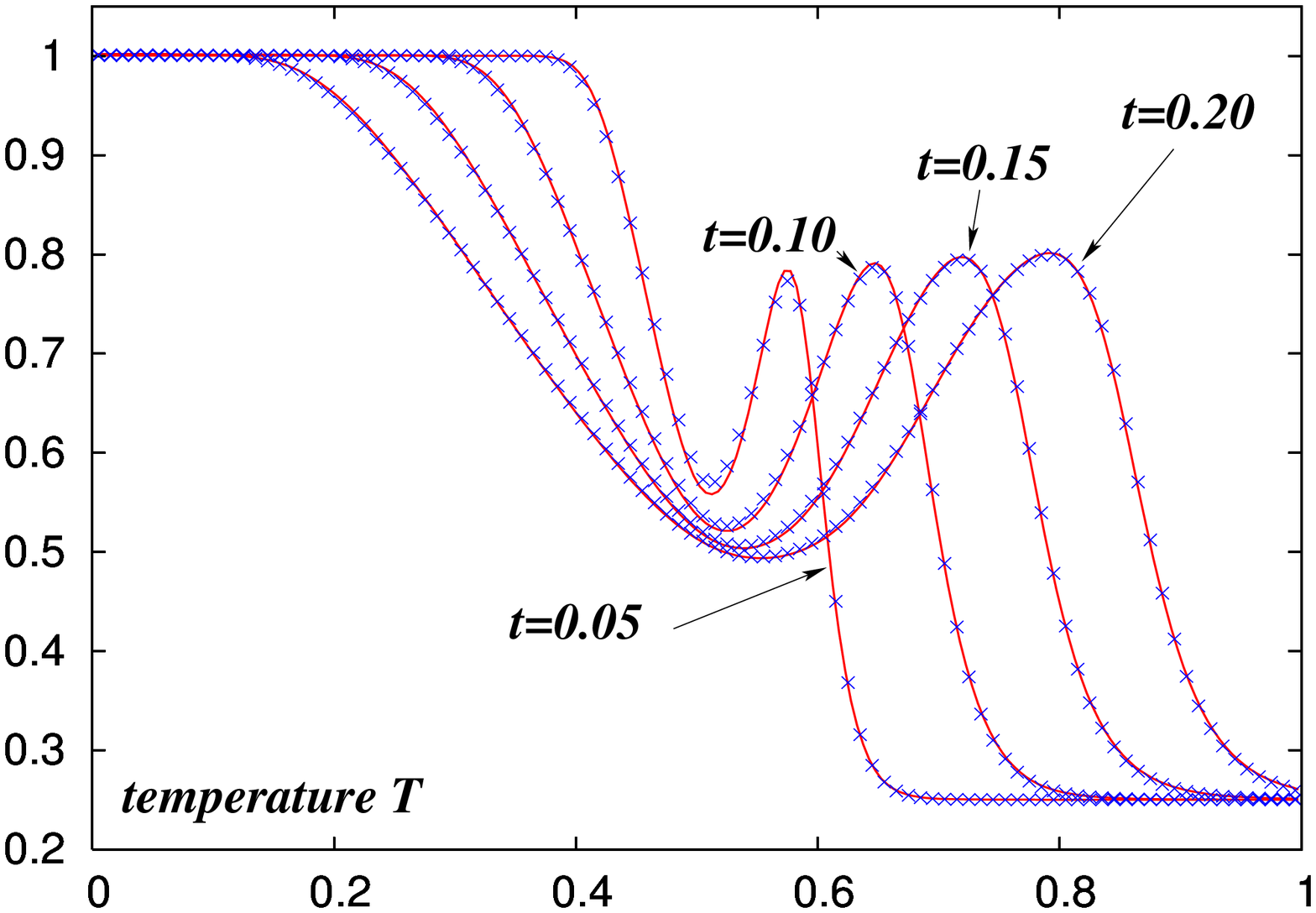}   
&
\includegraphics[width=7.75cm]{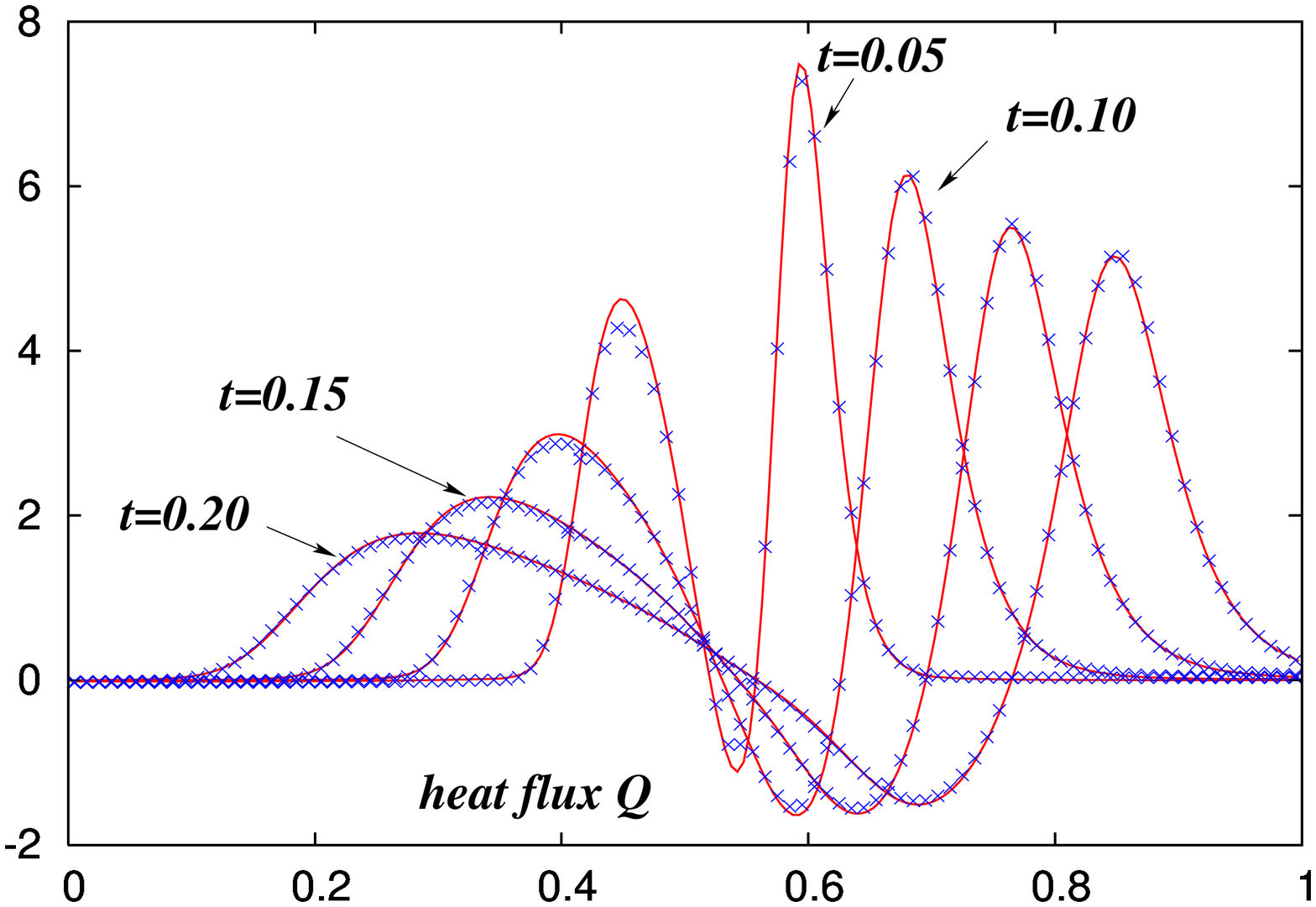}   
\\
(3)&(4)
\end{tabular}
\caption{Sod tube problem ($\varepsilon=10^{-2}$), dots ({\tt x}) represent the numerical solution obtained with our second order method (\ref{sch:02}) and lines with the Runge-Kutta method: evolution of  (1) the density $\rho$, (2) mean velocity $u$, (3) temperature $T$ and  (4) heat flux $\QQ$ at time $t=0.05$, $0.1$, $0.15$ and $0.2$.}
\label{fig:03-2}
\end{figure}

Now, we investigate the cases of small values of $\varepsilon$ for which an 
explicit scheme requires the time step to be of order $O(\var)$. In order to 
evaluate the accuracy of our method (\ref{sch:02}) in the Navier-Stokes regime (for small $\varepsilon\ll 1$ but not negligible), we  compared  the numerical solution  for $\varepsilon=10^{-3}$ with one obtained with a small time step $\Delta t =O(\varepsilon)$ (for  which the computation is still feasible). Note that a direct comparison with the numerical solution to the compressible Navier-Stokes system (\ref{eq:CNS}) is difficult since  the viscosity $\mu_\varepsilon =\mu(T_\varepsilon)$ and the thermal conductivity $\kappa_\varepsilon =\kappa(T_\varepsilon)$ are not explicitly known. Therefore, in  Figure~\ref{fig:04-1}, we report the numerical results for  $\varepsilon=10^{-3}$ and propose a comparison between the numerical solution obtained with the scheme (\ref{sch:02}) and the one obtained with a second order explicit Runge-Kutta method. In this case, the behavior of macroscopic quantities (density, mean velocity, temperature and heat flux) agree very well even it the time step is at least ten times larger with our method (\ref{sch:01}) or (\ref{sch:02}).

\begin{figure}[htbp]
\begin{tabular}{cc}
\includegraphics[width=7.75cm]{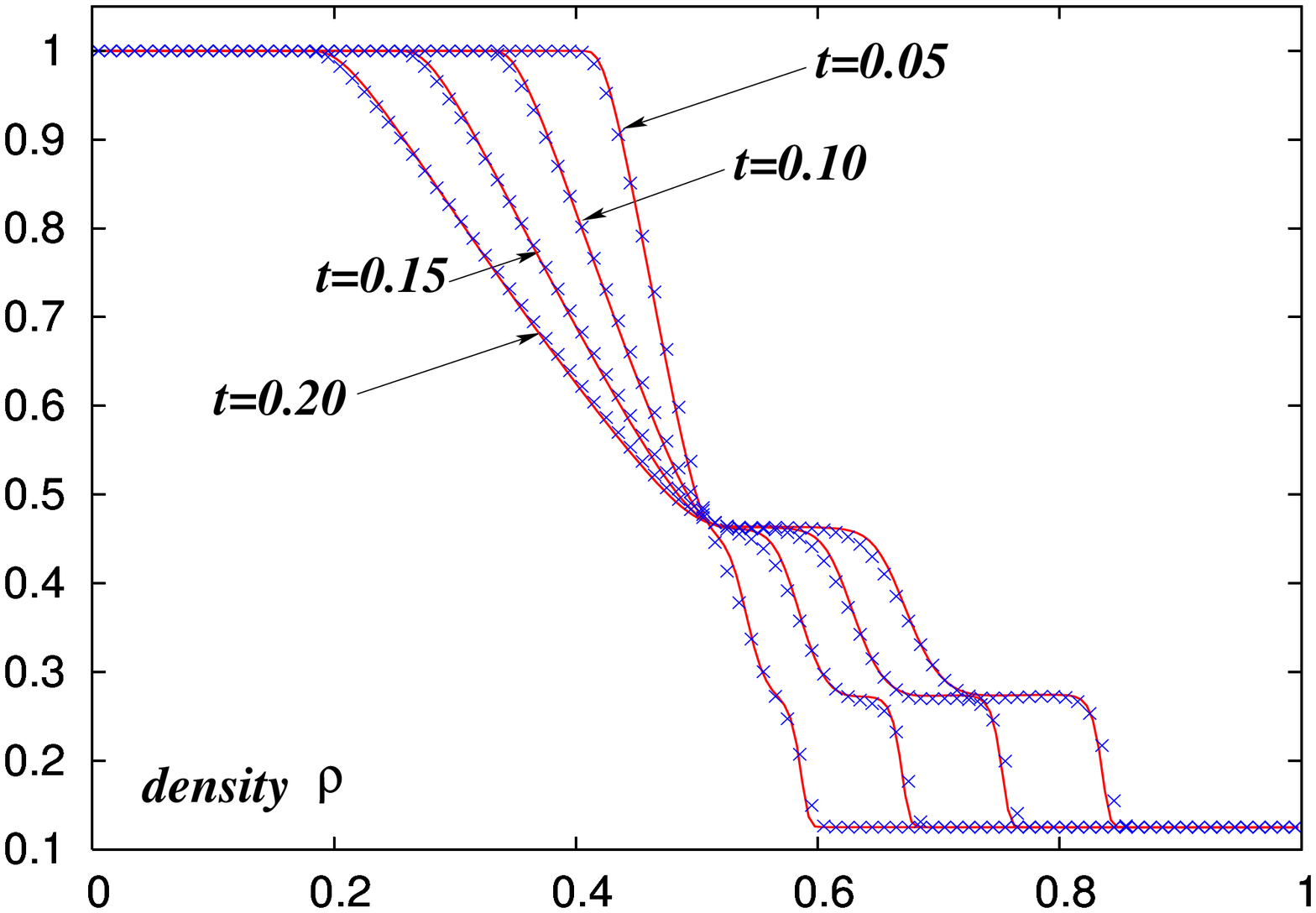}    
&
\includegraphics[width=7.75cm]{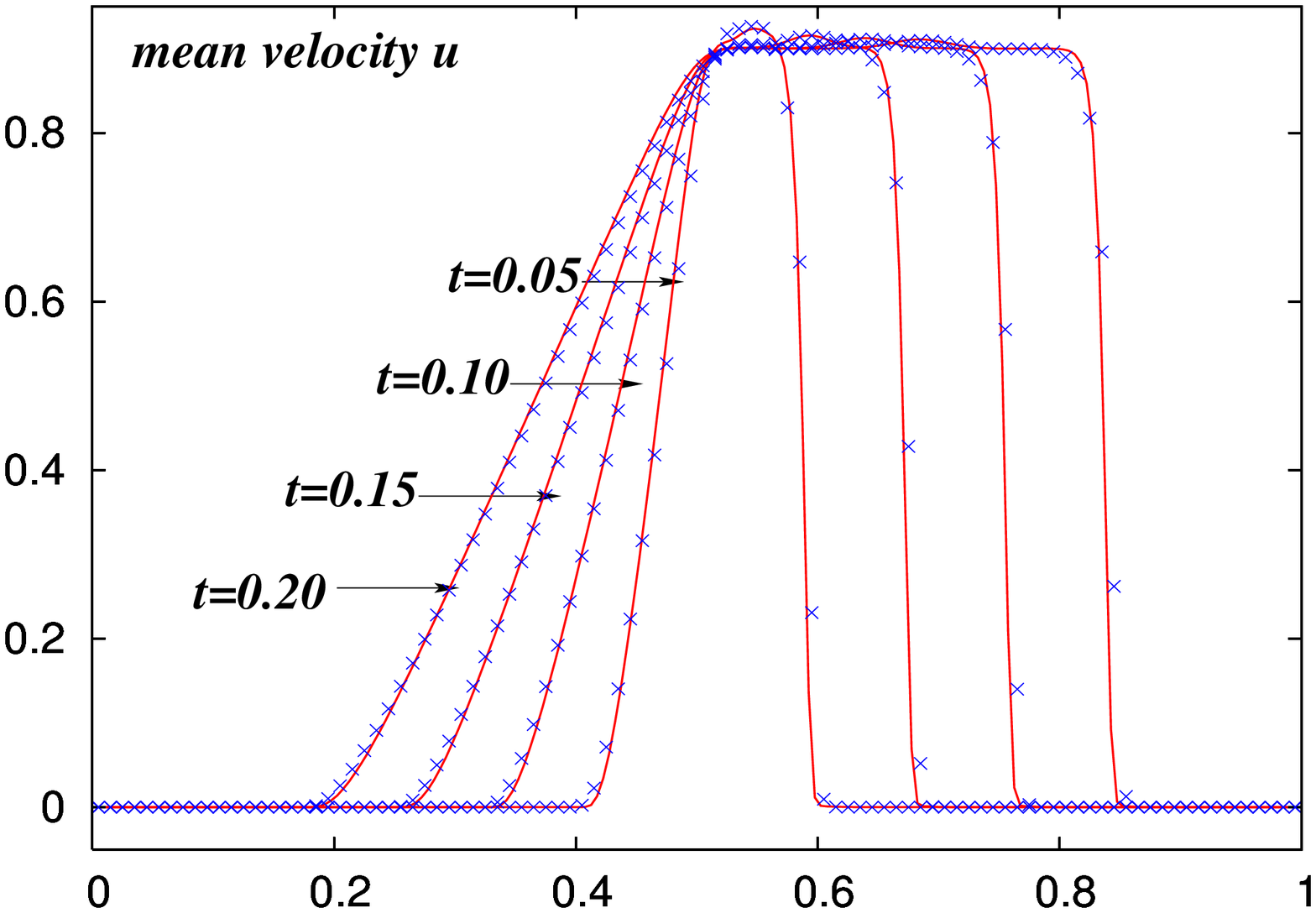}    
\\
(1)&(2)
\\
\includegraphics[width=7.75cm]{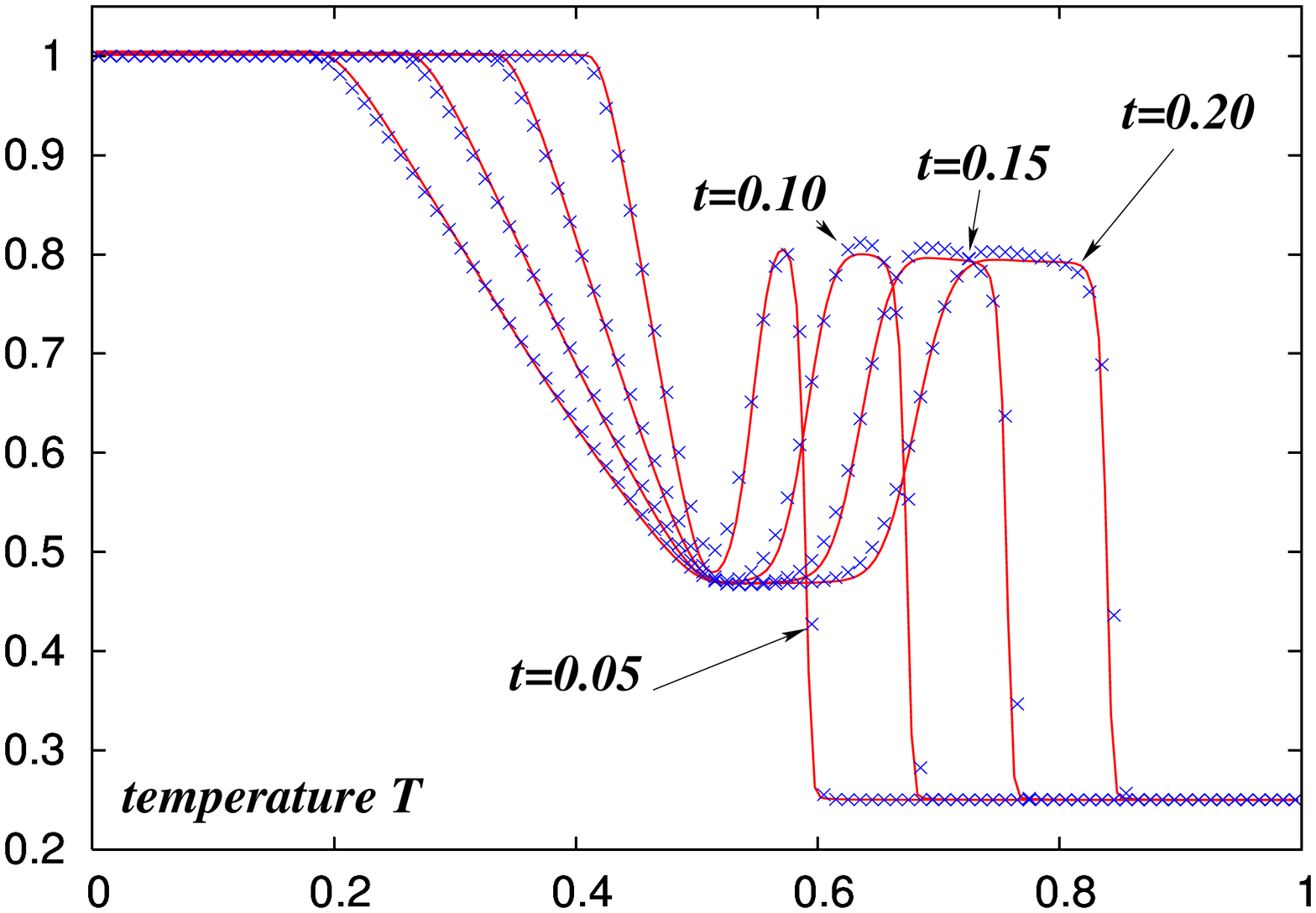}   
&
\includegraphics[width=7.75cm]{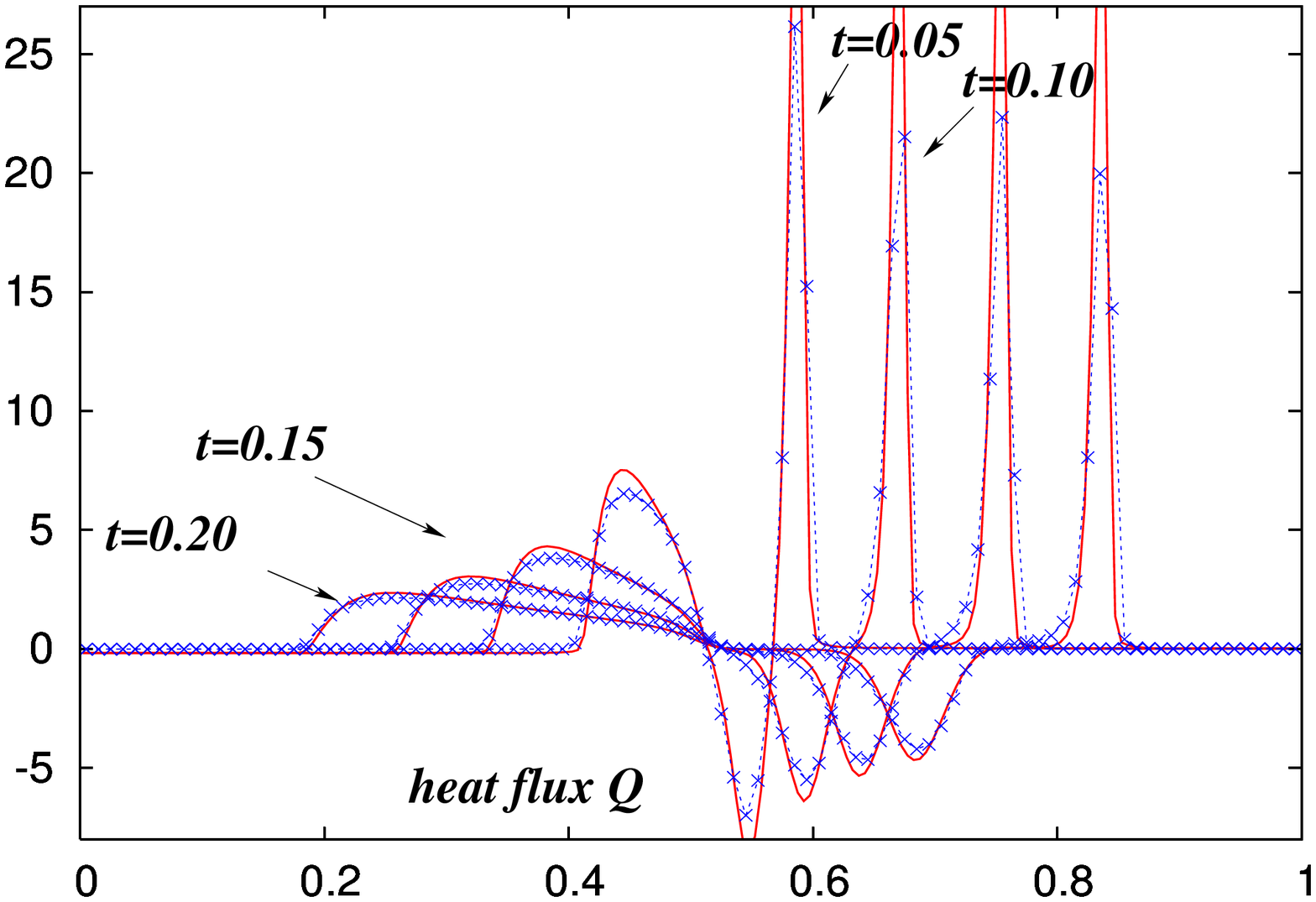}   
\\
(3)&(4)
\end{tabular}
\caption{Sod tube problem ($\varepsilon=10^{-3}$), dots ({\tt x}) represent the numerical solution obtained with our second order method (\ref{sch:02}) and lines with the Runge-Kutta method: evolution of  (1) the density $\rho$, (2) mean velocity $u$, (3) temperature $T$ and  (4) heat flux $\QQ$ at time $t=0.05$, $0.1$, $0.15$ and $0.2$.}
\label{fig:04-1}
\end{figure}

Finally in  Figure~\ref{fig:04-2}, we compare the numerical solution of the Boltzmann equation (\ref{eq:B}) with the numerical solution to the compressible Navier-Stokes system derived from the BGK model since the viscosity and heat conductivity are in that case explicitly known \cite{BLM}.  To approximate
 the compressible Navier-Stokes system, we apply a second order Lax-Friedrich scheme using a large number of points ($n_x=1000$) whereas we only used $n_x=100,$ and $200$ points in space and $n_v^2=32^2$ points in velocity for the approximation of the kinetic equation (\ref{eq:B}). In this problem, the density, mean velocity and temperature are relatively close to the one obtained with the approximation of the Navier-Stokes system. Even the qualitative behavior of the heat flux  agrees well with  the heat flux  corresponding to the compressible Navier-Stokes  system $\kappa_\varepsilon\,\nabla_x T_\varepsilon$, with $\kappa_\varepsilon=\rho_\varepsilon\,T_\varepsilon$ (see Figure~\ref{fig:04-2}), yet some differences can be observed, which means that the use of BGK models to derive macroscopic models has a strong influence on the heat flux.


\begin{figure}[htbp]
\begin{tabular}{cc}
\includegraphics[width=7.75cm]{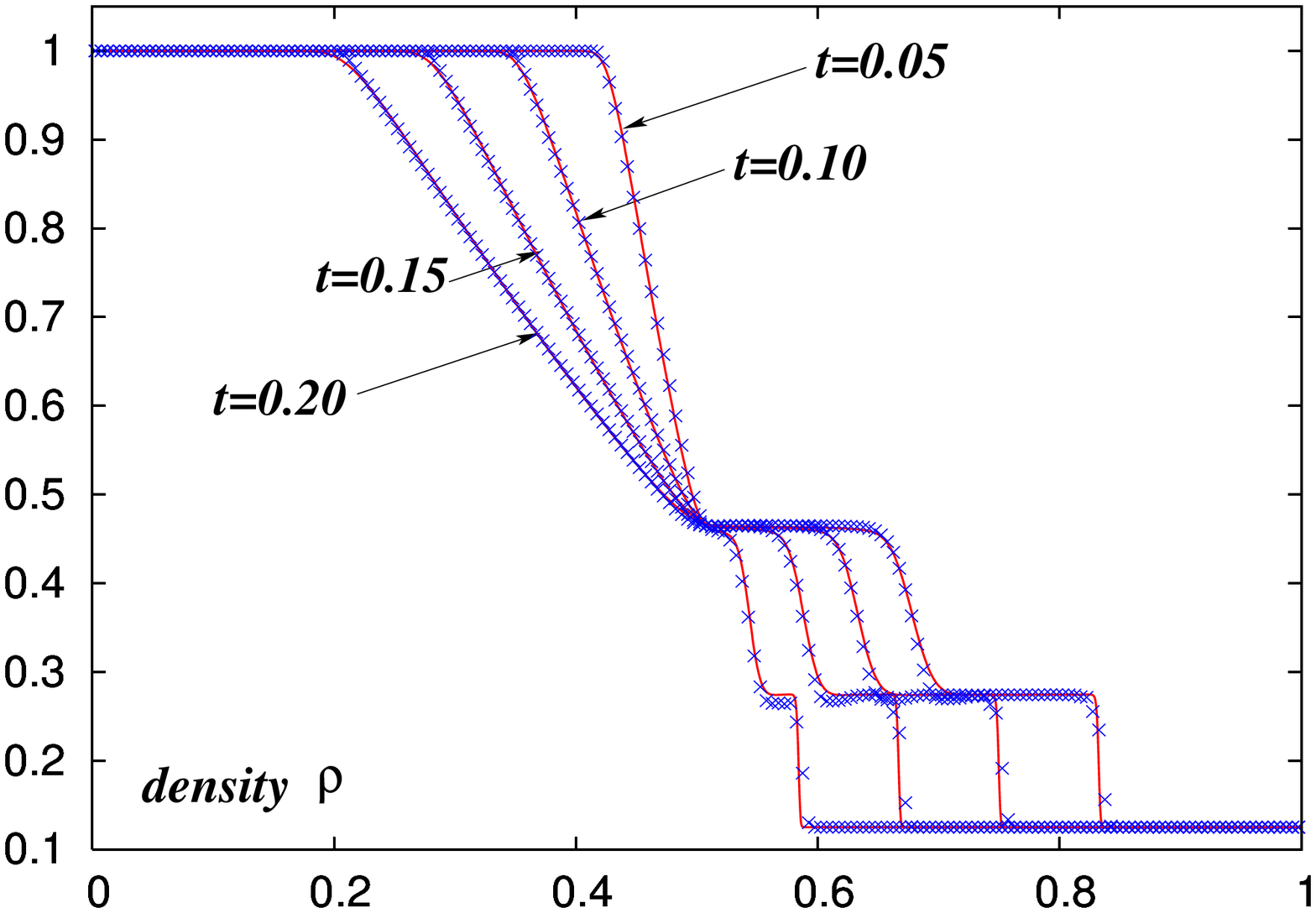}    
&
\includegraphics[width=7.75cm]{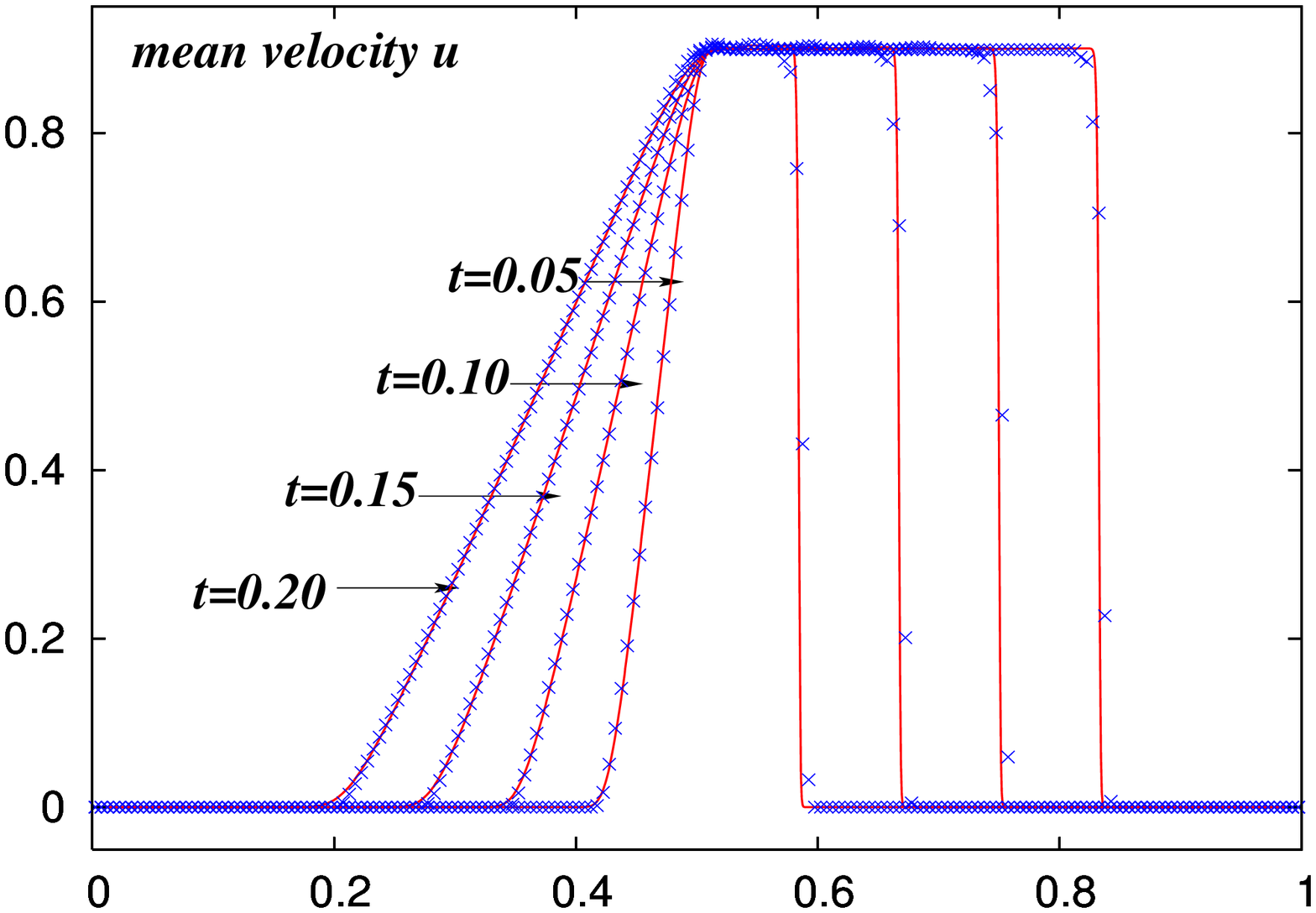}    
\\
(1)&(2)
\\
\includegraphics[width=7.75cm]{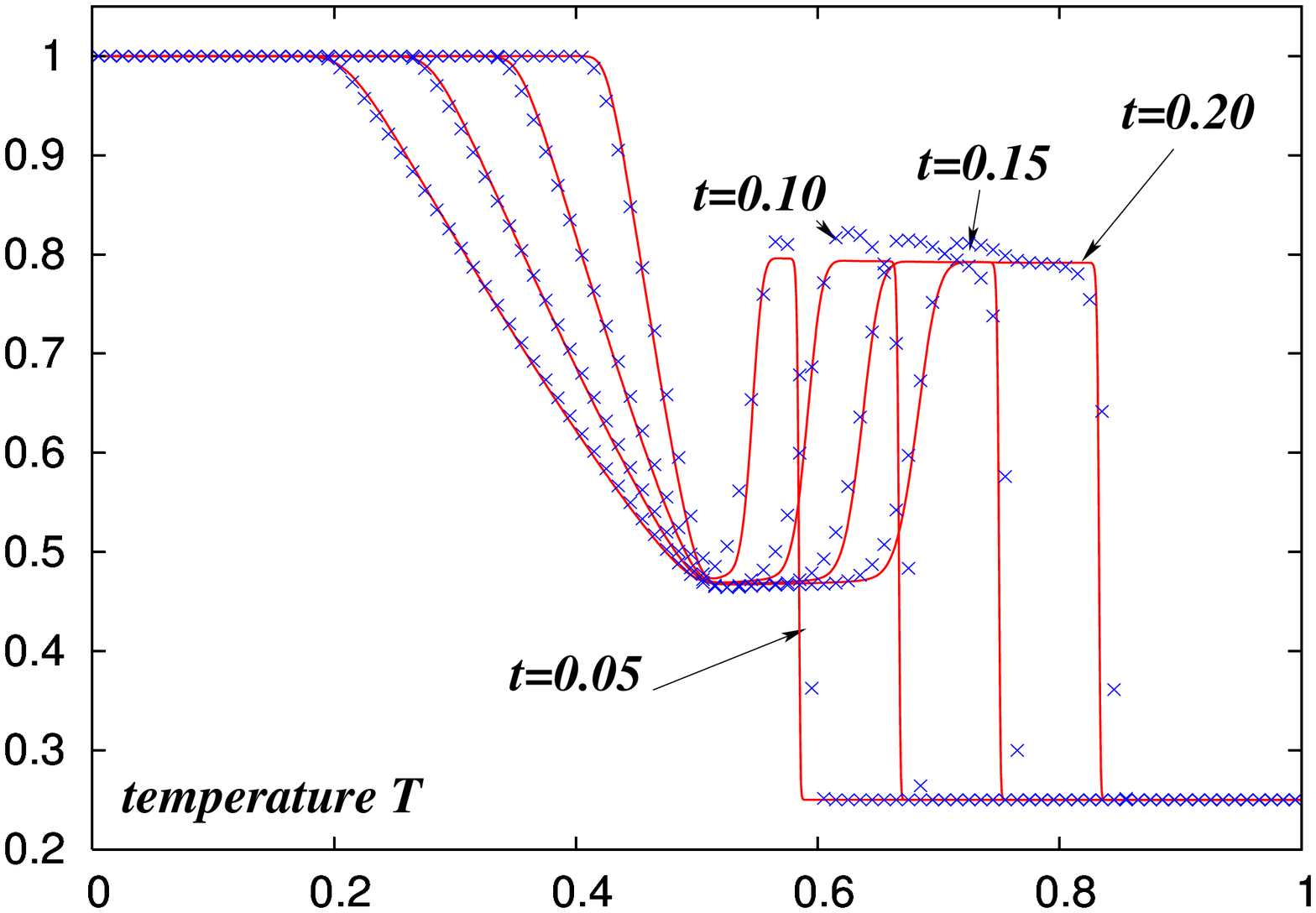}   
&
\includegraphics[width=7.75cm]{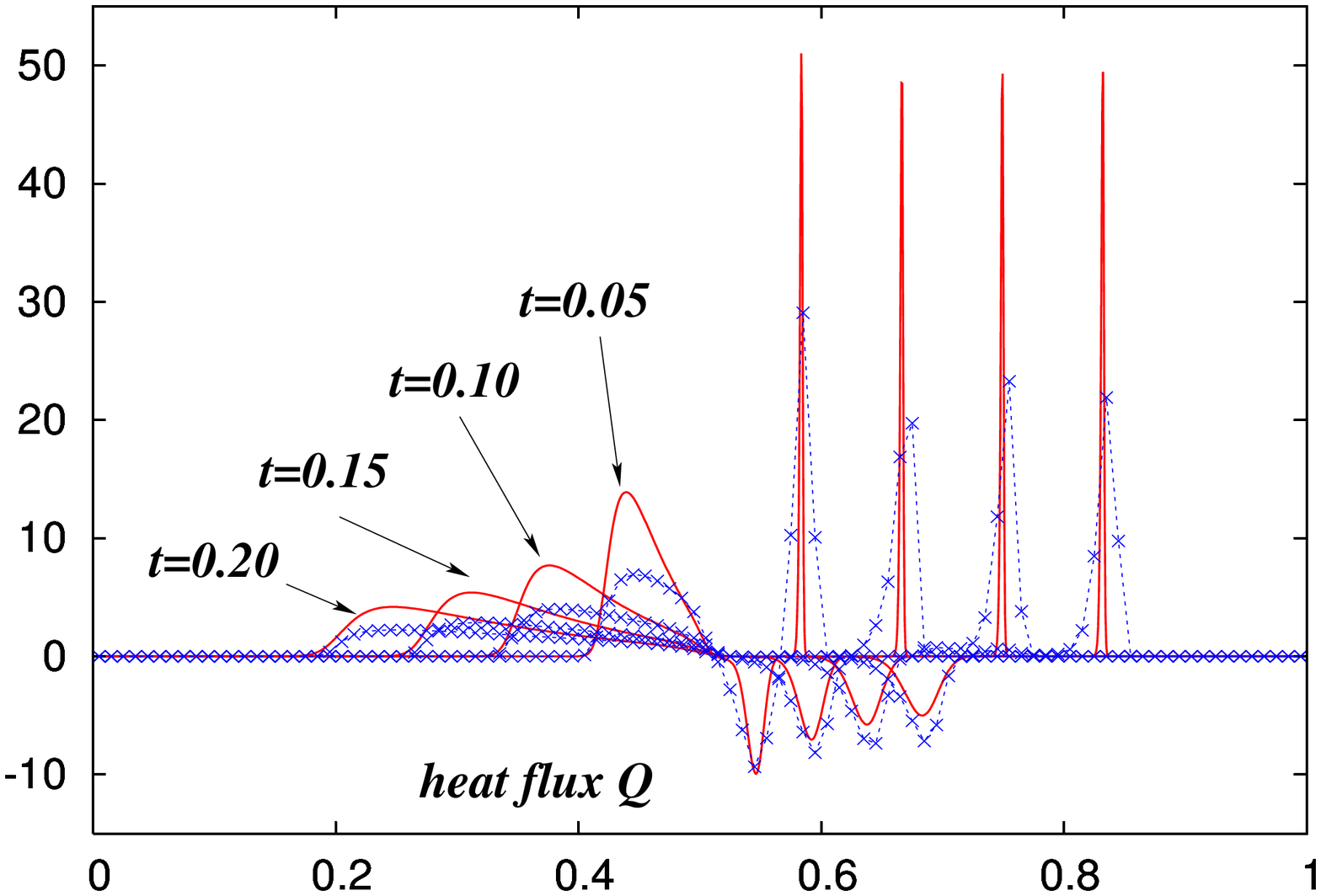}   
\\
(3)&(4)
\end{tabular}
\caption{Sod tube problem ($\varepsilon=10^{-4}$), comparison of the numerical solution to the Boltzmann equation with the our second order method
 (\ref{sch:02}) represented with dots ({\tt x}) with the numerical solution to the compressible Navier-Stokes system (lines):  evolution of  (1) the density $\rho$, (2) mean velocity $u$, (3) temperature $T$ and  (4) heat flux $\QQ$ at time $t=0.05$, $0.1$, $0.15$ and $0.2$.}
\label{fig:04-2}
\end{figure}

\subsection{A problem with mixing regimes}
Now we consider the Boltzmann equation (\ref{eq:B}) with the Knudsen number 
$\varepsilon>0$ depending on the space variable in a wide range of mixing 
scales.
\begin{center}
\begin{tabular}{cc}
\begin{minipage}{8.4cm}
This kind of problem was already studied by several authors  for the BGK model
 \cite{DJM} or the radiative transfer equation \cite{JPT2}. In this
problem,  $\varepsilon:\,\R\mapsto \R^+$ is given by  
$$
\varepsilon(x) = \varepsilon_0 \,+\, \frac{1}{2}\left[\tanh(1-11\,x)\,+\, \tanh(1+11\,x)\,\right],
$$

which varies smoothly from $\varepsilon_0$  to $O(1)$.
\end{minipage} &
\begin{minipage}{6.76cm}
\includegraphics[width=6.75cm]{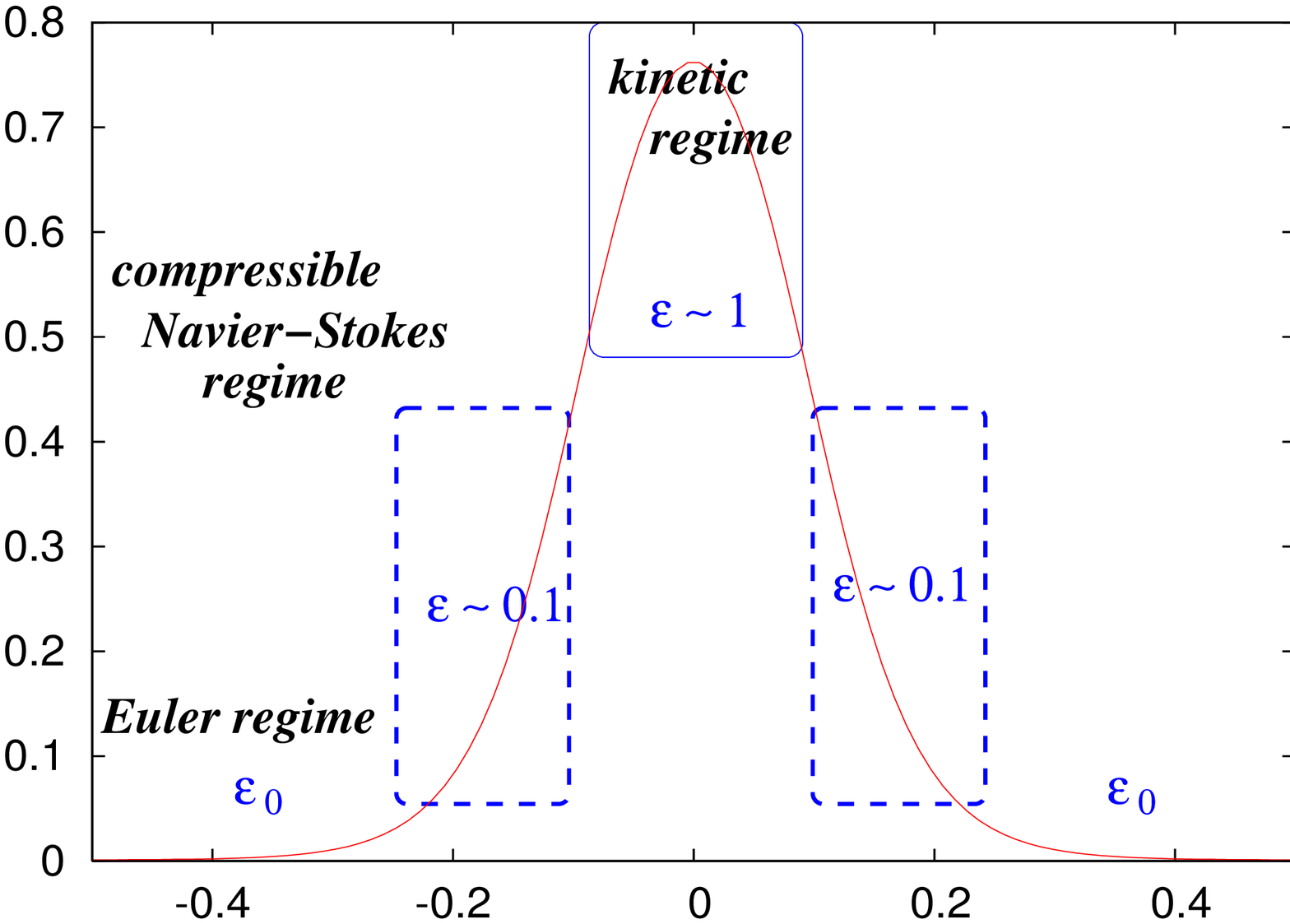}    
\end{minipage}
\end{tabular}
\end{center}

This numerical test is difficult because different scales are involved. It 
requires a good accuracy of the numerical scheme for all range of $\var$.  In order to focus on the multi-scale nature we only consider periodic boundary conditions, even if the method has also been used with specular reflection in space. Furthermore, to increase the difficulty we consider an initial data which is far from the local equilibrium of the collision operator:
$$
f_0(x,v) = \frac{\rho_0}{2}\left[ \,\exp\left(-\frac{|v-u_0|^2}{T}\right) \,+\,\exp\left(-\frac{|v+u_0|^2}{T_0}\right)\,\right], \quad x\,\in\, [-L,L], \,\,v\,\in\,\R^2
$$
with $u_0=(3/4,-3/4)$, 
$$
\rho_0(x)\,=\, \frac{2 + \sin(k\,x)}{2}, \quad T_0(x)\,=\, \frac{5+2\,\cos(k\,x)}{20}
$$
where $k=\pi/L$ and  $L=1/2$.

Here we cannot compare the numerical solution with the 
one obtained by a macroscopic model. From the numerical simulations, 
we observe that the solution  is  smooth during a short time and some discontinuities are formed in the region where the Knudsen number $\varepsilon$ is very small and then propagate into the physical domain.  

On the one hand, we only take $\varepsilon_0=10^{-3}$  in order to  propose a comparison of numerical solutions computed with a second order method
 using a time step $\Delta t=0.001$ (such that the CFL condition for the transport part is satisfied)   and the one by the second order explicit Runge-Kutta 
method with a smaller time step 
$\Delta t=0.0001$)
to get stability. The number of points in space is $n_x=200$ and in velocity is $n_v^2=32^2$. Clearly,  in  Figure~\ref{fig:05-1},  the results are in good agreement even if our new method does not solve accurately small time scales when 
the solution is for from the local equilibrium.  Moreover  in  
Figure~\ref{fig:05-2}, we present numerical results with only $n_x=50$ and 
$n_x=200$, and $n_v^2=32^2$ to show the performance of the method 
with a small number of discretization points in space. With $n_x=50$ points the qualitative behavior of the macroscopic quantities $(\rho,u,T)$ is fairly good.

On the other hand, we have performed different numerical results when $\varepsilon_0=10^{-4}$, then the variations of  $\varepsilon$ starts from $10^{-4}$ to $1$ in the space domain. In that case, the computational time of a fully explicit scheme would be more than one hundred times larger than the one required for the asymptotic preserving scheme (\ref{sch:02}). We observe that discontinuities appear on the density, mean velocity and temperature and then propagate accurately into the domain. The shock speed is roughly the same for the different numerical resolutions.  Therefore, this method gives a very good compromise between accuracy and stability for the different regimes. Numerical results are not plotted since they are relatively close to the ones presented in Figures \ref{fig:05-1} and \ref{fig:05-2}.    

\begin{center}
\begin{figure}[htbp]
\begin{tabular}{ccc}
\includegraphics[width=5.cm]{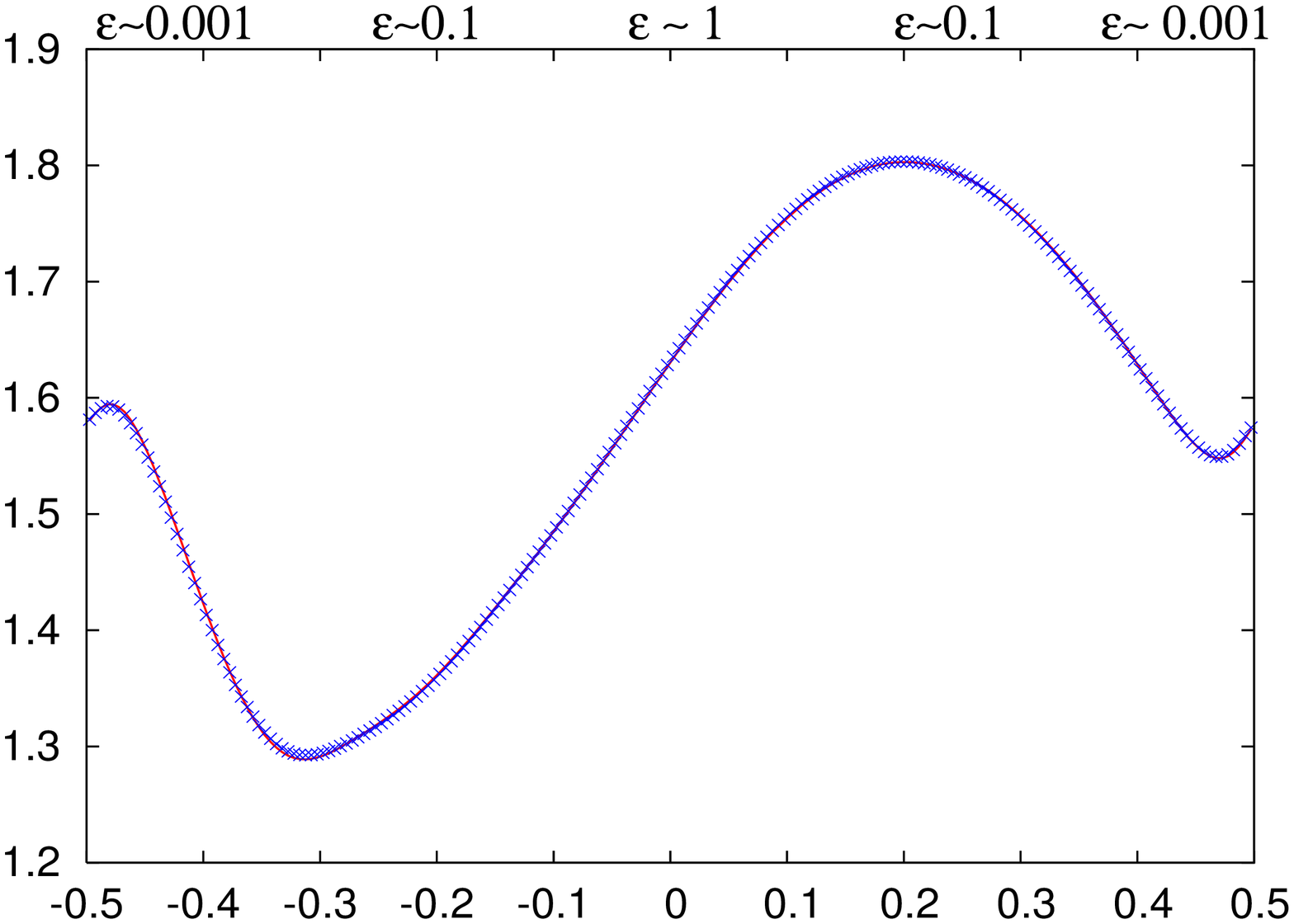}    
&
\includegraphics[width=5.cm]{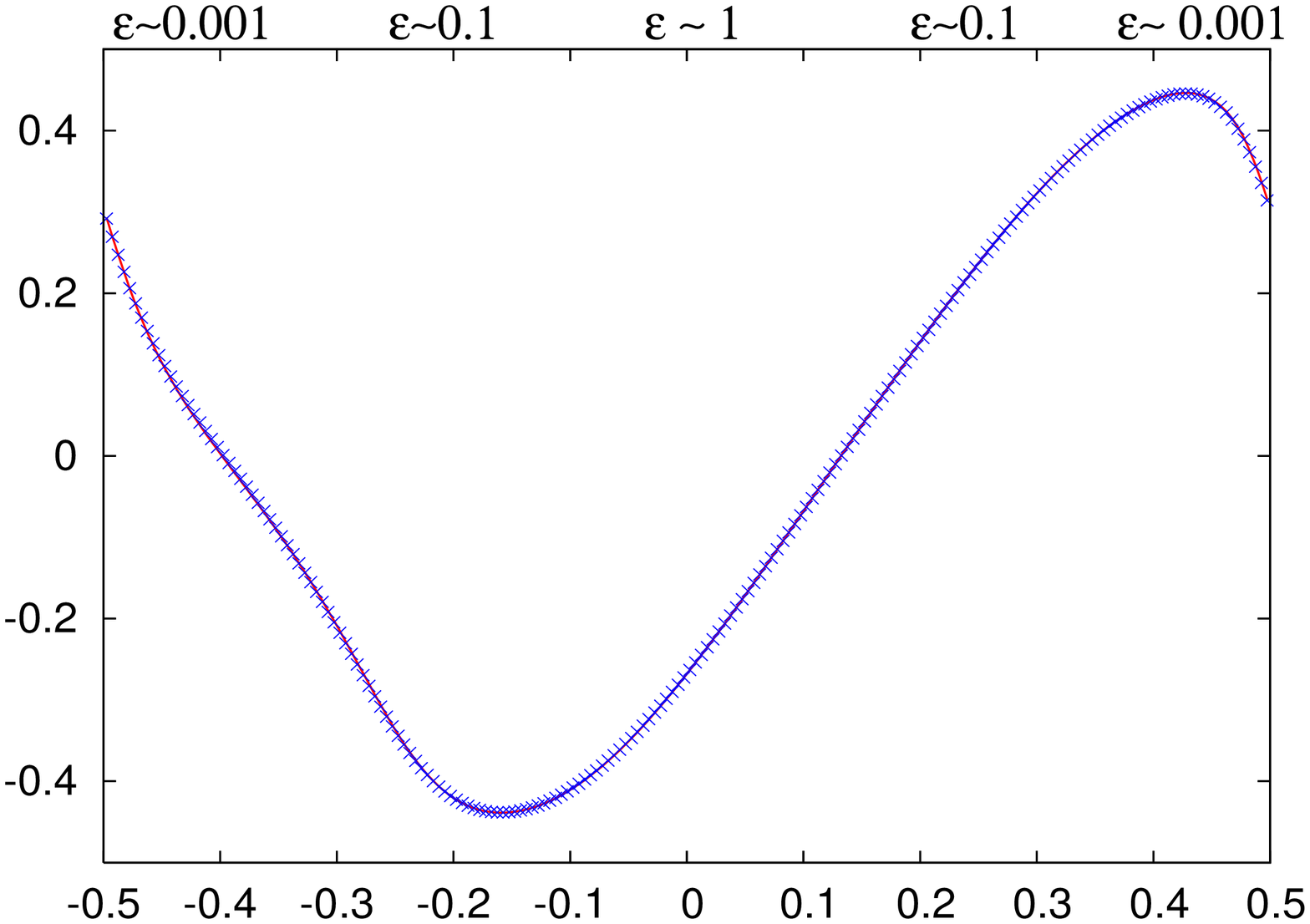}    
&
\includegraphics[width=5.cm]{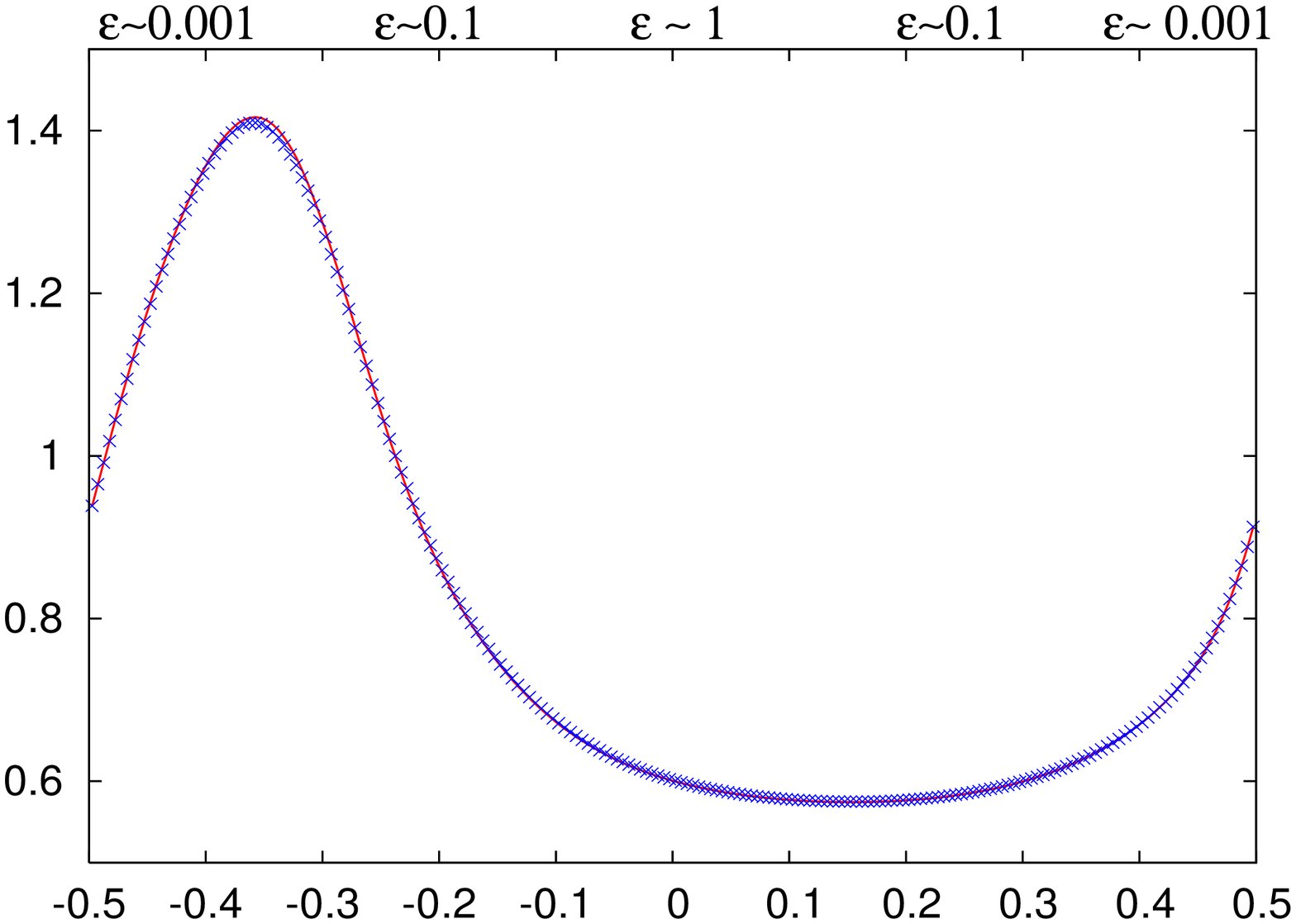}   
\\
\includegraphics[width=5.cm]{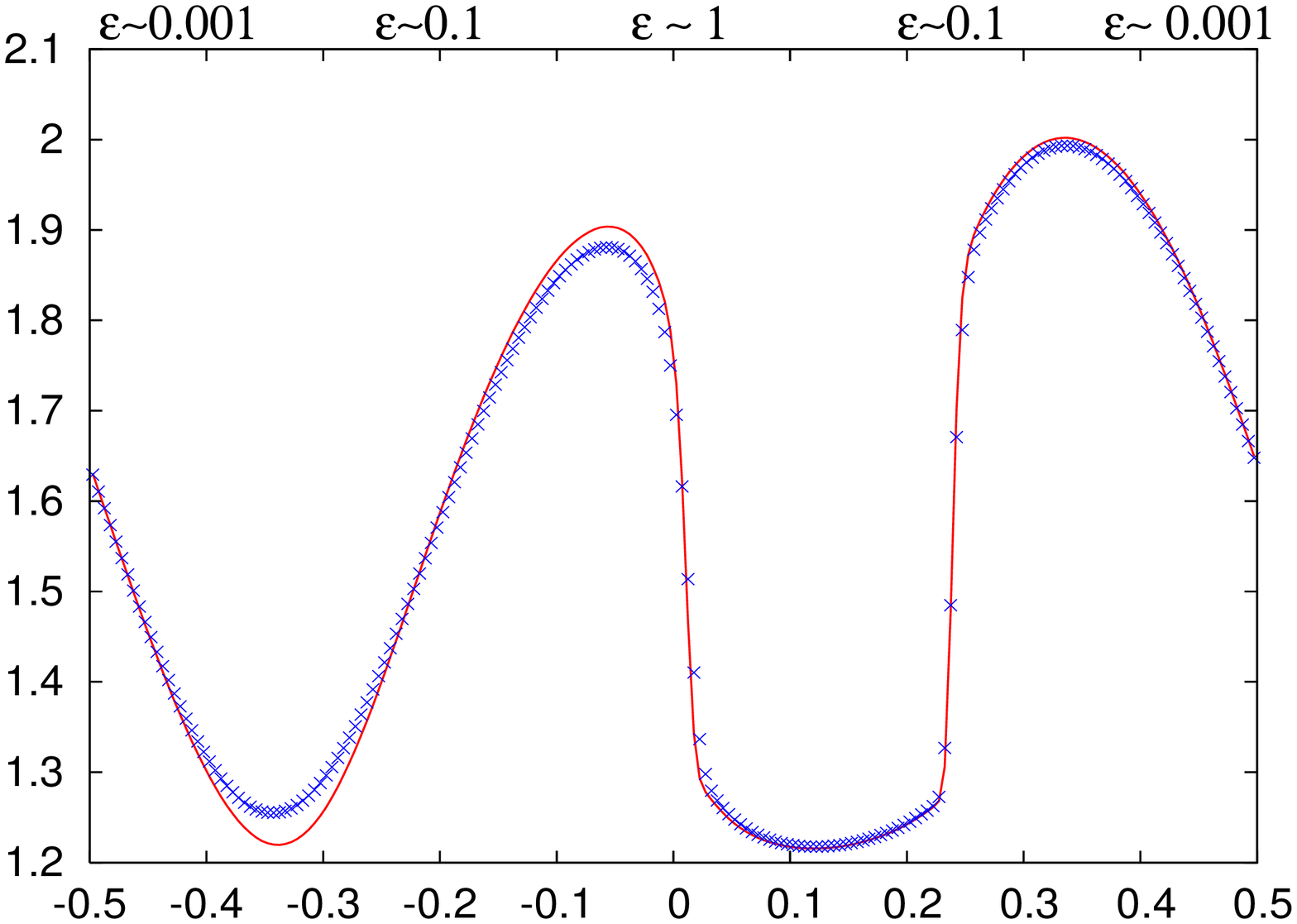}    
&
\includegraphics[width=5.cm]{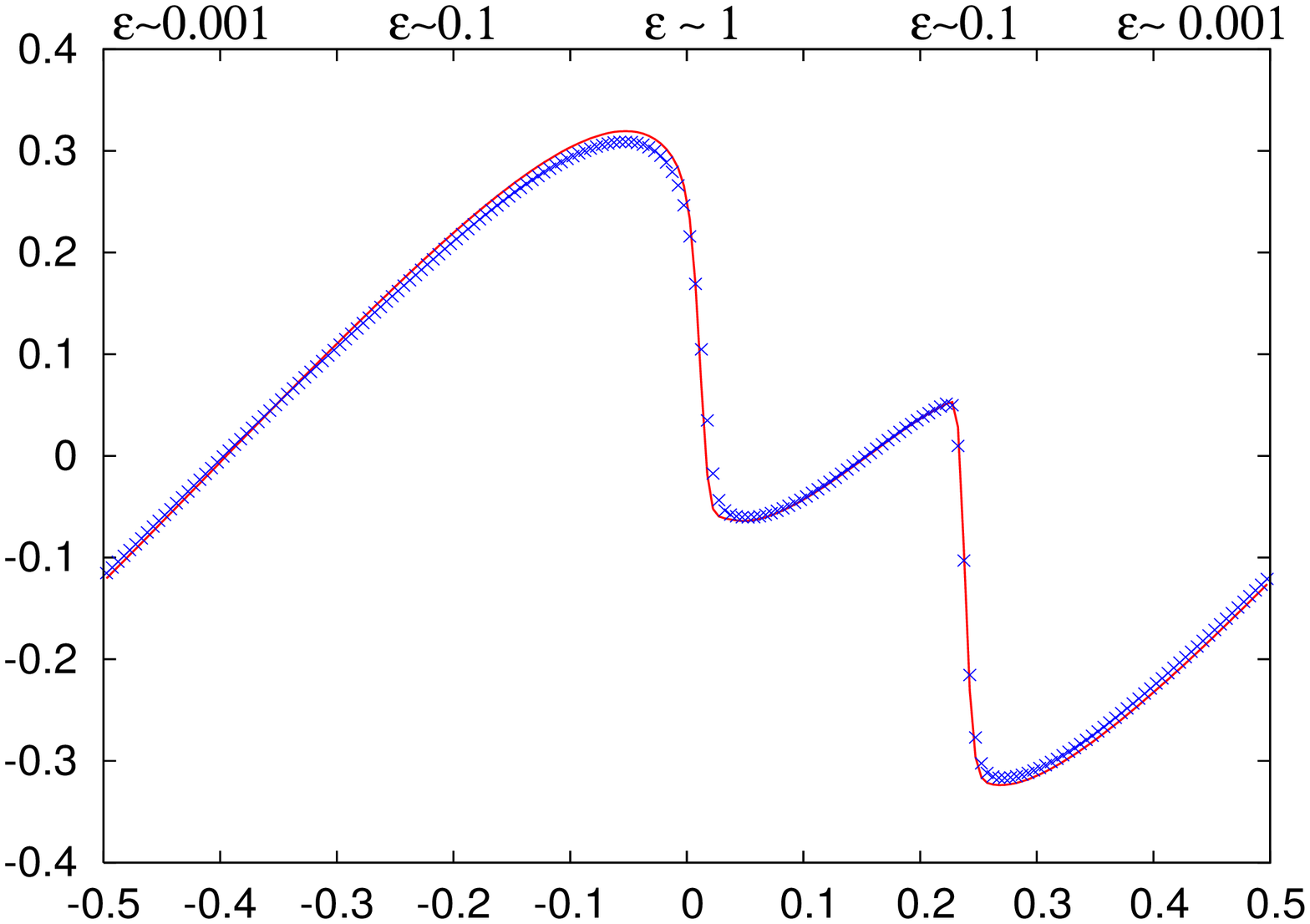}    
&
\includegraphics[width=5.cm]{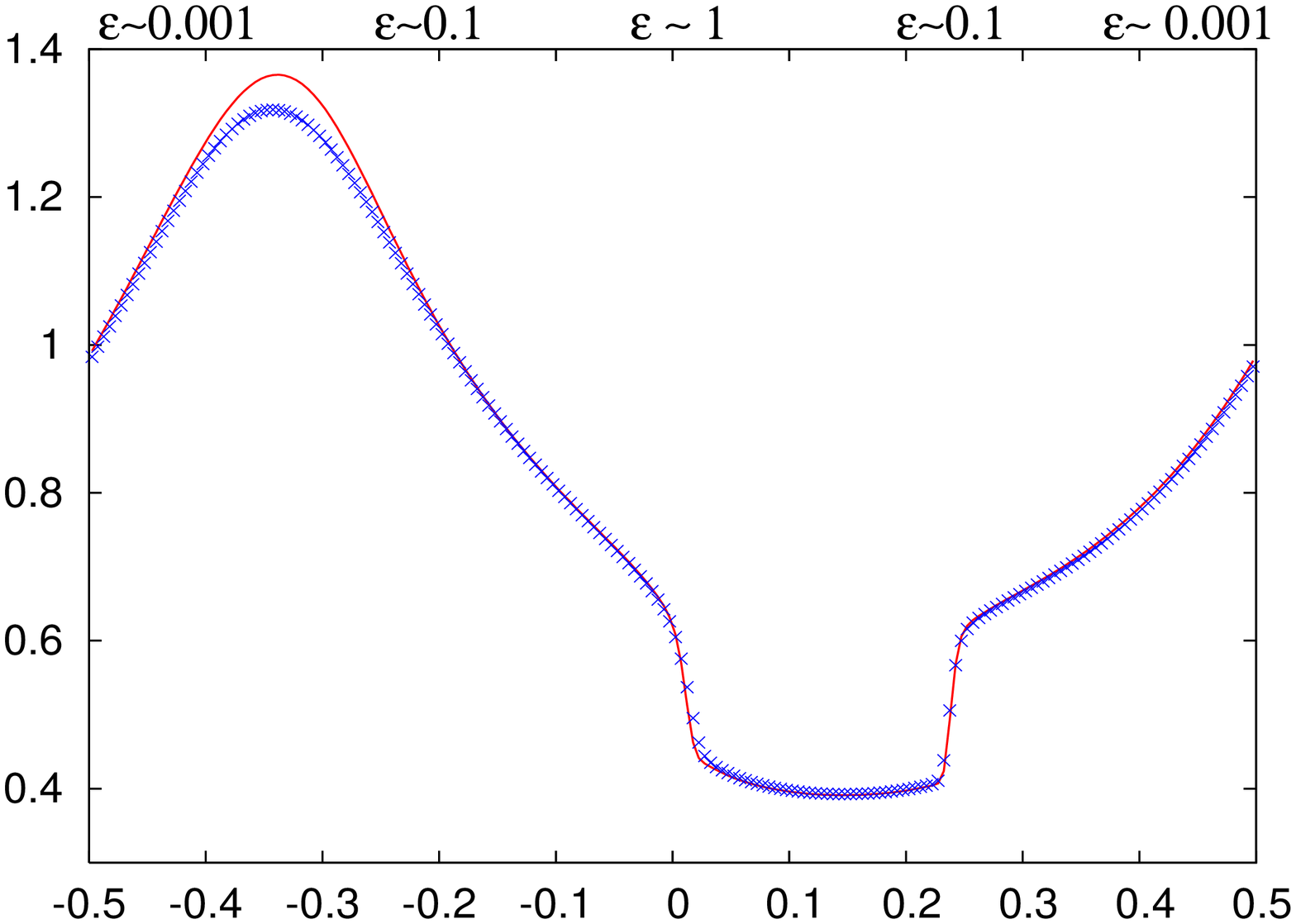}   
\\
\includegraphics[width=5.cm]{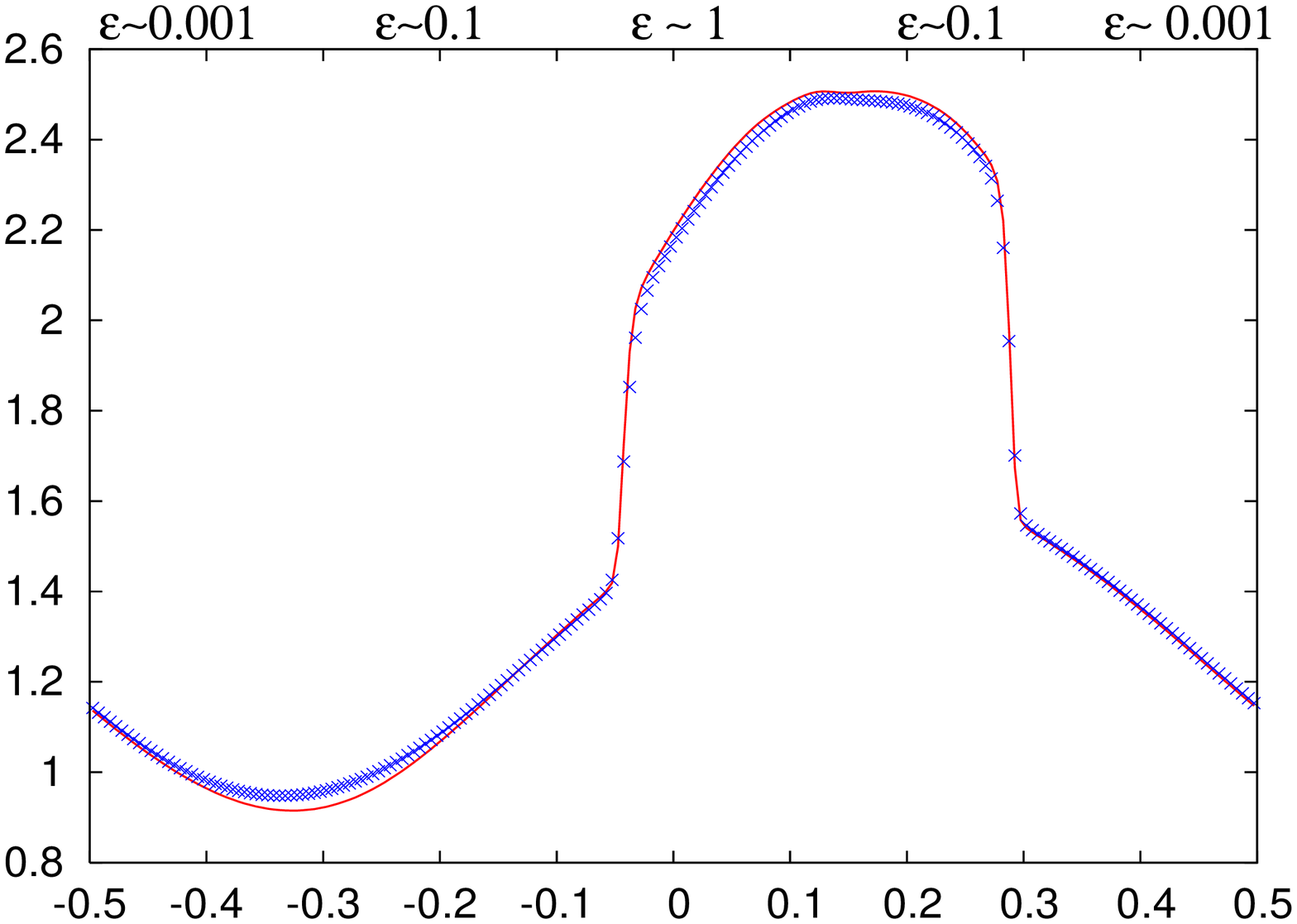}    
&
\includegraphics[width=5.cm]{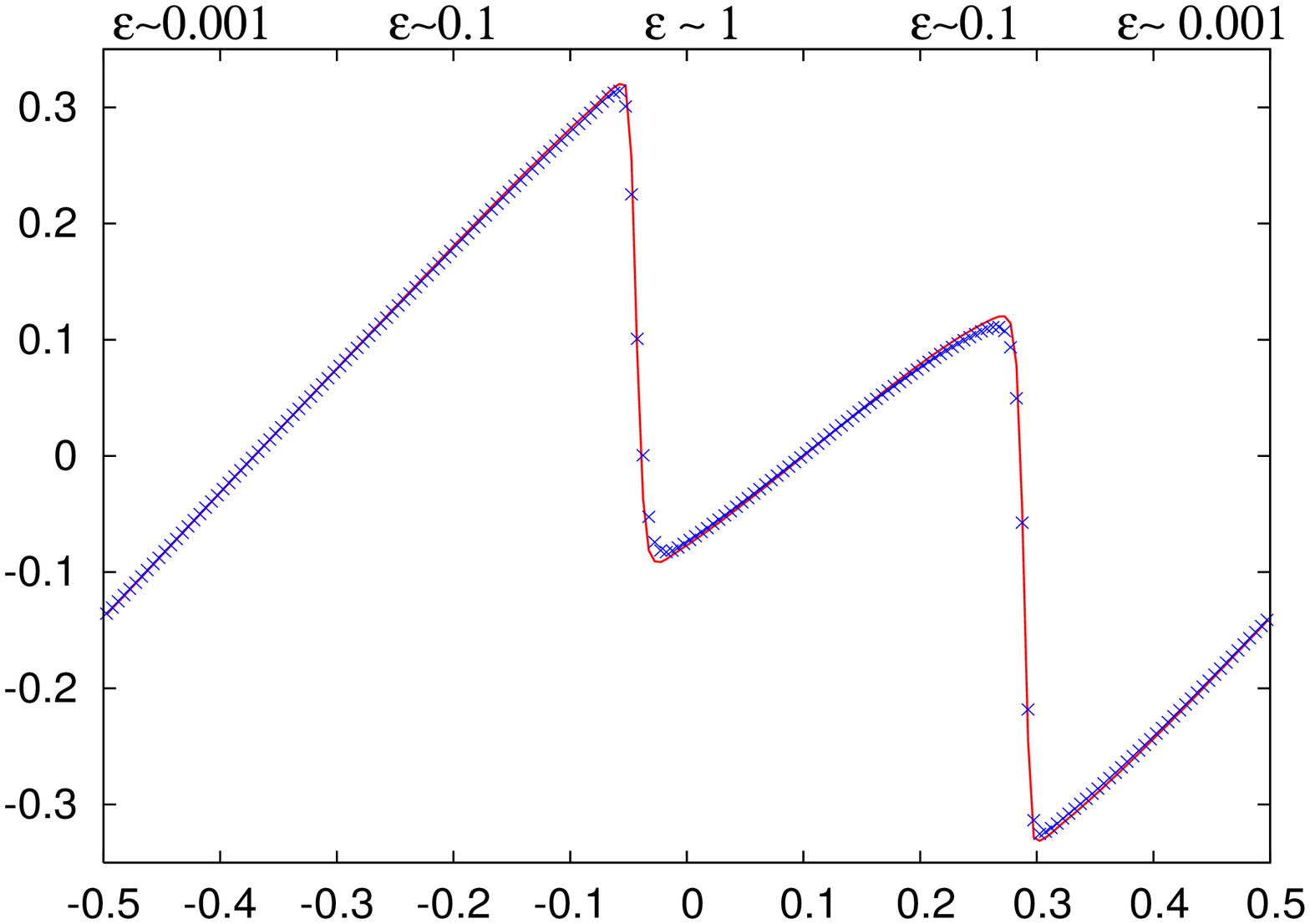}    
&
\includegraphics[width=5.cm]{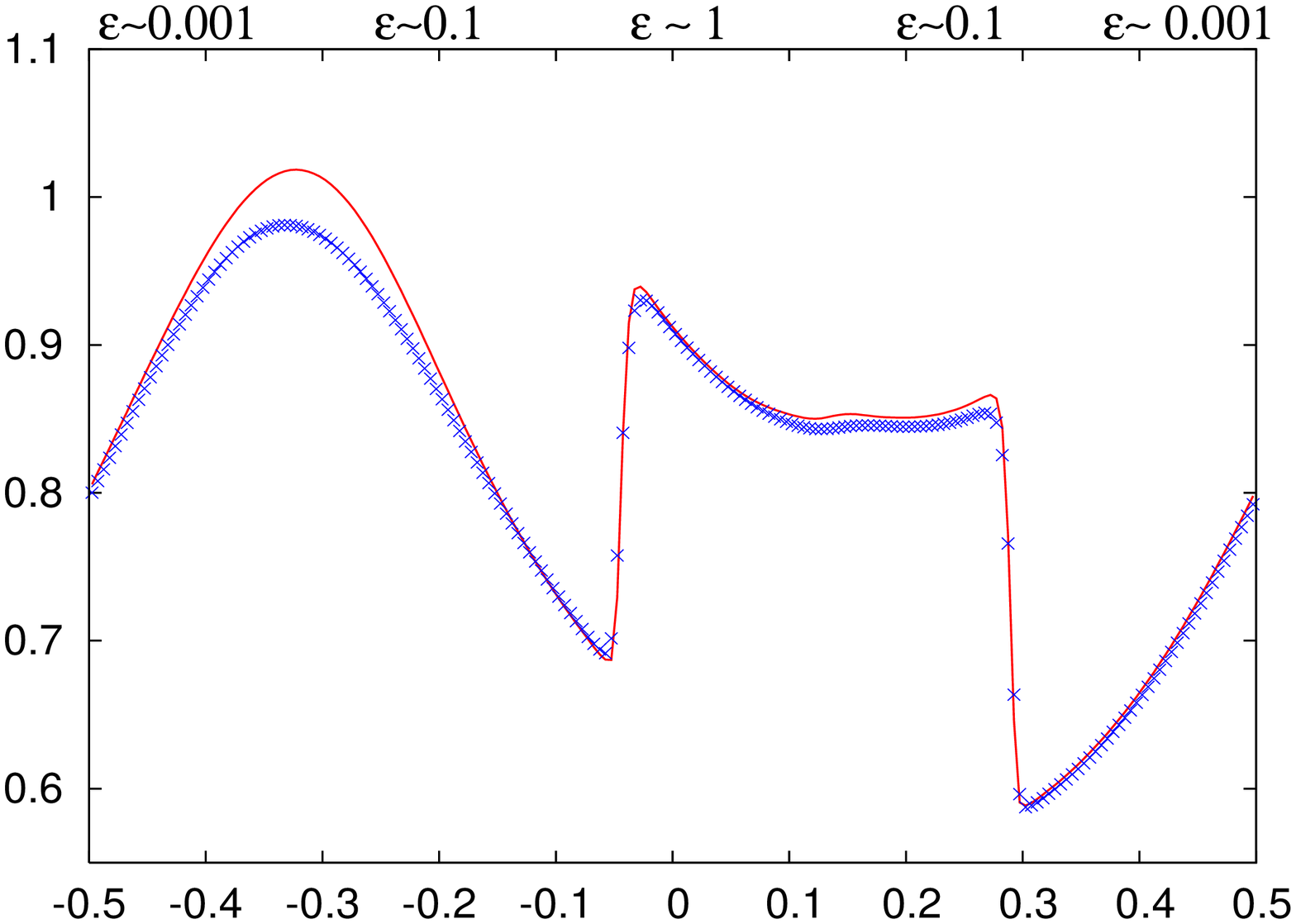}   
\\
(1)&(2)&(3)
\end{tabular}
\caption{Mixing regime problem ($\varepsilon_0=10^{-3}$), comparison of the numerical solution to the Boltzmann equation with the second order method (\ref{sch:02}) represented with dots ({\tt x}) with the numerical solution obtained with the explicit Runge-Kutta method using a small time step (line): evolution of  (1) the density $\rho$, (2) mean velocity $u$, (3) temperature $T$ at time $t=0.25$, $0.5$ and $0.75$.}
\label{fig:05-1}
\end{figure}
\end{center}

\begin{center}
\begin{figure}[htbp]
\begin{tabular}{ccc}
\includegraphics[width=5.cm]{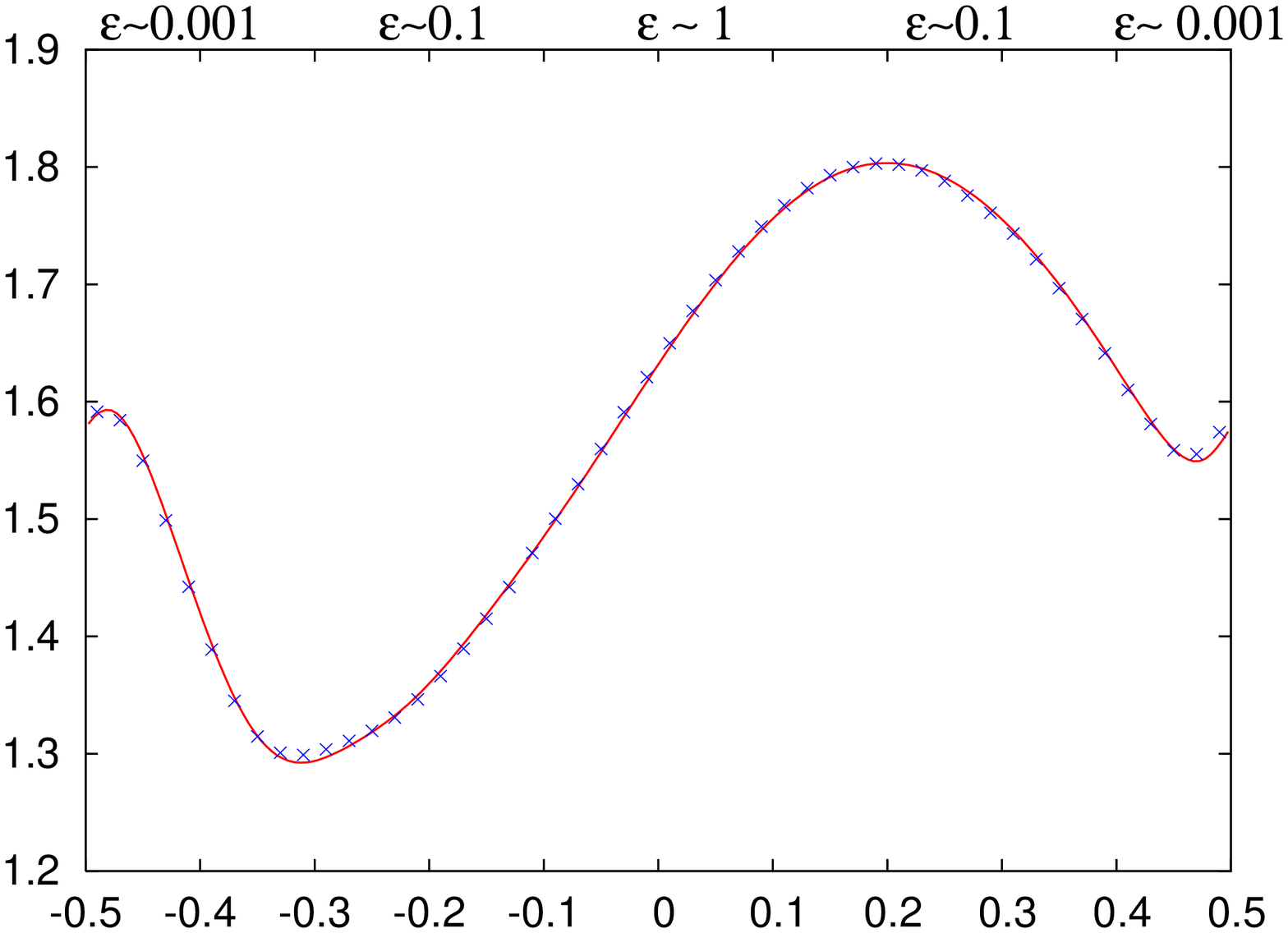}    
&
\includegraphics[width=5.cm]{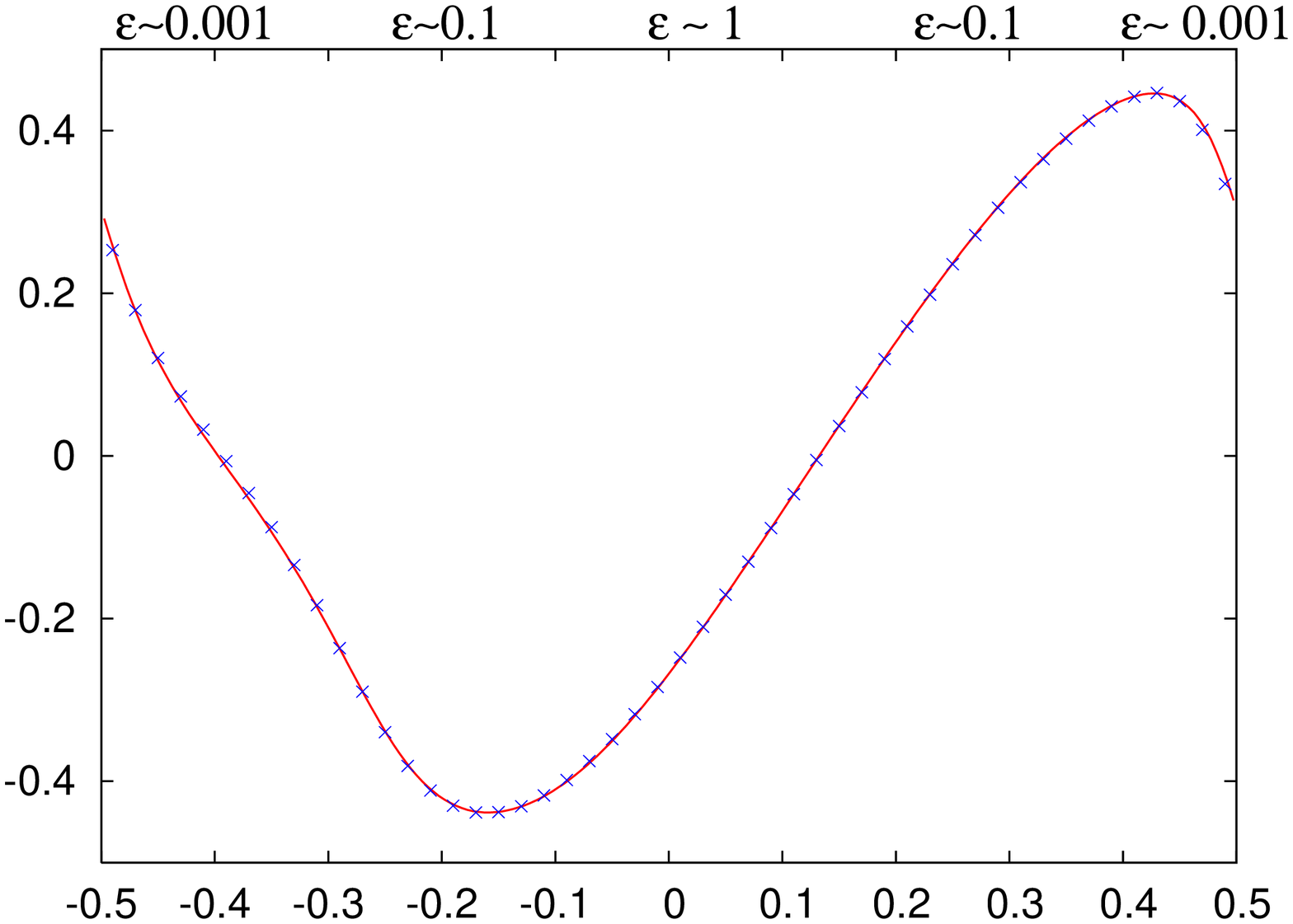}    
&
\includegraphics[width=5.cm]{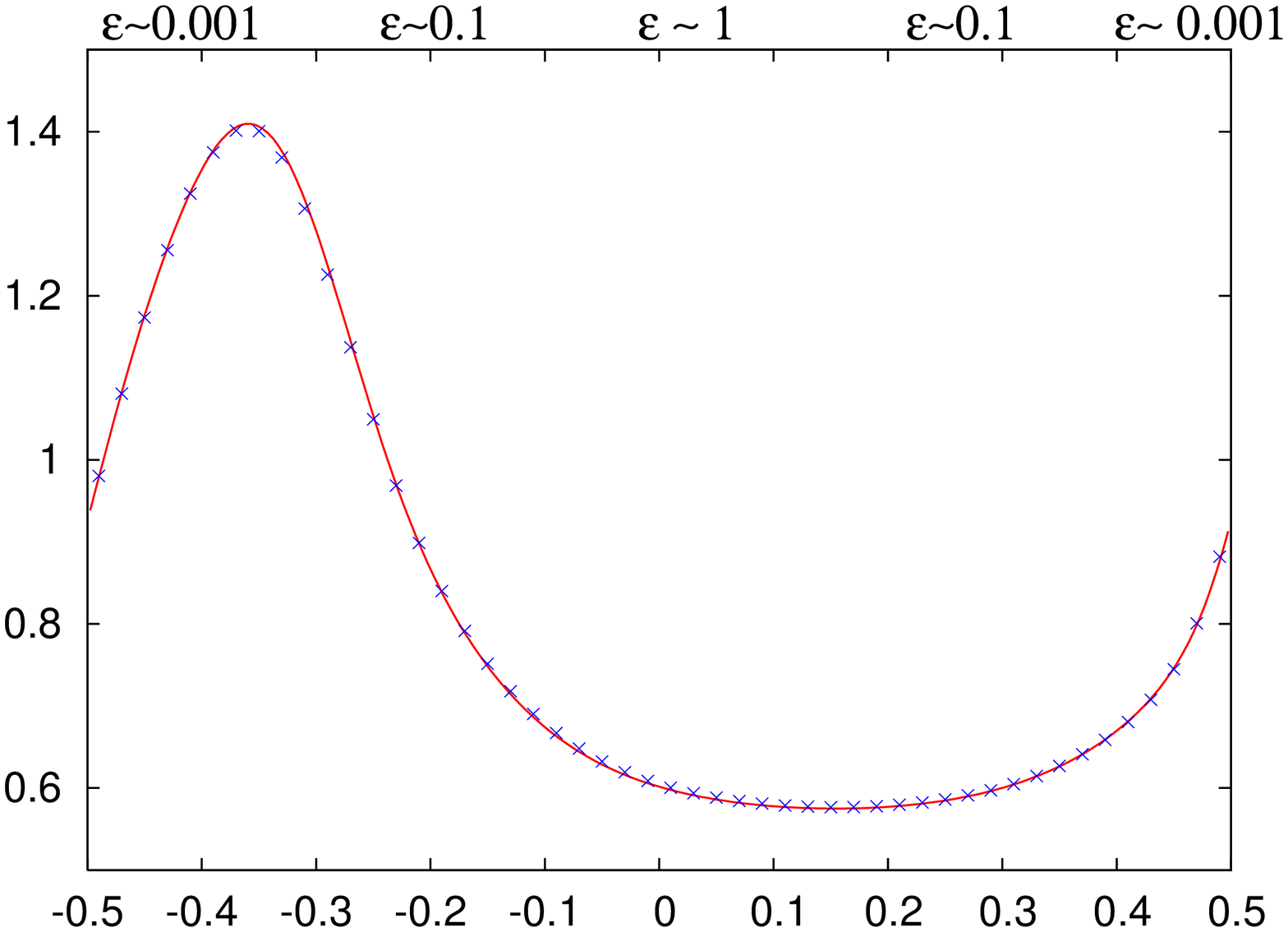}   
\\
\includegraphics[width=5.cm]{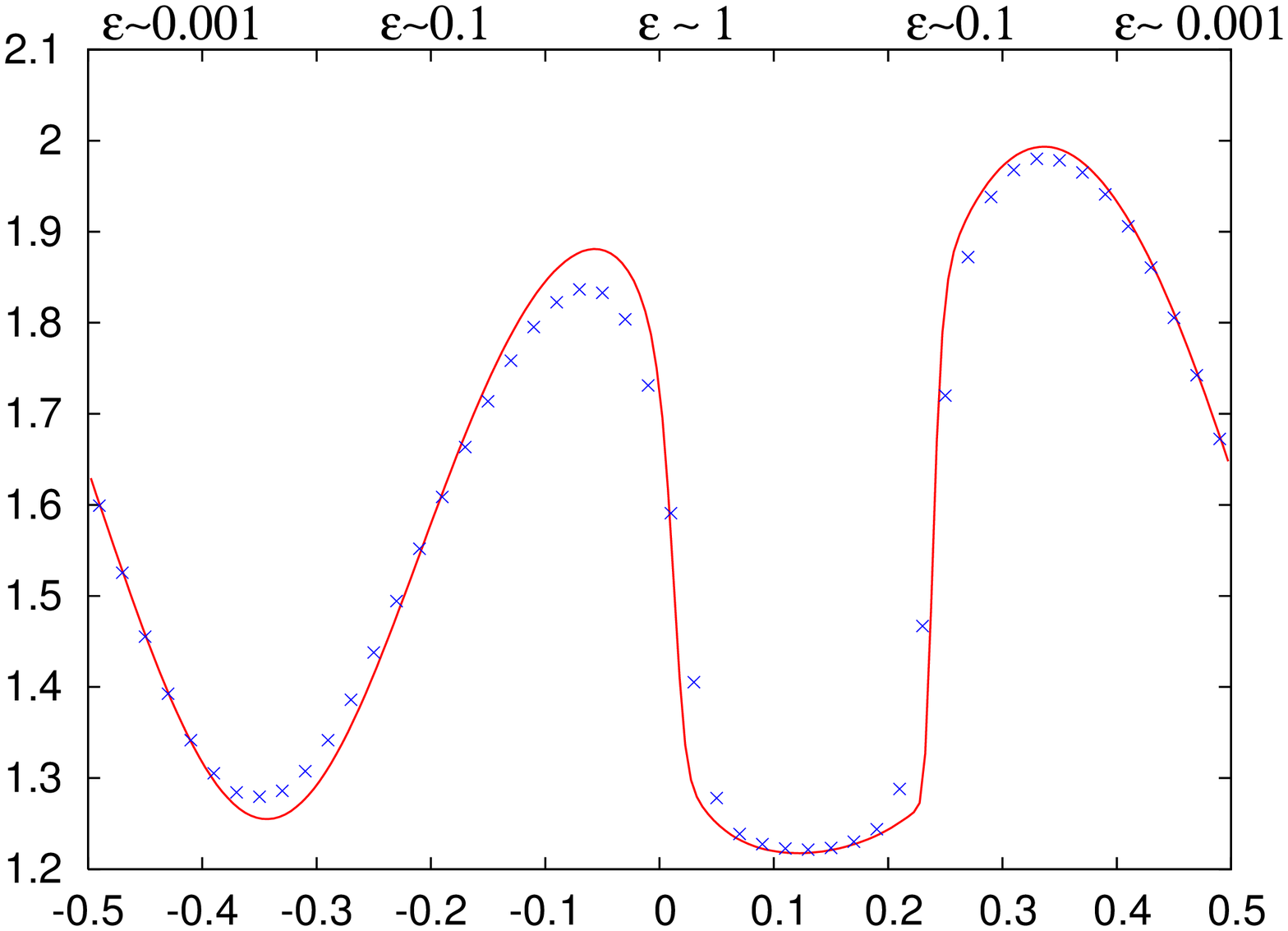}    
&
\includegraphics[width=5.cm]{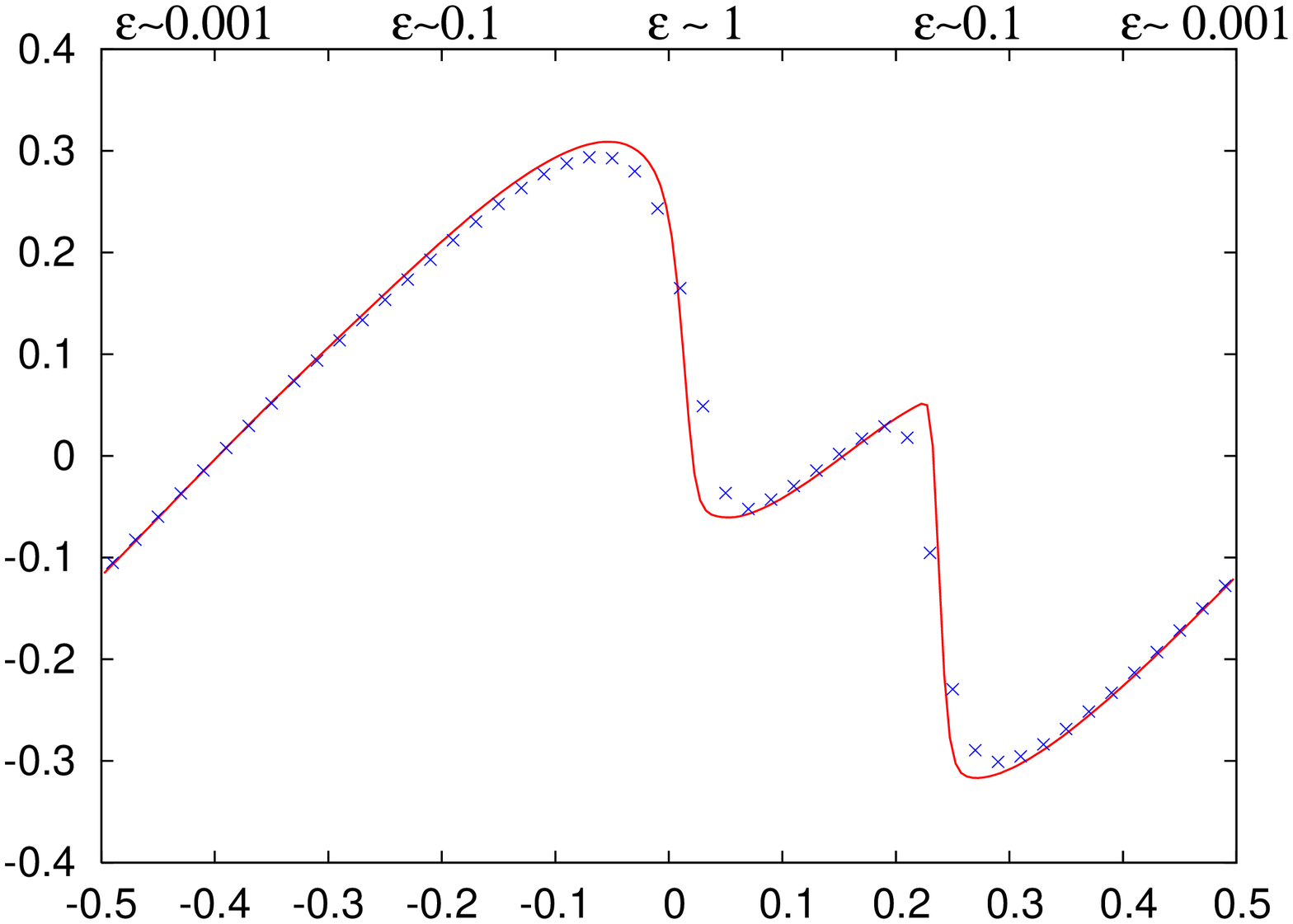}    
&
\includegraphics[width=5.cm]{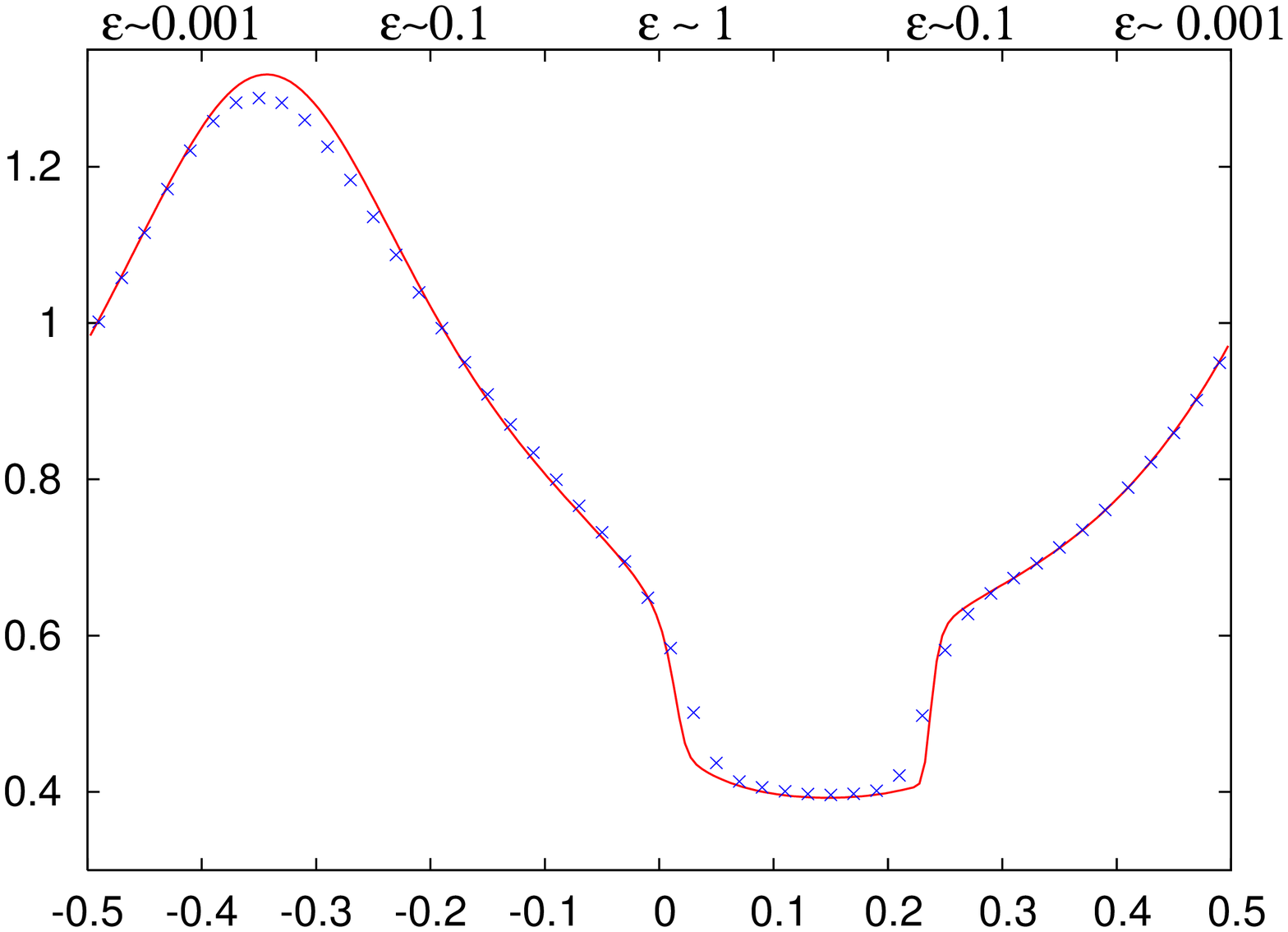}   
\\
\includegraphics[width=5.cm]{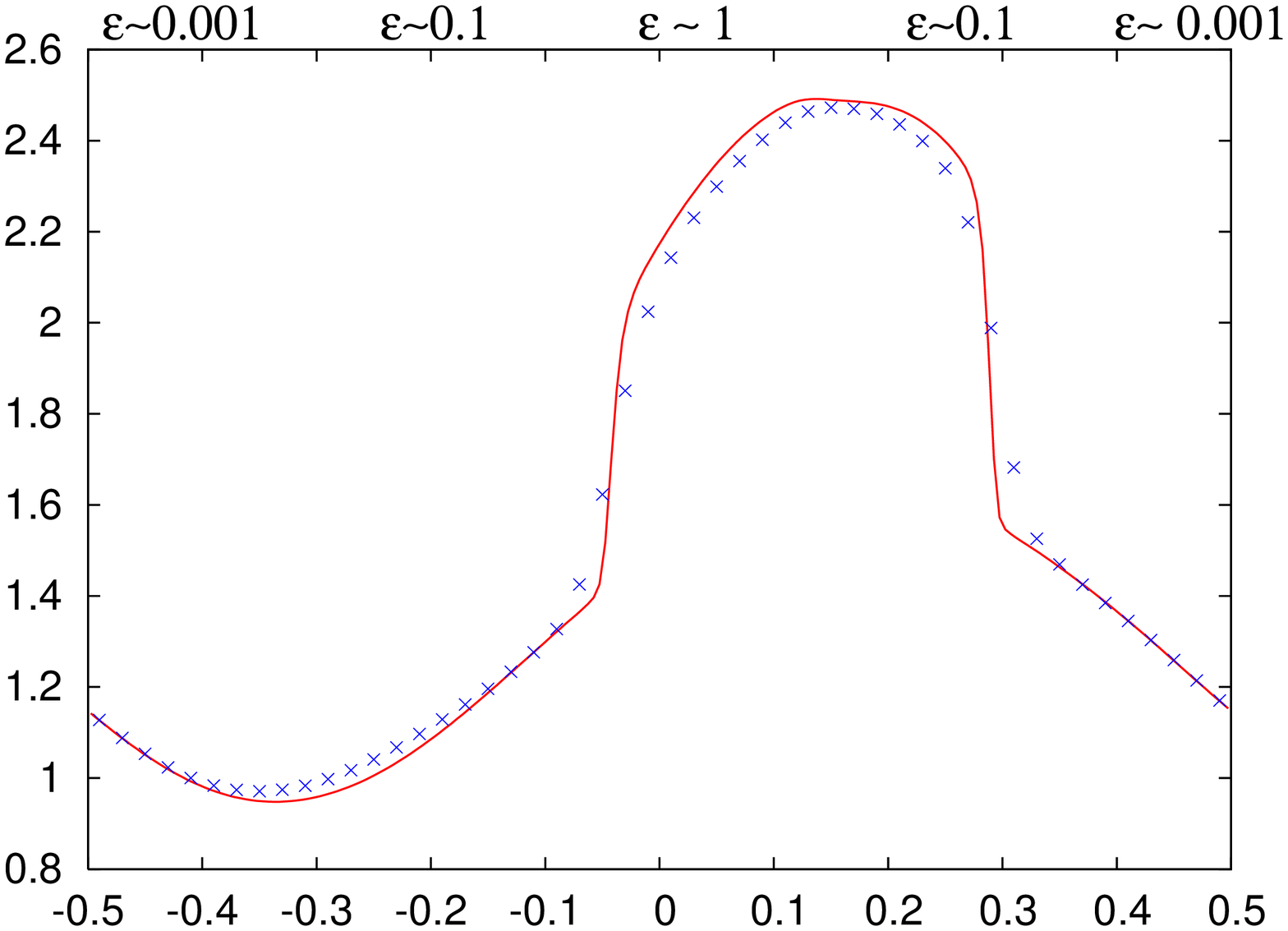}    
&
\includegraphics[width=5.cm]{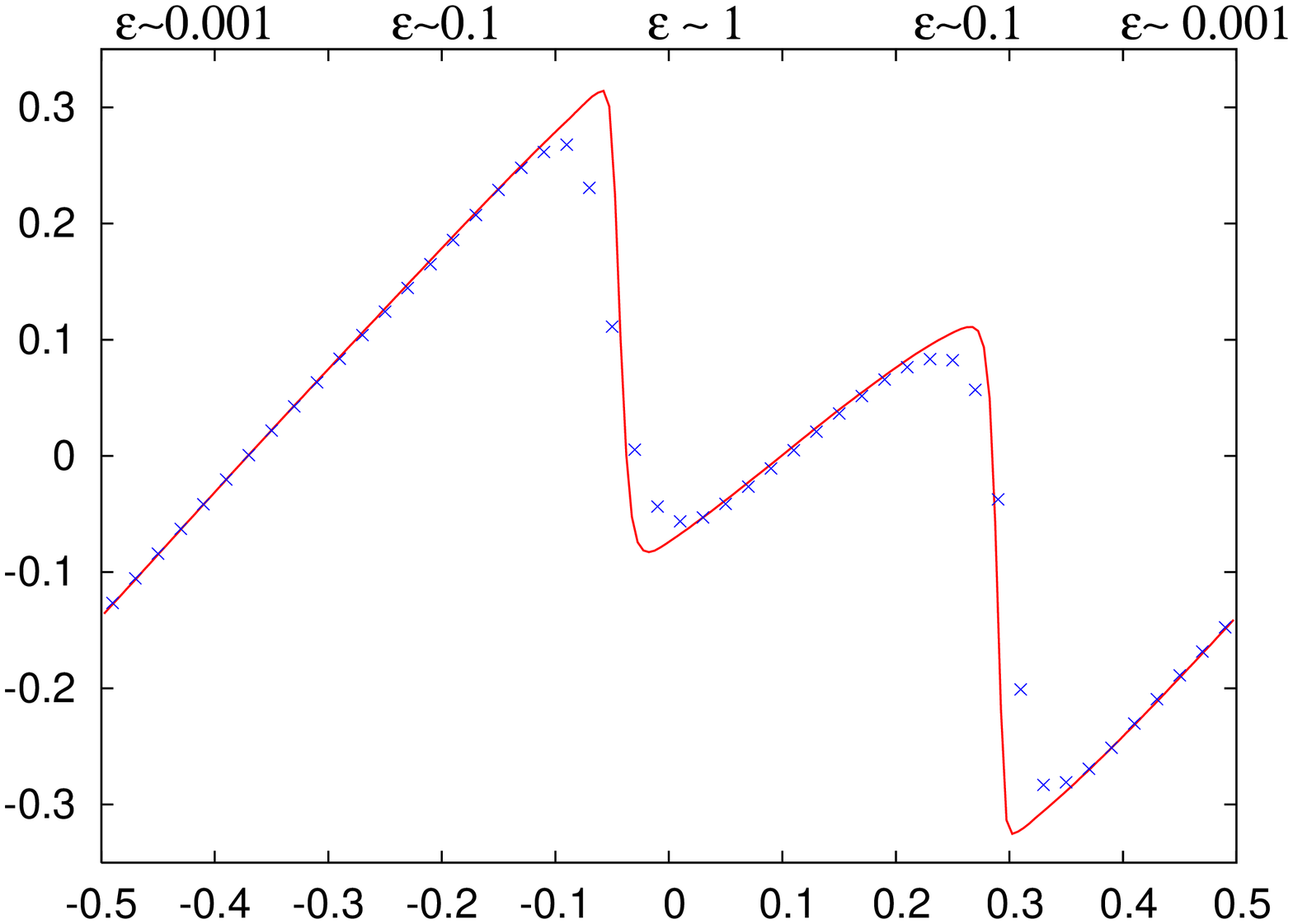}    
&
\includegraphics[width=5.cm]{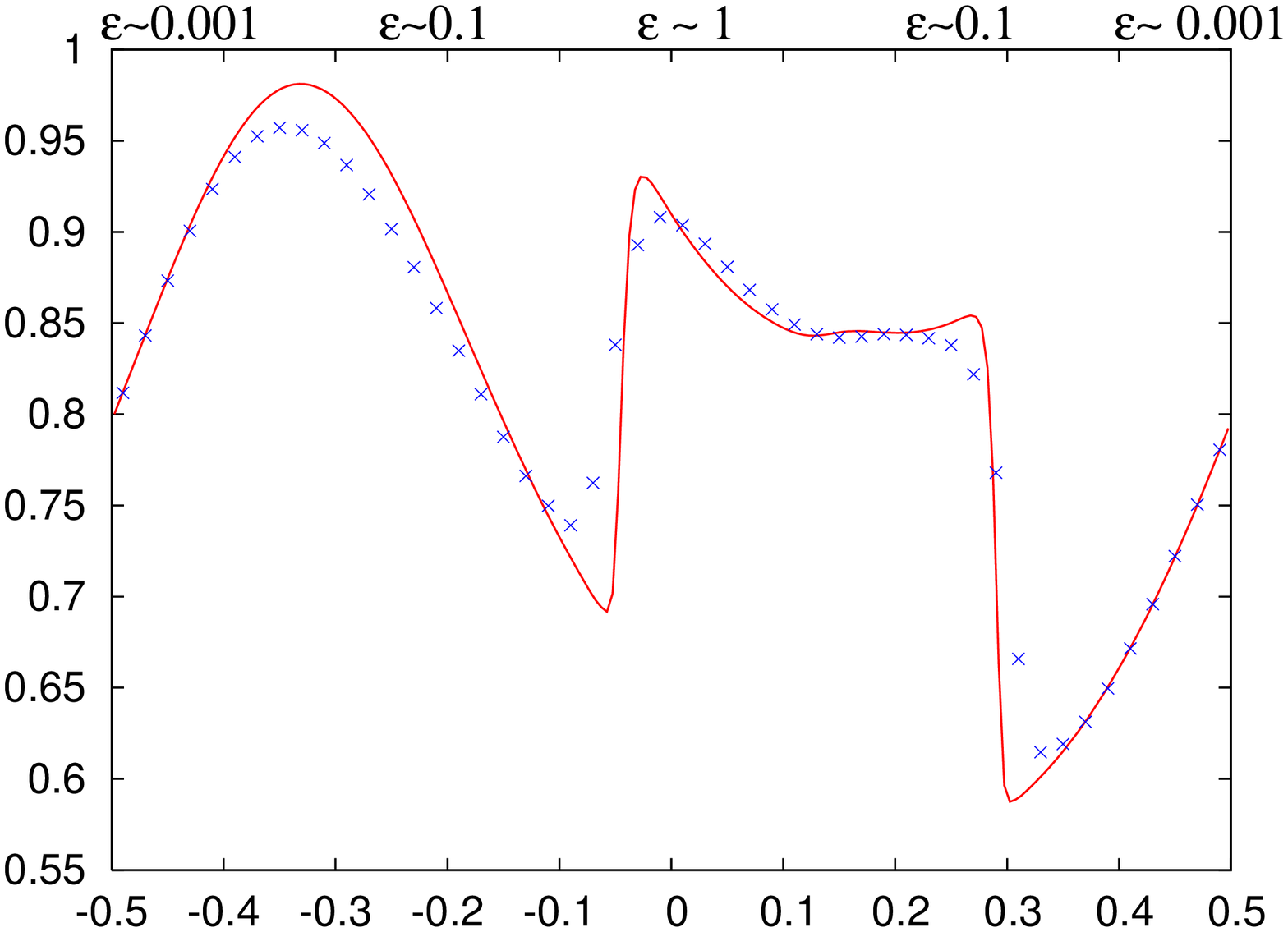}   
\\
(1)&(2)&(3)
\end{tabular}
\caption{Mixing regime problem ($\varepsilon_0=10^{-3}$), comparison of the numerical solution to the Boltzmann equation obtained with the AP scheme (\ref{sch:02}) using $n_x=50$  (dots {\tt x}) and $n_x=200$ points (line):  evolution of  (1) the density $\rho$, (2) mean velocity $u$, (3) temperature $T$ at time $t=0.25$, $0.5$ and $0.75$.}
\label{fig:05-2}
\end{figure}
\end{center}

\section{Other applications: numerical stability}
\label{sec4}
\setcounter{equation}{0}
In this section, we want to illustate the efficiency of the asymptotic preserving scheme to treat high order differential operators. We have already applied such a scheme for Willmore flow (fourth order differential operator \cite{filbet:shu,smereka}). Here, we consider the flow of gas in a two dimensional porous medium with initial density $g_0(v)\geq 0$. The distribution function $g(t,v)$ then satisfies the nonlinear degenerate parabolic equation
\begin{equation}
\label{eq:Porous}
\frac{\partial g}{\partial t} = \Delta_v g^m,
\end{equation}
where $m>1$ is a physical constant. Assuming that 
$$
\int_{\R^2} (1+|v|^2) \, g_0(v) dv < +\infty,
$$
J.A. Carrillo and G. Toscani \cite{CT} proved that $g(t,v)$ behaves asymptotically in a self-similar way like the Barenblatt-Pattle solution, as $t\rightarrow +\infty$. More precisely, it is easy to see that if we consider the change of variables
\begin{equation}
\label{dynscal}
g(t,v) \,=\,  \frac{1}{s(t)}\, f\left(\log(s(t)), \frac{v}{s(t)} \right)\,,
\end{equation}
where $s(t):=\sqrt{1+2t}$,  the new distribution function $f$  is solution to 
$$
\label{eq:Porous2}
\frac{\partial f}{\partial t} = \nabla_v\cdot\left( v\, f \,+\,  \nabla_v f^m\right),
$$
and converges to the Barenblatt-Pattle distribution 
$$
\M(v) = \left( C - \frac{m-1}{2\,m}\,|v|^2\right)_+^{1/(m-1)},
$$
where $C$ is uniquely determined and depends on the initial mass  $g_0$ but not on the ``details'' of the initial data.

Instead of working on (\ref{eq:Porous})  directly, we will study the asymptotic decay towards its equilibrium. The key argument on the proof of J.A. Carrillo and G. Toscani is the control  of the entropy functional 
$$
H(f) = \int_{\R^2} \left[ |v|^2 \, f(t,v) + \frac{m}{m-1} f^m(t,v)\right]\,dv, 
$$  
which satisfies
$$
\frac{dH(f)}{dt} = -\int_{\R^2}f(t,v)\,\left| v + \frac{m}{m-1} \nabla f^{m-1}\right|^2\,dv \leq 0
$$
or the control of the relative entropy  $H(f|\M) = H(f) - H(\M)$ with respect to the steady state $\M$.

Numerical discretization of this problem leads to the following difficulty : explicit schemes are constrained by a CFL condition $\Delta t \simeq \Delta v^2$ whereas implicit schemes require the numerical resolution of a nonlinear problem at each time step (with a local constraint on the time step). We refer to \cite{CF,F06} for a fully implicit approximation  preserving steady states for nonlinear Fokker-Planck type equations. 

Here we do not focus on the velocity discretization, but only want to apply our splitting operator technique to remove this severe constraint on the time step. Here the parameter $\varepsilon$ does not represent  a physical time scale but is only related to the velocity space discretization $\Delta v$. Therefore, we set $Q(f)  = \nabla_v\cdot \left(v\, f \,+\, \nabla_v f^m  \right) $ and $P(f)=\nabla Q(\M) \,(f-\M)$, which leads to the following decomposition
$$
\frac{\partial f }{\partial t} \,=\, \underbrace{\Delta_v \left(f^m \,-\, m\,\M^{m-1}\,f \right)}_\textrm{non stiff part} \,\,+\,\,  \underbrace{\nabla_v\cdot \left(v\, f  \,+\, m \,\nabla_v\left(\M^{m-1}\,f\right) \right)}_\textrm{stiff linear part}.
$$ 
Then we apply a simple IMEX scheme which only requires the numerical resolution of a linear system at each time step.  

We choose $m=3$ and a discontinuous initial datum far from the equilibrium
$$
f_0(v)  = \sum_{l\in \{1,2\}} \sum_{k\in \{0,\dots,n-1\}} \frac{1}{10}\,{\bf 1}_{\B(0,r_0)}(v-v_{k,l})
$$
where $n=12$, $r_0=1/4$ and $v_{k,l}\,=l\,{\rm e}^{i\,\theta_k}$, with $\theta_k=\,2\,k\,\pi/n$, $k=0,\dots,n-1$. We use a standard velocity discretization in the velocity space based on an upwind finite volume approximation for the transport term and a center difference for the diffusive part. We take $n_v^2=120^2$
 in velocity and a time step  $\Delta t= 0.02$ which is much larger than the time step satisfying  a classical CFL condition for this problem $\Delta t\simeq O(\Delta v^2)$. The numerical scheme (\ref{sch:01}) is still stable and the numerical  solution preserves  nonegativity at each time step (see Figure~\ref{fig:06-1})! For large time, the solution converges to an approximation of
 the steady state even  if the present scheme is not exactly well-balanced (it does not preserve exactly the steady state). Moreover,
 to get a better idea on the behavior of the numerical solution, we plot the evolution of the entropy and its dissipation for different time steps. More surprisingly, the numerical entropy is decreasing and the dissipation converges towards zero when times goes to infinity.

\begin{center}
\begin{figure}[htbp]
\begin{tabular}{cc}
\includegraphics[width=7.75cm,height=6.75cm]{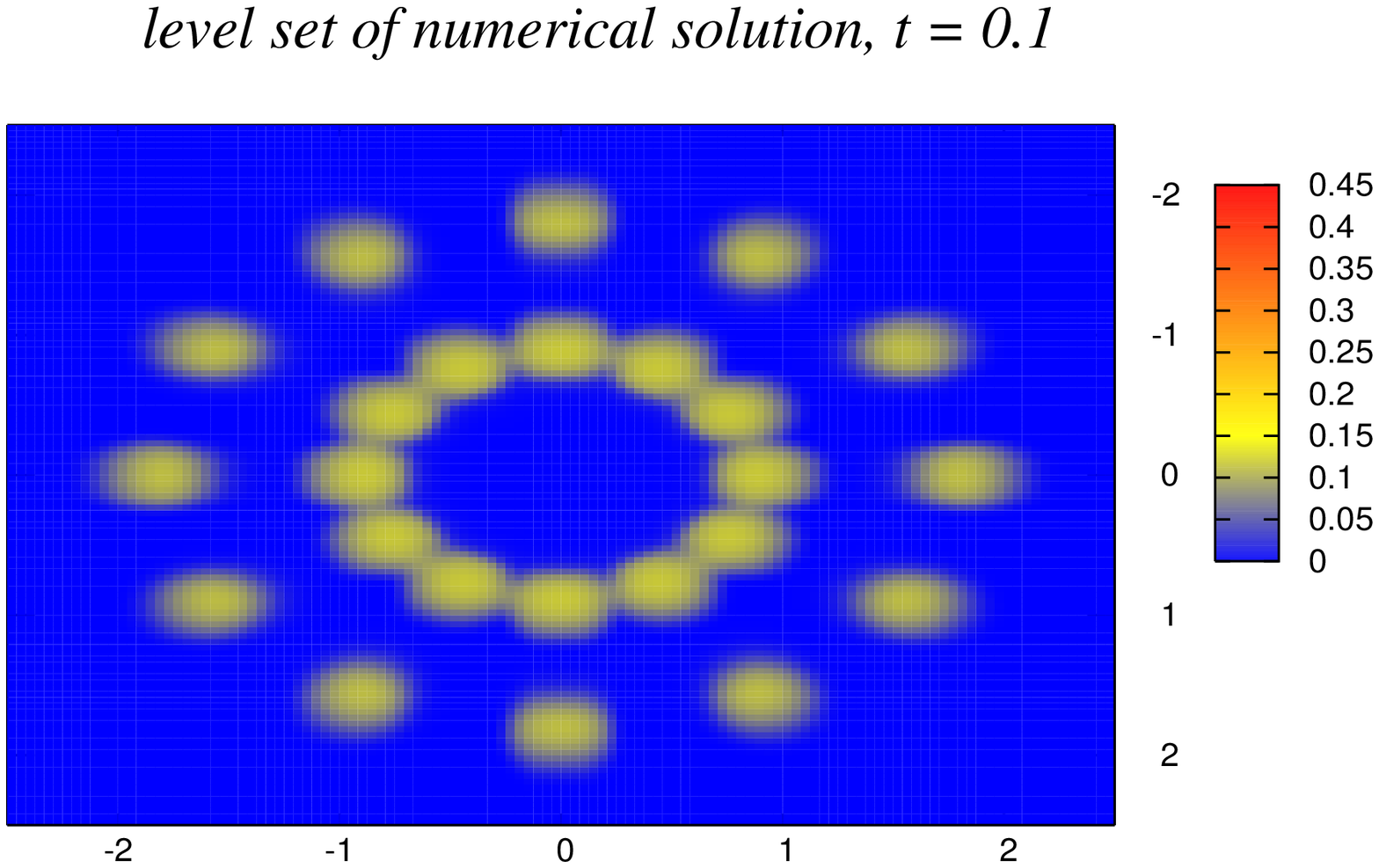}    
&
\includegraphics[width=7.75cm,height=6.75cm]{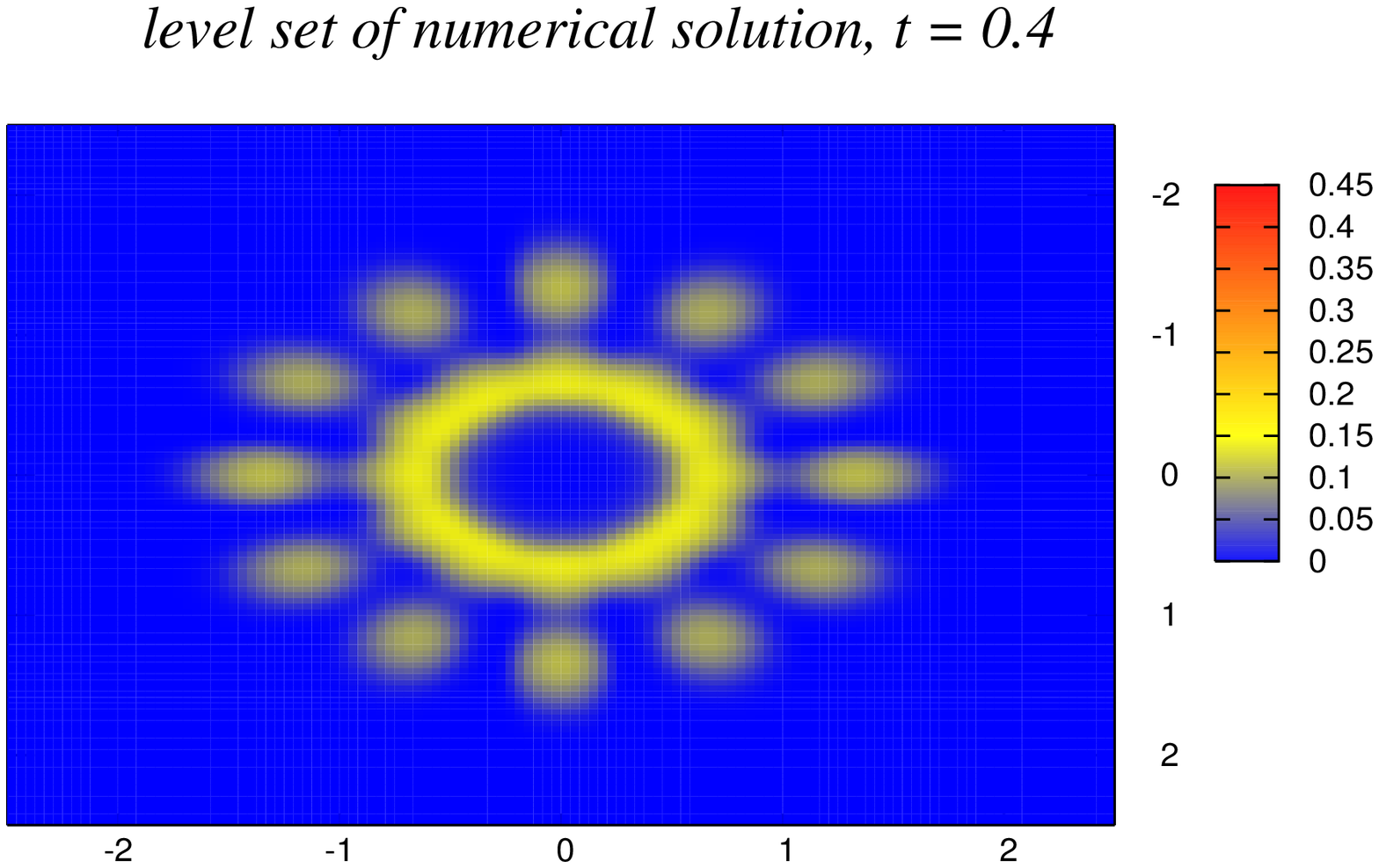}    
\\
\includegraphics[width=7.75cm,height=6.75cm]{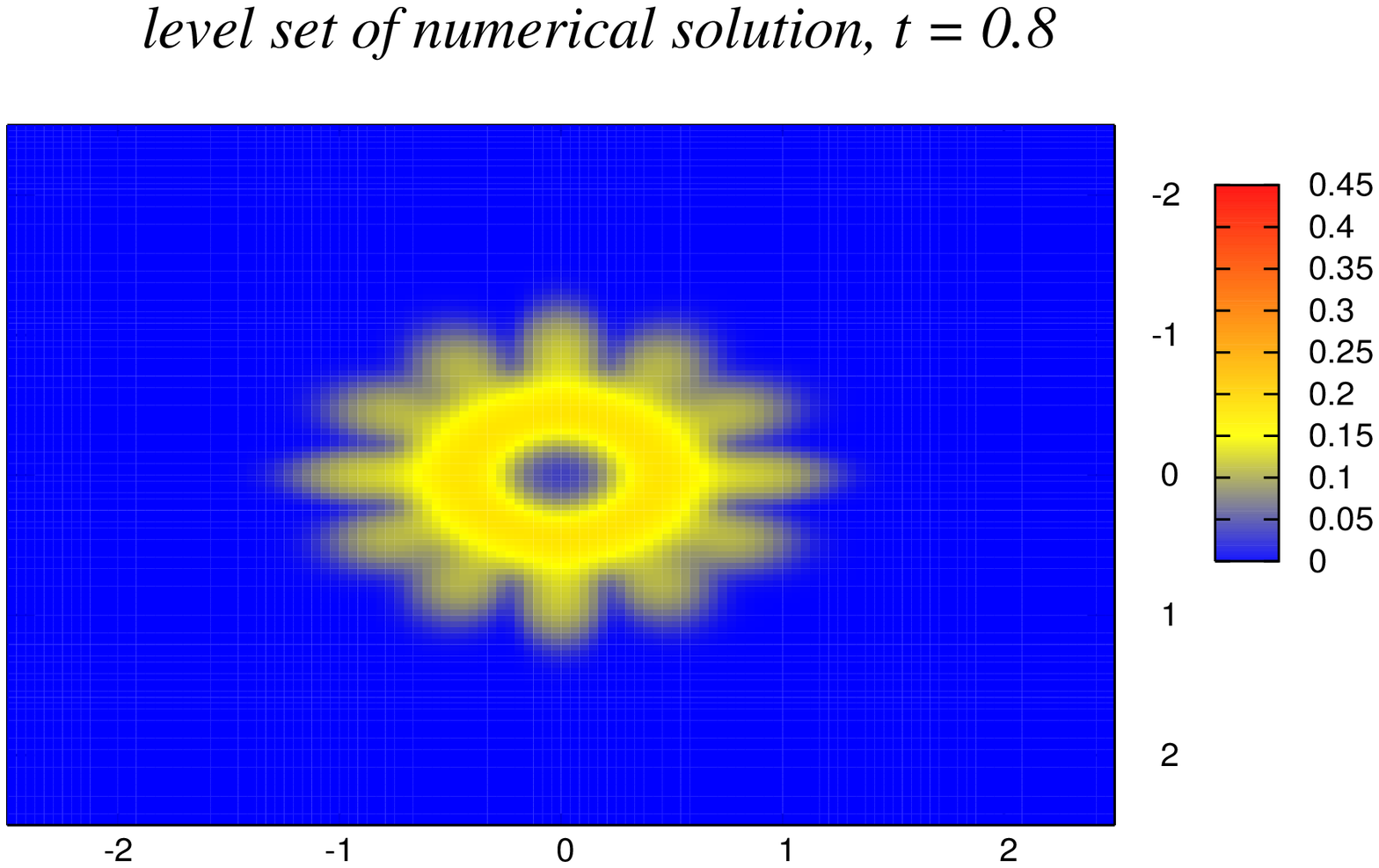}    
&
\includegraphics[width=7.75cm,height=6.75cm]{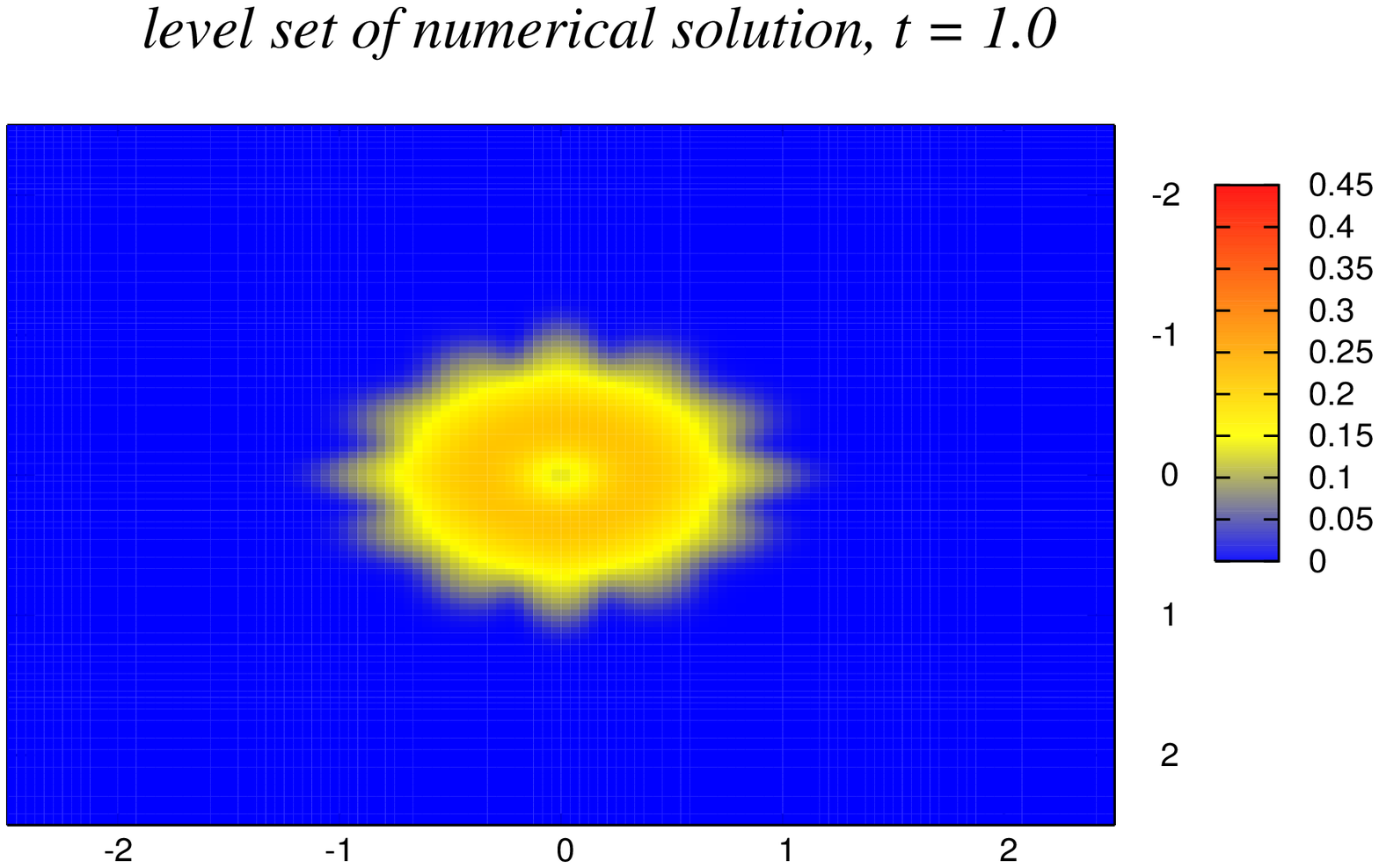}    
\\
\includegraphics[width=7.75cm,height=6.75cm]{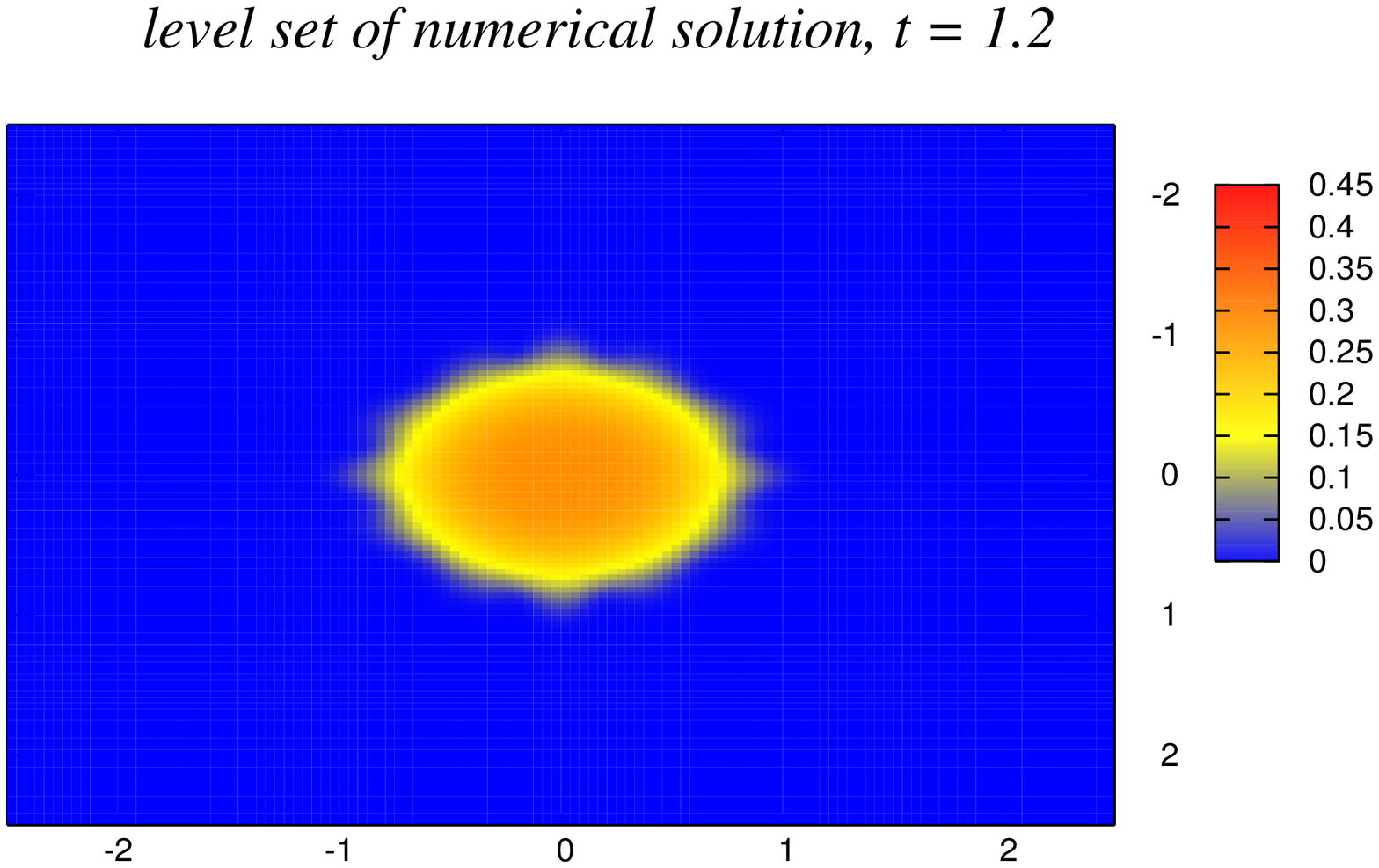}    
&
\includegraphics[width=7.75cm,height=6.75cm]{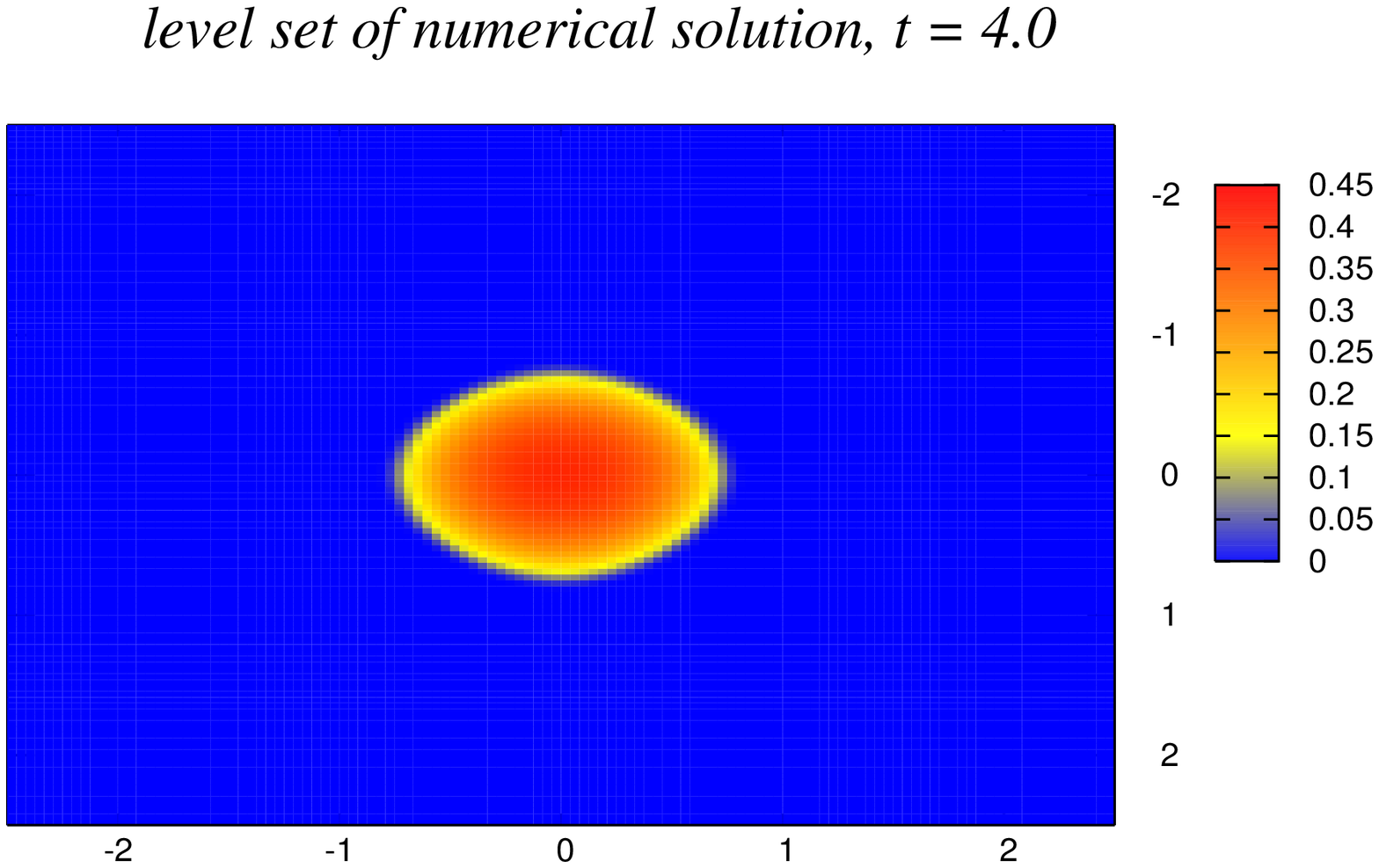}    
\end{tabular}
\caption{Nonlinear Fokker-Planck solution: convergence toward equilibrium (Barenblatt-Pattle distribution) obtained with the first order method (\ref{sch:01}) using $n_x=100$   at time $t=0.1$, $0.4$, $0.8$, $1.0$, $1.2$ and $4$ with a large time step.}
\label{fig:06-1}
\end{figure}
\end{center}

\begin{center}
\begin{figure}[htbp]
\begin{tabular}{cc}
\includegraphics[width=7.75cm]{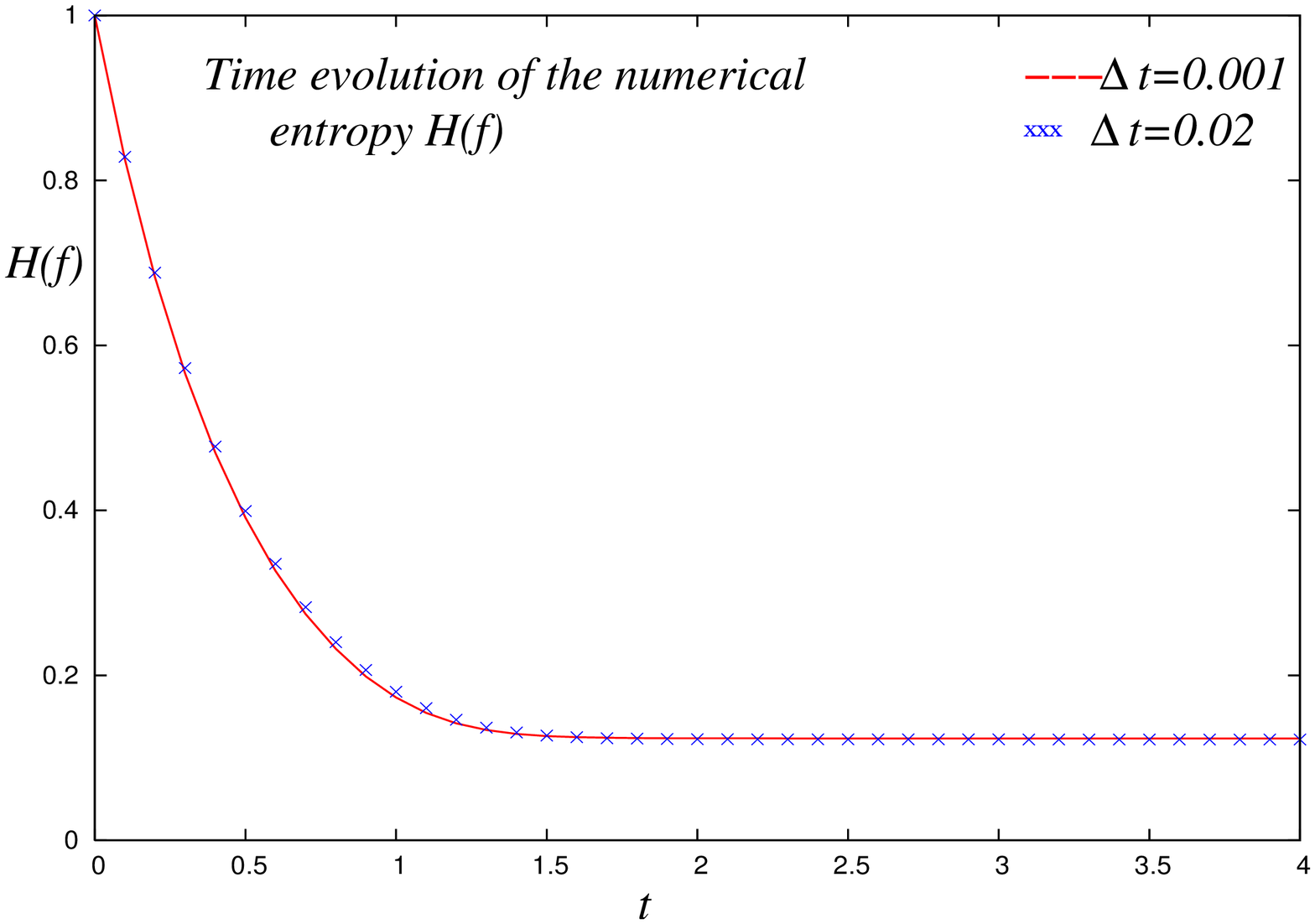}    
&
\includegraphics[width=7.75cm]{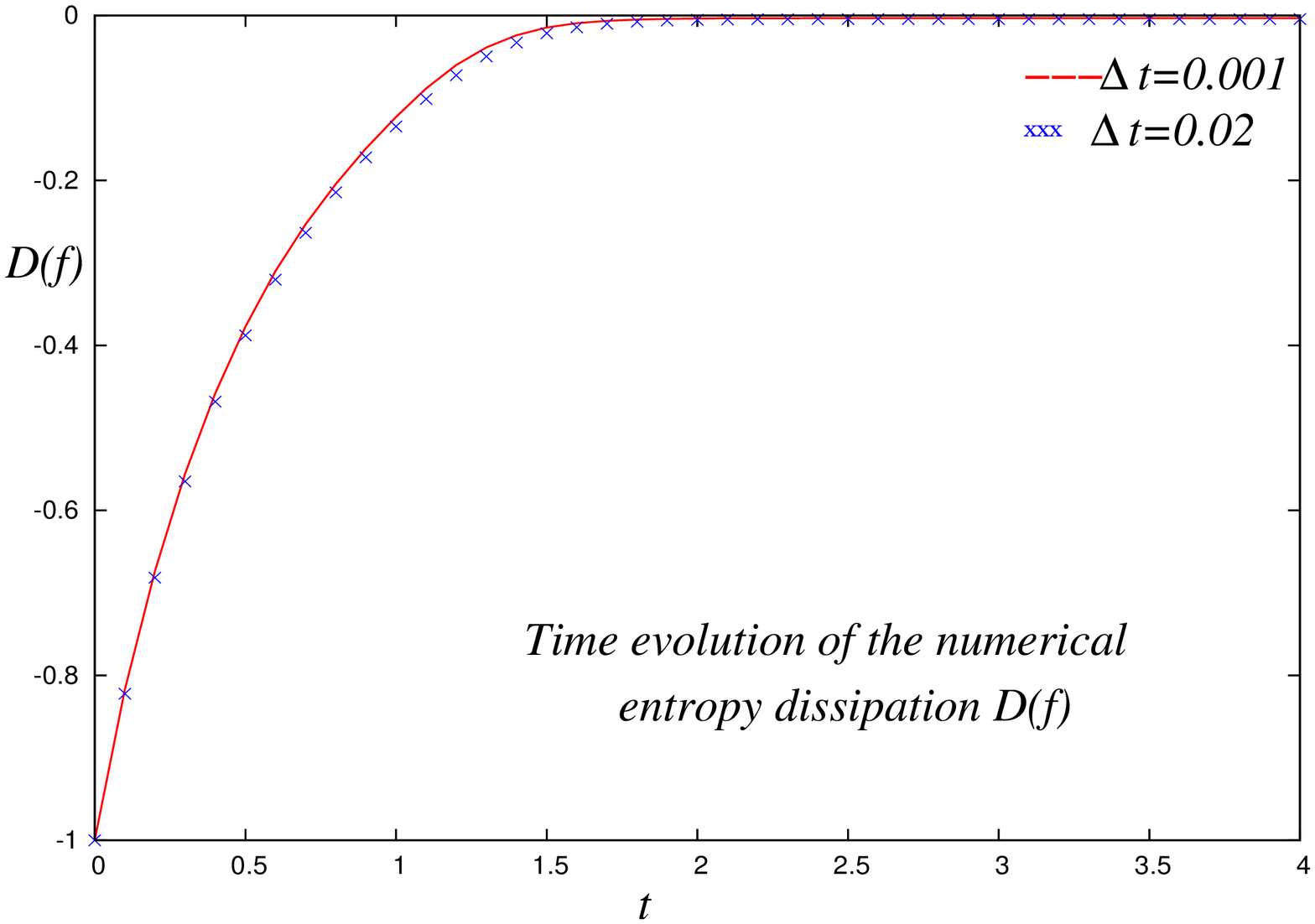}    
\end{tabular}
\caption{Nonlinear Fokker-Planck solution: convergence toward equilibrium (Barenblatt-Pattle distribution) obtained with the first order method
 (\ref{sch:01}) using $n_x=100$  with $\Delta t=0.02$ and $0.001$.}
\label{fig:06-2}
\end{figure}
\end{center}

\section{Conclusion}

We have proposed a new class of numerical schemes for physical problems with
multiple time and spatial scales described by a still nonlinear source
term. A prototype equation of this type is the Boltzmann equation for
rarified gas. When the Knudsen number is small, the stiff collision term of
the Boltzmann equation drives the density distribution to the local Maxwellian,
thus the macroscopic quantities such as mass, velocity and temperature
are be evolved according to fluid dynamic equations such as the Euler or
Navier-Stokes equations. Asmptotic-preserving (AP) schemes for kinetic 
equations have been 
successful since they capture the fluid dynamic behavior even without 
numerically resolving the small Knudsen number. However, the AP schemes need
to treat the stiff collision terms implicitly, thus it yeilds a complicated
numerical algebraic problem due to the nonlinearity and nonlocality of the
collision term. In this paper, we propose to augement the nonlinear
Boltzmann collision operator by a much simpler BGK collision operator,
and impose implicity only on the BGK operators which can be handled much
more easily. We show that this method is AP in the Euler regime, and is
also consistent to the Navier-Stokes approximations for suitably small
time steps and mesh sizes.  Numerical examples, including those with
mixing scales and non-local-Maxwellian initial data, decomstrate the AP
property as well as uniform convergence (in the Knudsen number) of this
method.

This method can be extended to a wide class of PDEs (or ODEs) with stiff
source terms that admit a stable and unique local equilibirum.  We use
the Fokker-Planck equation as an example to illustrate this point, and
will pursue more applications in the future.

\bigskip

\subsection*{Acknowledgments}  F. Filbet thanks Ph. Lauren\c cot, M. Lemou, P. Degond and L. Pareschi for usefull discussions on the topic.


\begin{flushleft} 
\signff 
\end{flushleft}
\vspace{-4.25cm}
\begin{flushright} 
\signsj 
\end{flushright}

\end{document}